\documentclass[a4paper,10pt]{article}
\usepackage[utf8]{inputenc}
\usepackage[T1]{fontenc}
\usepackage{geometry}
\usepackage{tabularx}
\usepackage{booktabs}
\usepackage{array}
\usepackage{microtype}
\usepackage{setspace}
\usepackage{xcolor}
\definecolor{lightgreen}{rgb}{.20,.60,.22}
\usepackage{pdfpages}
\usepackage[normalem]{ulem}

\usepackage{amssymb,amsmath,amsopn,amsxtra,amsthm,amsfonts}
\usepackage{mathtools}
\usepackage[mathcal]{euscript}
\let\mathscr\mathcal
\usepackage[bb=px]{mathalfa}

\usepackage{tikz}
\usetikzlibrary{matrix}
\usepackage{tikz-cd}
\usepackage{enumitem}
\usepackage{xspace}
\usepackage{xifthen}
\usepackage{xparse}

\usepackage[
backend=biber,
style=alphabetic,
maxnames=99,        
maxalphanames=99,   
minalphanames=99    
]{biblatex}
\setlist[enumerate,1]{label={(\arabic*)},itemsep=\parskip} 
\setlist[itemize,1]{itemsep=\parskip} 

\newlist{thmlist}{enumerate}{2}
\setlist[thmlist,1]{label={\em(\roman*)},ref={(\roman*)},%
	itemsep=\parskip,leftmargin=*,align=left}
\setlist[thmlist,2]{label={\em(\alph*)},ref={(\alph*)},%
	itemsep=\parskip,leftmargin=*,align=left,topsep=0.1cm}

\newlist{defnlist}{enumerate}{2}
\setlist[defnlist,1]{label={(\roman*)},ref={(\roman*)},itemsep=\parskip,%
	leftmargin=*,align=left}
\setlist[defnlist,2]{label={(\alph*)},ref={(\alph*)},itemsep=\parskip,%
	leftmargin=*,align=left,topsep=0.1cm}

\DeclareFontFamily{U}{min}{}
\DeclareFontShape{U}{min}{m}{n}{<-> udmj30}{}

\newcommand{\matone}{\ensuremath{\bigl(\begin{smallmatrix}1 & * \\ 0 & 1\end{smallmatrix}}\bigr)}
\newcommand{\matoneF}{\ensuremath{\bigl(\begin{smallmatrix}1 & * \\ 0 & 1\end{smallmatrix}}\bigr)}

\newcommand{\nc}{\newcommand}
\nc{\renc}{\renewcommand}
\nc{\ssec}{\subsection}
\nc{\sssec}{\subsubsection}
\nc{\on}{\operatorname}
\nc{\term}[1]{#1\xspace}

\nc{\sA}{\ensuremath{\mathcal{A}}\xspace}
\nc{\sB}{\ensuremath{\mathcal{B}}\xspace}
\nc{\sC}{\ensuremath{\mathcal{C}}\xspace}
\nc{\sD}{\ensuremath{\mathcal{D}}\xspace}
\nc{\sE}{\ensuremath{\mathcal{E}}\xspace}
\nc{\sF}{\ensuremath{\mathcal{F}}\xspace}
\nc{\sG}{\ensuremath{\mathcal{G}}\xspace}
\nc{\sH}{\ensuremath{\mathcal{H}}\xspace}
\nc{\sI}{\ensuremath{\mathcal{I}}\xspace}
\nc{\sJ}{\ensuremath{\mathcal{J}}\xspace}
\nc{\sK}{\ensuremath{\mathcal{K}}\xspace}
\nc{\sL}{\ensuremath{\mathcal{L}}\xspace}
\nc{\sM}{\ensuremath{\mathcal{M}}\xspace}
\nc{\sN}{\ensuremath{\mathcal{N}}\xspace}
\nc{\sO}{\ensuremath{\mathcal{O}}\xspace}
\nc{\sP}{\ensuremath{\mathcal{P}}\xspace}
\nc{\sQ}{\ensuremath{\mathcal{Q}}\xspace}
\nc{\sR}{\ensuremath{\mathcal{R}}\xspace}
\nc{\sS}{\ensuremath{\mathcal{S}}\xspace}
\nc{\sT}{\ensuremath{\mathcal{T}}\xspace}
\nc{\sU}{\ensuremath{\mathcal{U}}\xspace}
\nc{\sV}{\ensuremath{\mathcal{V}}\xspace}
\nc{\sW}{\ensuremath{\mathcal{W}}\xspace}
\nc{\sX}{\ensuremath{\mathcal{X}}\xspace}
\nc{\sY}{\ensuremath{\mathcal{Y}}\xspace}
\nc{\sZ}{\ensuremath{\mathcal{Z}}\xspace}

\nc{\bA}{\ensuremath{\mathbf{A}}\xspace}
\nc{\bB}{\ensuremath{\mathbf{B}}\xspace}
\nc{\bC}{\ensuremath{\mathbf{C}}\xspace}
\nc{\bD}{\ensuremath{\mathbf{D}}\xspace}
\nc{\bE}{\ensuremath{\mathbf{E}}\xspace}
\nc{\bF}{\ensuremath{\mathbf{F}}\xspace}
\nc{\bG}{\ensuremath{\mathbf{G}}\xspace}
\nc{\bH}{\ensuremath{\mathbf{H}}\xspace}
\nc{\bI}{\ensuremath{\mathbf{I}}\xspace}
\nc{\bJ}{\ensuremath{\mathbf{J}}\xspace}
\nc{\bK}{\ensuremath{\mathbf{K}}\xspace}
\nc{\bL}{\ensuremath{\mathbf{L}}\xspace}
\nc{\bM}{\ensuremath{\mathbf{M}}\xspace}
\nc{\bN}{\ensuremath{\mathbf{N}}\xspace}
\nc{\bO}{\ensuremath{\mathbf{O}}\xspace}
\nc{\bP}{\ensuremath{\mathbf{P}}\xspace}
\nc{\bQ}{\ensuremath{\mathbf{Q}}\xspace}
\nc{\bR}{\ensuremath{\mathbf{R}}\xspace}
\nc{\bS}{\ensuremath{\mathbf{S}}\xspace}
\nc{\bT}{\ensuremath{\mathbf{T}}\xspace}
\nc{\bU}{\ensuremath{\mathbf{U}}\xspace}
\nc{\bV}{\ensuremath{\mathbf{V}}\xspace}
\nc{\bW}{\ensuremath{\mathbf{W}}\xspace}
\nc{\bX}{\ensuremath{\mathbf{X}}\xspace}
\nc{\bY}{\ensuremath{\mathbf{Y}}\xspace}
\nc{\bZ}{\ensuremath{\mathbf{Z}}\xspace}

\nc{\dA}{\ensuremath{\mathds{A}}\xspace}
\nc{\dB}{\ensuremath{\mathds{B}}\xspace}
\nc{\dC}{\ensuremath{\mathds{C}}\xspace}
\nc{\dD}{\ensuremath{\mathds{D}}\xspace}
\nc{\dE}{\ensuremath{\mathds{E}}\xspace}
\nc{\dF}{\ensuremath{\mathds{F}}\xspace}
\nc{\dG}{\ensuremath{\mathds{G}}\xspace}
\nc{\dH}{\ensuremath{\mathds{H}}\xspace}
\nc{\dI}{\ensuremath{\mathds{I}}\xspace}
\nc{\dJ}{\ensuremath{\mathds{J}}\xspace}
\nc{\dK}{\ensuremath{\mathds{K}}\xspace}
\nc{\dL}{\ensuremath{\mathds{L}}\xspace}
\nc{\dM}{\ensuremath{\mathds{M}}\xspace}
\nc{\dN}{\ensuremath{\mathds{N}}\xspace}
\nc{\dO}{\ensuremath{\mathds{O}}\xspace}
\nc{\dP}{\ensuremath{\mathds{P}}\xspace}
\nc{\dQ}{\ensuremath{\mathds{Q}}\xspace}
\nc{\dR}{\ensuremath{\mathds{R}}\xspace}
\nc{\dS}{\ensuremath{\mathds{S}}\xspace}
\nc{\dT}{\ensuremath{\mathds{T}}\xspace}
\nc{\dU}{\ensuremath{\mathds{U}}\xspace}
\nc{\dV}{\ensuremath{\mathds{V}}\xspace}
\nc{\dW}{\ensuremath{\mathds{W}}\xspace}
\nc{\dX}{\ensuremath{\mathds{X}}\xspace}
\nc{\dY}{\ensuremath{\mathds{Y}}\xspace}
\nc{\dZ}{\ensuremath{\mathds{Z}}\xspace}

\nc{\bbA}{\ensuremath{\mathbb{A}}\xspace}
\nc{\bbB}{\ensuremath{\mathbb{B}}\xspace}
\nc{\bbC}{\ensuremath{\mathbb{C}}\xspace}
\nc{\bbD}{\ensuremath{\mathbb{D}}\xspace}
\nc{\bbE}{\ensuremath{\mathbb{E}}\xspace}
\nc{\bbF}{\ensuremath{\mathbb{F}}\xspace}
\nc{\bbG}{\ensuremath{\mathbb{G}}\xspace}
\nc{\bbH}{\ensuremath{\mathbb{H}}\xspace}
\nc{\bbI}{\ensuremath{\mathbb{I}}\xspace}
\nc{\bbJ}{\ensuremath{\mathbb{J}}\xspace}
\nc{\bbK}{\ensuremath{\mathbb{K}}\xspace}
\nc{\bbL}{\ensuremath{\mathbb{L}}\xspace}
\nc{\bbM}{\ensuremath{\mathbb{M}}\xspace}
\nc{\bbN}{\ensuremath{\mathbb{N}}\xspace}
\nc{\bbO}{\ensuremath{\mathbb{O}}\xspace}
\nc{\bbP}{\ensuremath{\mathbb{P}}\xspace}
\nc{\bbQ}{\ensuremath{\mathbb{Q}}\xspace}
\nc{\bbR}{\ensuremath{\mathbb{R}}\xspace}
\nc{\bbS}{\ensuremath{\mathbb{S}}\xspace}
\nc{\bbT}{\ensuremath{\mathbb{T}}\xspace}
\nc{\bbU}{\ensuremath{\mathbb{U}}\xspace}
\nc{\bbV}{\ensuremath{\mathbb{V}}\xspace}
\nc{\bbW}{\ensuremath{\mathbb{W}}\xspace}
\nc{\bbX}{\ensuremath{\mathbb{X}}\xspace}
\nc{\bbY}{\ensuremath{\mathbb{Y}}\xspace}
\nc{\bbZ}{\ensuremath{\mathbb{Z}}\xspace}

\nc{\mrm}[1]{\ensuremath{\mathrm{#1}}\xspace}
\nc{\mbf}[1]{\ensuremath{\mathbf{#1}}\xspace}
\nc{\mcal}[1]{\ensuremath{\mathcal{#1}}\xspace}
\nc{\msc}[1]{\ensuremath{\mathscr{#1}}\xspace}
\nc{\mfr}[1]{\ensuremath{\mathfrak{#1}}\xspace}

\renc{\bar}[1]{\overline{#1}}

\let\S\relax

\nc{\sub}{\subset}
\nc{\too}{\longrightarrow}
\nc{\hook}{\hookrightarrow}
\nc*{\hooklongrightarrow}{\ensuremath{\lhook\joinrel\relbar\joinrel\rightarrow}}
\nc{\hooklong}{\hooklongrightarrow}
\nc{\twoheadlongrightarrow}{\relbar\joinrel\twoheadrightarrow}
\nc{\shiso}{\approx}
\nc{\isoto}{\xrightarrow{\sim}}
\nc{\isofrom}{\xleftarrow{\sim}}
\renc{\ge}{\geqslant}
\renc{\le}{\leqslant}

\nc{\id}{\mathrm{id}}

\DeclareMathOperator{\Hom}{\on{Hom}}
\nc{\uHom}{\underline{\smash{\Hom}}}

\DeclareMathOperator{\End}{\on{End}}

\nc{\uEnd}{\underline{\smash{\End}}}

\renc{\lim}{\varprojlim}

\makeatletter
\newcommand{\colim@}[2]{%
	\vtop{\m@th\ialign{##\cr
			\hfil$#1\operator@font colim$\hfil\cr
			\noalign{\nointerlineskip\kern1.5\ex@}#2\cr
			\noalign{\nointerlineskip\kern-\ex@}\cr}}%
}
\newcommand{\colim}{%
	\mathop{\mathpalette\colim@{\rightarrowfill@\textstyle}}\nmlimits@
}
\makeatother

\nc{\Cofib}{\on{Cofib}}
\nc{\Fib}{\on{Fib}}
\nc{\initial}{\varnothing}
\newcommand{\opp}{\mathrm{op}}

\usepackage{url}
\usepackage{hyperref}
\hypersetup{
	colorlinks=true,
	linkcolor=purple,
	citecolor=lightgreen,
	urlcolor=cyan,
}

\usepackage[capitalise,nameinlink]{cleveref}

\newtheorem{ThmAlpha}{Theorem}
\newtheorem{thm}{Theorem}[section]
\newtheorem{cor}[thm]{Corollary}
\newtheorem{lem}[thm]{Lemma}
\newtheorem{prop}[thm]{Proposition}

\newtheorem{qu}[thm]{Question}

\theoremstyle{definition}

\newtheorem{de}[thm]{Definition}

\newtheorem{rem}[thm]{Remark}

\newtheorem{ex}[thm]{Example}

\newtheorem{nota}[thm]{Notation}

\newtheorem{warn}[thm]{Warning}
\newtheorem{cov}[thm]{Convention}

\renewcommand{\eqref}[1]{(\ref{#1})}

\numberwithin{equation}{subsection}

\nc{\Spc}{\mrm{Spc}}
\nc{\Spt}{\mrm{Spt}}
\nc{\Spec}{\on{Spec}}
\nc{\Stk}{\mrm{Stk}}
\nc{\Sch}{\mrm{Sch}}
\nc{\aff}{\mrm{aff}}
\nc{\A}{\mbf{A}}
\renc{\P}{\mbf{P}}
\nc{\cl}{{\mrm{cl}}}
\nc{\bDelta}{\mathbf{\Delta}}
\nc{\un}{\mathbf{1}}
\nc{\Tot}{\on{Tot}}
\nc{\Cech}{\textnormal{\v{C}}}
\nc{\Mod}{\mrm{Mod}}
\nc{\Qcoh}{\on{Qcoh}}
\nc{\free}{\mrm{free}}

\nc{\perf}{\mrm{perf}}
\nc{\aperf}{\mrm{aperf}}
\nc{\coh}{\mrm{coh}}
\newcommand{\Cat}{\mrm{Cat}}
\nc{\unitm}{\mbf{1}}
\nc{\sphere}{\mbf{S}}
\nc{\Z}{\mbf{Z}}
\nc{\Map}{\mrm{Map}}
\nc{\map}{\mrm{map}}
\nc{\PrL}{\mathcal{P}r^\mrm{L}}
\nc{\PrLst}{\mathcal{P}r^\mrm{L}_\mrm{St}}
\nc{\Motnc}{\mathcal{M}_{\mrm{loc}}}
\nc{\Motadd}{\mathcal{M}_{\mrm{add}}}
\nc{\Einfty}{{\sE_\infty}}
\nc{\E}[1]{{\sE_{#1}}}
\nc{\modmod}{/\!\!/}
\nc{\heart}{\heartsuit}
\nc{\proj}{\mrm{proj}}
\nc{\LL}{\on{L}}
\nc{\K}{\on{K}}
\nc{\G}{\on{G}}
\nc{\GL}{\on{GL}}
\nc{\BGL}{\on{BGL}}
\nc{\M}{\on{M}}
\nc{\KH}{\on{KH}}
\nc{\Alg}{\on{Alg}}
\nc{\CAlg}{\on{CAlg}}
\nc{\cn}{\mrm{cn}}
\nc{\hw}{\mrm{Hw}}
\nc{\htt}{\mrm{Ht}}
\nc{\Fun}{\on{Fun}}
\nc{\Funadd}{\on{Fun}_{\mrm{add}}}
\nc{\Funex}{\on{Fun}_{\mrm{ex}}}
\nc{\Ind}{\on{Ind}}
\nc{\Pro}{\on{Pro}}
\nc{\Kar}{\on{Kar}}
\nc{\Obj}{\on{Obj}}

\nc{\scr}{\term{simplicial commutative ring}}
\nc{\scrs}{\term{simplicial commutative rings}}

\nc{\Einfring}{\term{$\Einfty$-ring}}
\nc{\Einfrings}{\term{$\Einfty$-rings}}

\nc{\Ering}{\term{$\sE_1$-ring}}
\nc{\Erings}{\term{$\sE_1$-rings}}

\nc{\inftyCat}{\term{$\infty$-category}}
\nc{\inftyCats}{\term{$\infty$-categories}}

\nc{\inftyTop}{\term{$\infty$-topos}}
\nc{\inftyTops}{\term{$\infty$-toposes}}

\nc{\inftyGrpd}{\term{$\infty$-groupoid}}
\nc{\inftyGrpds}{\term{$\infty$-groupoids}}

\def\mc{\mathcal}
\def\mb{\mathbf}
\def\mbb{\mathbb}
\def\op{\mathrm}
\def\ein{\term{$\mathbb{E}_{\infty}$}}
\long\def\enu#1{%
	\begin{enumerate}[label=(\arabic*),font=\normalfont]
		#1
	\end{enumerate}
}

\newcommand{\alg}{\operatorname{Alg}}
\newcommand{\calg}{\operatorname{CAlg}}

\newcommand{\modu}{\operatorname{Mod}}
\newcommand{\lmodu}{\operatorname{LMod}}
\newcommand{\rmodu}{\operatorname{RMod}}

\newcommand{\mapp}{\operatorname{Map}}
\newcommand{\unmap}{\underline{\operatorname{Map}}}
\newcommand{\udmap}{\underline{\operatorname{Map}}}

\newcommand{\colimit}{\mathrm{colim}}
\newcommand{\funct}{\operatorname{Fun}}
\newcommand{\taug}{\tau_{\geq0}}

\newcommand{\groo}{\ensuremath{\operatorname{Groth}_{1}}\xspace}

\newcommand{\ageq}{\ensuremath{\mathcal{A}_{\geq0}}\xspace}
\newcommand{\aaa}{\ensuremath{\mathcal{A}}\xspace}

\newcommand{\prl}{\ensuremath{\mc{P}\mathrm{r}^L}\xspace}
\newcommand{\prlv}{\ensuremath{\mc{P}\mathrm{r}^L_{\mathcal{V}}}\xspace}

\newcommand{\prv}{\ensuremath{\mc{P}\mathrm{r}_{\mathcal{V}}}\xspace}

\newcommand{\pr}{\ensuremath{\mc{P}\mathrm{r}}\xspace}
\newcommand{\prad}{\ensuremath{\mc{P}\mathrm{r}_{\op{ad}}}\xspace}
\newcommand{\prlad}{\ensuremath{\mc{P}\mathrm{r}_{\op{ad}}^L}\xspace}

\newcommand{\pradone}{\ensuremath{\mc{P}\mathrm{r}_{\op{ad},1}}\xspace}
\newcommand{\prladone}{\ensuremath{\mc{P}\mathrm{r}^L_{\op{ad},1}}\xspace}
\newcommand{\prlst}{\ensuremath{\mc{P}\mathrm{r}_{\op{st}}^L}\xspace}

\newcommand{\spgeq}{\ensuremath{\operatorname{Sp}_{\geq0}}\xspace}

\newcommand{\aotimes}{\ensuremath{\mc{A}^\otimes}\xspace}
\newcommand{\botimes}{\ensuremath{\mc{B}^\otimes}\xspace}
\newcommand{\cotimes}{\ensuremath{\mc{C}^\otimes}\xspace}

\newcommand{\hatcat}{\ensuremath{\widehat{\op{Cat}}_{\infty}}\xspace}

\newcommand{\eone}{\term{\mathbb{E}_1}}

\newcommand{\einfring}{\term{$\mathbb E_\infty$-ring}}

\newcommand{\infcat}{\term{$\infty$-category}}
\newcommand{\infcats}{\term{$\infty$-categories}}

\newcommand{\syn}{\op{Syn}}

\newcommand{\mca}{\ensuremath{\mathcal{A}}\xspace}
\newcommand{\mcb}{\ensuremath{\mathcal{B}}\xspace}
\newcommand{\mcc}{\ensuremath{\mathcal{C}}\xspace}
\newcommand{\mcd}{\ensuremath{\mathcal{D}}\xspace}
\newcommand{\mce}{\ensuremath{\mathcal{E}}\xspace}

\newcommand{\mci}{\ensuremath{\mathcal{I}}\xspace}

\newcommand{\mcp}{\ensuremath{\mathcal{P}}\xspace}

\newcommand{\mcs}{\ensuremath{\mathcal{S}}\xspace}

\newcommand{\mcv}{\ensuremath{\mathcal{V}}\xspace}
\newcommand{\mcw}{\ensuremath{\mathcal{W}}\xspace}

\newcommand{\spec}{\op{Spec}}

\newcommand{\coker}{\op{coker}}

\newcommand{\opsp}{\ensuremath{\op{Sp}}\xspace}

\begin{document}
	\title{Higher algebra in $t$-structured tensor triangulated $\infty$-categories}
	\author{Jiacheng Liang}   
	\date{}
	\let\mathbb=\mathbf

	\maketitle
	
	\setstretch{1.1}
	\setcounter{tocdepth}{2}
	
	{\renewcommand{\thefootnote}{}\footnotetext{Date: August 26, 2026}}
	\vspace{-1.5em}
	\begin{abstract}
		We generalize fundamental notions of higher algebra, traditionally developed
		within the $\infty$-category of spectra, to presentably  symmetric
		monoidal stable $\infty$-categories equipped with compatible accessible
		$t$-structures, which we call $ttt$-$\infty$-categories. Under a natural structural condition, which we call ``projective rigidity'', we establish higher categorical analogues of Lazard's theorem and prove the existence and universal property of Cohn localizations.
		
		Furthermore, we generalize higher almost ring theory \cite{hebestreit2024note} to the $ttt$-$\infty$-categorical setting, showing that $\pi_0$-epimorphic idempotent algebras are in natural bijection with idempotent ideals. Additionally, by exploiting deformation theory, we establish a general \text{\text{\'{e}}}tale rigidity theorem.
		
		Finally, we characterize the moduli of such projectively rigid
		$ttt$-$\infty$-categories and show that the $\infty$-category of
		$\mathrm{Sp}$-valued presheaves on the $1$-dimensional framed cobordism
		$\infty$-category is universal among them.
	\end{abstract}
	
	\tableofcontents

	\section*{Introduction}
	\addcontentsline{toc}{section}{Introduction}
	
	Over the past few decades, the integration of homotopy theory and higher categories into derived geometry has fundamentally reshaped modern mathematics. Classical algebraic geometry and ring theory, while incredibly powerful, often struggle to gracefully handle derived phenomena such as non-transverse intersections, higher Tor-groups, and the subtleties of deformation theory. To resolve these limitations, people turned to the derived categories of rings $\mathcal{D}(R)$. However, the structural deficiencies of classical triangulated categories, such as the lack of functorial mapping cones and well-behaved limits, eventually necessitated a more robust and coherent framework. This culminated in the development of \textit{higher algebra} \cite{ha}, which systematically replaces discrete commutative rings with $\mathbb{E}_\infty$-rings and classical module categories with the stable $\infty$-categories of modules over $\mathbb{E}_\infty$-rings. 
	
	In this stable $\infty$-categorical setting, the classical tools of commutative algebra (including tensor products, localizations, flat descent, étale morphisms, and the cotangent complex) are not merely recovered but are vastly generalized and computationally enriched. The resulting machinery over the stable $\infty$-category of spectra $\mathrm{Sp}$ has been highly successful, providing foundational techniques for derived algebraic geometry, chromatic homotopy theory, and modern arithmetic geometry.  By natively encoding derived phenomena, it provides a rigorous homotopical foundation for geometric constructions.

	However, the geometric and algebraic behaviors of  spectra are frequently governed not by their topological nature, but by their interaction with the underlying $t$-structure. A key fact motivating this paper is that the flatness of an  $R$-module over a connective $\mathbb{E}_\infty$-ring $R$ can be entirely characterized by the $t$-exactness of the tensor product functor, as observed in \cite[\textsection7.2.2]{ha}.
	This suggests that the machinery of higher algebra, including (faithfully) flat morphisms, finitely presented morphisms, and étale morphisms, does not strictly depend on the specific $\infty$-category of spectra. Instead, these notions can be generalized to the setting of
	$ttt$-$\infty$-categories, namely presentably  symmetric monoidal stable
	$\infty$-categories equipped with compatible accessible $t$-structures.

	Adopting this generalized framework yields a unified approach to higher algebra
	across diverse underlying geometries, such as equivariant, motivic
	\cite{bachmann2022chow,burklund2020galois}, and analytic settings, where
	nonstandard $t$-structures naturally arise and play an important role. The foundation of this approach is the structural interplay between an accessible $t$-structure and a compatible symmetric monoidal structure, which provides the essential categorical scaffolding required to support these operations. Similar philosophies have recently surfaced in the context of derived (analytic) geometry  \cite{kelly2025localisinginvariantsderivedbornological,ben2024perspective,mantovani2023localizations,kelly2022analytic,raksit2020hochschild}.
	
	The ubiquity of this framework is perhaps best reflected by the rich variety of examples that naturally inhabit it. A central focus of this paper is the notion of \textit{projective rigidity}, a structural condition under which the dualizable objects of the connective part coincide with its compact projectives. This condition is remarkably widespread. Key examples of projectively rigid $ttt$-$\infty$-categories include the following (for a comprehensive list, see \cref{A.6}):
	
	\begin{itemize}
		\item \textbf{Spectra and filtered variants:} The standard $\infty$-category of spectra $\opsp$, graded spectra $\op{Gr(\op{Sp})}$, and filtered spectra $\op{Fil}(\opsp)$ equipped with either the pointwise or the homotopy $t$-structure.
		
		\item \textbf{Synthetic and Dirac spectra:} For an $\mathbb{E}_\infty$-ring $R$, the $\infty$-category of (connective) synthetic $R$-modules $\mathrm{Syn}_{R,\geq 0} = \mathcal{P}_\Sigma(\mathrm{Mod}_R^{\mathrm{ff}}(\op{Sp}))$ \cite{hesselholt2023dirac}. Here, $\mathrm{Mod}_R^{\mathrm{ff}}(\op{Sp}) \subset \modu_R(\opsp)$ denotes the smallest full subcategory containing $R$ that is closed under finite direct sums, shifts, and retractions. This forms a projectively rigid $ttt$-$\infty$-category whose heart is $\mathrm{Syn}_R^\heartsuit \simeq \mathrm{Mod}_{\pi_*(R)}(\mathrm{Ab})$.
		
		\item \textbf{Equivariant and motivic contexts:} The $\infty$-category of genuine $G$-spectra $\op{Sp}_{G}$ for a finite group $G$ equipped with the pointwise $t$-structure, the Artin motivic spectra $\op{SH}(k)^{\op{A}}$ over a perfect field $k$ equipped with the very effective $t$-structure \cite{burklund2020galois}, and the Voevodsky $\infty$-category $\op{DM}(k,\mathbb{Z}[1/p])$ equipped with the Chow $t$-structure \cite{bondarko2010weight,bachmann2021norms}.
		
		\item \textbf{Geometric and sheaf-theoretic examples:} (Derived) quasi-coherent sheaves $\op{QCoh}(X)$ on an affine quotient stack, and the $\infty$-category of sheaves $\op{Shv}(X,\op{Sp})^\otimes$ on a stone space.
	\end{itemize}
	
	In this work, we systematically develop the theory of higher algebra internal to $ttt$-$\infty$-categories. Under the assumption of projective rigidity, we establish generalized analogues of Lazard's theorem and étale rigidity. Furthermore, we connect our framework with recent advances in higher almost ring theory \cite{hebestreit2024note} and locally rigid $\infty$-categories \cite{ramzi2024locally}. Ultimately, we show that the class of projectively rigid $ttt$-$\infty$-categories is governed by a universal geometric example related to the $1$-dimensional cobordism category.
	\begin{cov}
		Throughout this paper, a $ttt$-$\infty$-category
		\[
		(\mathcal{A}^\otimes,\mathcal{A}_{\geq0},\mathcal{A}_{\leq0})
		\]
		means a presentably  symmetric monoidal stable $\infty$-category
		$\mathcal{A}^\otimes\in\operatorname{CAlg}(\prlst)$ equipped with an
		accessible $t$-structure
		$(\mathcal{A}_{\geq0},\mathcal{A}_{\leq0})$
		compatible with the symmetric monoidal structure in the following sense:
		\enu{
			\item The unit $\mathbf{1}$ belongs to $\mathcal{A}_{\geq0}$;
			\item If $X,Y\in\mathcal{A}_{\geq0}$, then
			$X\otimes Y\in\mathcal{A}_{\geq0}$.
		}
	\end{cov}
	\begin{rem}
		The term $ttt$ is used here as a name and is motivated by the tensor
		triangulated structure on the homotopy category
		$\operatorname{h}\mathcal{A}$ together with the chosen $t$-structure
		on $\mathcal{A}$.
	\end{rem}
	\subsection*{Main results}
	\addcontentsline{toc}{subsection}{Main results}

	We begin with the structure of flat and almost perfect modules. Under the assumption of projective rigidity, we establish a higher categorical analogue of Lazard's theorem, demonstrating that flat modules are entirely controlled by the filtered colimits of compact projective modules:
	
	\begin{ThmAlpha}[\cref{cocompfil}, and \cref{cocompsimplicial}]
		Let $R \in \operatorname{Alg}(\mathcal{A}_{\geq 0})$ be a connective $\eone$-algebra in \mca. Then the following hold:
		\enu{
			\item (Lazard’s Theorem.)\\
			Suppose that $\mathcal{A}^\otimes_{\geq 0}$ is projectively rigid (see \cref{projrigid}). Then an $R$-module is flat if and only if it can be written as a filtered colimit of compact projective $R$-modules. 
			\item 
			Suppose that $\mathcal{A}_{\geq 0}$ is projectively generated.	An $R$-module is almost perfect if and only if it can be written as a geometric realization of compact projective $R$-modules.
		}
	\end{ThmAlpha}
	
	This tight control over modules via projective rigidity allows us to systematically lift discrete algebraic structures from the heart $\mathcal{A}^\heartsuit$ to the ambient category $\mathcal{A}$. Specifically, we show that the derived category of a discrete ring in the heart provides a  model for its module category in $\mathcal{A}$, generalizing such result in the case of spectra.
	
	\begin{ThmAlpha}[\cref{dicre}]
		Suppose that the $ttt$-\infcat $(\mc{A}^\otimes, \mc{A}_{\geq0}) $ is projectively rigid. 
		For any discrete $\eone$-algebra $R\in \alg(\mathcal{A}^{\heartsuit})$, there exists a unique equivalence 
		$$\mc{D}(\lmodu_{\pi_0R}(\mathcal{A}^{\heartsuit}))\xrightarrow{\sim} \lmodu_{R}(\mathcal{A})$$
		which induces the identity functor on the heart.
		Furthermore, if $R$ is an \ein-algebra, then this equivalence is symmetric monoidal.	
	\end{ThmAlpha}

	We demonstrate that a faithfully flat map whose $\pi_0$-truncation is an $\aleph_n$-compact module automatically is descendable, generalizing a result in \cite{mathew2016galois}.
	\begin{ThmAlpha}[{\cref{ffdescendable}}]
		Suppose that the $ttt$-\infcat $(\mc{A}^\otimes, \mc{A}_{\geq0}) $ is projectively rigid.  Let $\phi: A\to B \in \calg(\mc{A})$ be a faithfully flat map of \ein-algebras such that $\pi_0B$ is a $\aleph_n$-compact $\pi_0A$-module\footnote{It means $\pi_0B\in\modu_{\pi_0A}(\mca^\heartsuit)^{\aleph_n}$.} for some $n\geq 0$. Then $\phi$ is descendable.
	\end{ThmAlpha}
	Another fundamental result about (higher) idempotent algebras is the theory of (higher) almost mathematics. Generalizing recent work of Hebestreit and Scholze \cite{hebestreit2024note} to the $ttt$-$\infty$-categorical setting, we demonstrate that $\pi_0$-epimorphic idempotent algebras can be classified entirely in terms of classical idempotent ideals in the discrete heart:
	
	\begin{ThmAlpha}[\cref{almostalg}]
		Assume that $\mathcal{A}$ is both right and left complete. Let $R \in \calg(\mathcal{A}_{\geq 0})$. Consider the full subcategory $\mathrm{LQ}_R$ of $\mathrm{CAlg}(\mathcal{A}_{\geq 0})_{R/}$ spanned by the maps $\varphi: R \rightarrow S$ for which
		\enu{ 
			\item the multiplication $S \otimes_R S \rightarrow S$ is an equivalence, i.e., $\varphi$ is idempotent,
			\item $\pi_0(\varphi): \pi_0 R \rightarrow \pi_0 S$ is epimorphic in $\mathcal{A}^\heartsuit$.
		}
		Then the functor
		$$\mathrm{LQ}_R \longrightarrow\left\{I \subset \pi_0 R \mid I^2=I\right\}, \quad \varphi \longmapsto \operatorname{Ker}(\pi_0 \varphi)$$
		is an equivalence of categories, where we regard the target as a poset via the inclusion ordering. The inverse image of some $I \subset \pi_0(R)$ can be described more directly as $R / I^{\infty}$, where
		$$I^{\infty}=\lim _{n \in \mathbb{N}^{\mathrm{op}}} J_I^{\otimes_R n}$$
		with $J_I \rightarrow R$ the fibre of the canonical map $R \rightarrow H(\pi_0(R) / I)$. Furthermore, this inverse system stabilizes on $\pi_i$ for $n>i+1$. 
		
		The image of the fully faithful restriction functor $\operatorname{Mod}_{R / I^{\infty}}(\mathcal{A})\rightarrow \operatorname{Mod}_R(\mathcal{A})$ consists exactly of those modules whose homotopy is killed by $I$.
	\end{ThmAlpha}
	
	This higher algebraic behavior naturally leads to deformation-theoretic questions. 
	By utilizing the cotangent complex formalism over a $ttt$-$\infty$-base, we establish a general \text{\text{\'{e}}}tale rigidity theorem. 
	This generalizes both Lurie's foundational result for spectra \cite[\textsection7.5]{ha} and recent developments by Hesselholt and Pstr\k{a}gowski in Dirac geometry \cite[Theorem 4.33]{hesselholt2023dirac}\footnote{Obtained by applying $\mathcal{A}_{\geq0}=\mathrm{Syn}_{R,\geq0}$.}.
	
	\begin{ThmAlpha}[\cref{etrig}, Étale rigidity]
		Let $A\in \operatorname{CAlg}(\mathcal{A})$ be an \ein-algebra.
		\begin{enumerate}[label=(\arabic*),font=\normalfont]
			\item Suppose that $\mathcal{A}$ is Grothendieck (see \cref{gro}) and left complete. Let $\operatorname{CAlg}(\mathcal{A})_{A /}^{\text{fl},L\text{-}\mathrm{et}}$ denote the full subcategory of $\operatorname{CAlg}(\mathcal{A})_{A /}$ spanned by the flat L-étale maps $A \rightarrow B$. Suppose that $A$ is connective. Then the functor $\pi_0$ induces an equivalence
			$$\operatorname{CAlg}(\mathcal{A})_{A /}^{\text{fl},L\text{-}\mathrm{et}}\xrightarrow{\sim} \operatorname{CAlg}(\mathcal{A}^{\heartsuit})_{\pi_0A /}^{\text{fl},L\text{-}\mathrm{et}}$$
			with the (1-)category of the discrete flat L-étale commutative $\pi_0 A$-algebras.
			\item Suppose that the $ttt$-\infcat $(\mc{A}^\otimes, \mc{A}_{\geq0}) $ is projectively rigid. Let $\operatorname{CAlg}(\mathcal{A})_{A /}^{\mathrm{et}}$ denote the full subcategory of $\operatorname{CAlg}(\mathcal{A})_{A /}$ spanned by the étale maps $A \rightarrow B$. Then the functor $\pi_0$ induces an equivalence
			$$\operatorname{CAlg}(\mathcal{A})_{A /}^{\mathrm{et}}\xrightarrow{\sim} \operatorname{CAlg}(\mathcal{A}^{\heartsuit})_{\pi_0A /}^{\mathrm{et}}$$
			with the (1-)category of the discrete étale commutative $\pi_0 A$-algebras.
		\end{enumerate}
	\end{ThmAlpha}
	
	Finally, having developed this comprehensive algebraic machinery internal to a fixed $ttt$-$\infty$-category, we zoom out to study the global moduli of such categories. We prove that the entire landscape of projectively rigid $ttt$-$\infty$-categories is governed by a remarkable geometric universal  example:
	
	\begin{ThmAlpha}[\cref{uniexamp}, Universal projectively rigid $ttt$-$\infty$-category] The \infcat 	$\calg^{\op{rig,at}}_{\spgeq}$ of projectively rigid $ttt$-$\infty$-categories
		admits a compact generator
		$$\funct(\mb{Cob}_1^{\op{op}},\spgeq)^\otimes\in \calg^{\op{rig,at}}_{\spgeq}$$
		where the symmetric monoidal structure is given by Day convolution and $\mb{Cob}_1$ denotes the $1$-dimensional framed cobordism $(\infty,1)$-category.
	\end{ThmAlpha}
	\subsection*{Outline}
	\addcontentsline{toc}{subsection}{Outline}
	The paper is organized as follows.
	\begin{itemize}
		\item \textbf{Section 1 (Preliminaries):} We review the theory of prestable $\infty$-categories and presentable Grothendieck categories, establishing the foundational relationship between a presentable prestable $\infty$-category and its stabilization.
		\item \textbf{Section 2 (Flatness and Faithful Flatness):} We redefine flatness via the $t$-exactness of the relative tensor product. We establish the stability of flat modules under base change and filtered colimits.
		\item \textbf{Section 3 (Projective Modules):} We introduce the core condition of \textit{projective rigidity}. Under this hypothesis, we prove a generalized Lazard's Theorem, showing that flat modules are precisely filtered colimits of compact projectives, and establish the universal property of Cohn localizations.
		\item \textbf{Section 4 (Finiteness Properties):} We define perfect and almost perfect modules, utilizing Tor-amplitude bounds to control homological behavior in a $ttt$-$\infty$-category.
		\item \textbf{Section 5 (Faithful and Descendable Algebras):} We study descent in higher algebra. Generalizing a theorem of Hebestreit and Scholze \cite{hebestreit2024note}, we classify epimorphic idempotent algebras via idempotent ideals.
		\item \textbf{Section 6 (Deformation Theory and Étale Rigidity):} We employ the cotangent complex to establish the property of étale rigidity in this setting. We show that the étale topology of an $\mathbb{E}_\infty$-algebra depends only on its $\pi_0$-truncation, which extends the result found in \cite[\textsection7.5]{ha}.
		\item \textbf{Section 7 (The $\infty$-category of projectively rigid $ttt$-$\infty$-categories):} We prove that the presheaf $\infty$-category on the $1$-dimensional framed cobordism category serves as the universal projectively rigid $ttt$-$\infty$-category.
		\item \textbf{Section 8 (Questions and Future Directions):} We conclude by discussing several remaining questions and future directions.
	\end{itemize}
	
	\subsection*{Conventions and notations}
	\addcontentsline{toc}{subsection}{Conventions and notations}
	In this paper, we make heavy use of the theory of $\infty$-categories and related notions, as developed in Lurie’s foundational references \cite{htt,ha,sag}.
	
	\begin{cov}
		\enu{
			\item Since a $t$-structure is determined by its connective part, we will simply denote a $ttt$-$\infty$-category by $(\mathcal{A}^\otimes, \mathcal{A}_{\geq 0})$ rather than $(\mathcal{A}^\otimes, \mathcal{A}_{\geq 0}, \mathcal{A}_{\leq 0})$.
			\item All $ttt$-$\infty$-categories considered in this paper are
			presentable; we do not consider a small analogue.
			\item We will often say $\mathcal{A}$ satisfies some property if the $t$-structured stable $\infty$-category $(\mathcal{A}, \mathcal{A}_{\geq 0})$ satisfies property $P$. For example, we will say that $\mathcal{A}$ is right complete if its $t$-structure is right complete. We will say that $\aotimes$ satisfies some property if this property involves the monoidal structure.
			\item We often denote the tensor product in the truncated \infcat $\mathcal{A}_{[0,n]}$ by $-\overline{\otimes}-$.
		}
	\end{cov}
	
	\begin{nota}
		For the reader's convenience, we record below the recurring  notations used throughout.
		
		\begin{itemize}
			\item In this paper, \infcats refer to $(\infty,1)$-categories.
			\item We denote by $\mathcal{S}$  the $\infty$-category of spaces (i.e., $\infty$-groupoids). We denote by $\mathrm{Cat}_{\infty}$  the $\infty$-category of small $\infty$-categories, functors, and natural transformations. We denote by $\mathrm{Sp}$  the $\infty$-category of spectra.
			
			\item For $n\in\mathbb{N}$, $\mathcal{S}_{\leq n}$ is the full subcategory of $\mathcal{S}$ spanned by $n$-truncated spaces. For example, $\mathcal{S}_{\leq -1}$ is the full subcategory spanned by $\varnothing$ and $*$, and $\mathcal{S}_{\leq -2}$ is the full subcategory spanned by $*$.
			
			\item For functor categories $\mathrm{Fun}(\mathcal{C},\mathcal{D})$, we use superscripts $\mathrm{lex}$, $\mathrm{rex}$, $\mathrm{ex}$, $\mathrm{lim}$, and $\mathrm{colim}$ to indicate the full subcategories of functors preserving finite limits, finite colimits, finite limits and finite colimits (exact functors), small limits, and small colimits, respectively.
			\item 
			We use the superscript in $\mcc^\otimes$ to denote the $\infty$-category \mcc equipped with a certain ($\mathbb{E}_k$-)monoidal structure for some $0 \leq k \leq \infty$.
			
			\item Let $\mcc^\otimes$ be a monoidal $\infty$-category. We denote by $\alg(\mc{C})$ the $\infty$-category of $\mathbb{E}_1$-algebras in $\mc{C}$.  If $\mcc^\otimes$ is $\mbb{E}_k$-monoidal  for some $0 \leq k \leq \infty$, we denote by $\alg_{\mbb{E}_k}(\mc{C})$ the $\infty$-category of $\mathbb{E}_k$-algebras in $\mc{C}$. In particular when $\mcc^\otimes$ is symmetric monoidal, we denote by $\calg(\mc{C})$ the $\infty$-category of $\mathbb{E}_{\infty}$-algebras (i.e. commutative algebra objects) in $\mc{C}$.
			\item We work relative to a chain of strongly inaccessible cardinals $\delta_{0}<\delta_{1}<\delta_{2}$. Then each $V_{\delta_i}$ is a Grothendieck universe, and elements of $V_{\delta_0},V_{\delta_1},V_{\delta_2}$ are called small, large, and very large, respectively. A hat $\widehat{\cdot}$ indicates largeness of the objects in the category, e.g., $\widehat{\mathcal{S}}$ or $\widehat{\mathrm{Cat}}_{\infty}$ for the (very large) $\infty$-categories of large spaces or large $\infty$-categories.
			\item We denote by $\prl$  the $\infty$-category of presentable $\infty$-categories with left adjoint functors. Its default symmetric monoidal structure is the Lurie tensor product, which corepresents functors out of the Cartesian product that are colimit-preserving in each variable. For a cardinal $\kappa$, $\prl_{\kappa}\subset\prl$ denotes the subcategory of $\kappa$-accessible presentable $\infty$-categories with functors that preserve $\kappa$-compact objects.
			\item We denote $\prlst$ to be  the $\infty$-category of presentable stable $\infty$-categories. We denote $\prlad$ to be  the $\infty$-category of presentable additive $\infty$-categories. We denote $\prladone$ to be  the $\infty$-category of presentable additive $1$-categories.
			\item We use the notation $\unmap$ to indicate  enrichment if the enriched context is clear.
		\end{itemize}
	\end{nota}
	
	\subsection*{Acknowledgments}
	\addcontentsline{toc}{subsection}{Acknowledgments}
	
	I would like to express my sincere appreciation to Germán Stefanich for sharing his work \cite{stefanich2023classification}, which has laid the groundwork for many foundational results closely tied to this paper. I am also grateful to Jack Davies for highlighting the possibility of étale rigidity in my context, and to Maxime Ramzi for emphasizing the equivalence between our notion of ``projectively rigid'' and the atomically generated rigid $\spgeq$-algebras. 
	Additionally, I wish to convey my deep gratitude to my advisor, David Gepner, for his unwavering support and inspiration throughout this endeavor.
	
	I would also like to thank Ko Aoki, Tobias Barthel, Tom Bachmann, Tim Campion, Marc Hoyois, Xiansheng Li, Marius Nielsen, Arpon Raksit, Xiangrui Shen, Vova Sosnilo, Yifei Zhu, and Changhan Zou for their helpful conversations at various stages of this project. I am also grateful to the anonymous referee for a careful reading of the manuscript and for numerous helpful comments and suggestions.
	
	Furthermore, much of this work was carried out during my visit to the Max Planck Institute for Mathematics (MPIM) in Spring 2025, and I gratefully acknowledge the hospitality extended to me by the institute.
	\section{Preliminaries}
	The goal of this section is to gather basic preliminaries about prestable \infcats and abelian categories.
	
	\subsection{Prestable $\infty$-categories}
	Before developing homological algebra over arbitrary bases, we must establish the formal properties of the connective part underlying our $ttt$-$\infty$-categories. In this subsection, we recall the theory of prestable $\infty$-categories as developed in \cite{sag}, focusing heavily on Grothendieck prestable $\infty$-categories. These categories naturally serve as the connective parts of Grothendieck $t$-structures. 
	
	\begin{de}[{See \cite{sag} C.1.2.2}]
		A prestable $\infty$-category is an $\infty$-category $\mathcal{C}$ satisfying the following properties:
		
		\begin{enumerate}[label=(\arabic*),font=\normalfont]
			\item The initial and final objects of $\mathcal{C}$ agree (that is, $\mathcal{C}$ is pointed).
			\item Every cofiber sequence in $\mathcal{C}$ is also a fiber sequence.
			\item Every map in $\mathcal{C}$ of the form $f: X \rightarrow \Sigma(Y)$ is the cofiber of its fiber.
		\end{enumerate}
		Moreover, we say $\mathcal{C}$ is a Grothendieck prestable $\infty$-category if it further satisfies that it is presentable and that filtered colimits and finite limits commute in $\mathcal{C}$. We let $\op{Groth}_{\infty}\subset \prl$ denote the full subcategory whose objects are Grothendieck prestable $\infty$-categories.
	\end{de}
	
	\begin{ex}
		\enu{
			\item Any stable $\infty$-category is prestable.
			\item Let $\mathcal{C}$ be a stable $\infty$-category equipped with a $t$-structure $\left(\mathcal{C}_{\geq 0}, \mathcal{C}_{\leq 0}\right)$. Then the full subcategory $\mathcal{C}_{\geq 0} \subset \mathcal{C}$ is prestable.
		}
	\end{ex}
	Prestable $\infty$-categories are closely related to stable
	$\infty$-categories equipped with $t$-structures. We will freely use
	the standard results on prestable and Grothendieck prestable
	$\infty$-categories from \cite[Appendix C]{sag}. In particular, a
	prestable $\infty$-category with finite limits can be identified with
	the connective part of a stable $\infty$-category equipped with a
	$t$-structure, and the Grothendieck condition is equivalent to the
	compatibility of the canonical $t$-structure on its stabilization with
	filtered colimits; see \cite[\S C.1]{sag}.
	\begin{thm}[See \cite{sag} C.4.2.1]
		The full subcategory $\mathrm{Groth}_{\infty}\subset \prlad$ contains the unit object of $\op{Sp}_{\geq0}$ and is closed under Lurie tensor products. Consequently, $\mathrm{Groth}_{\infty}$ inherits a symmetric monoidal structure for which the inclusion $\mathrm{Groth}_{\infty} \hookrightarrow \prlad$ is symmetric monoidal.
	\end{thm}
	
	\begin{de}[See \cite{sag} C.3.1.3]
		Let $\mathcal{C}$ be a presentable stable $\infty$-category. We define a full subcategory $\mathcal{C}_{\geq 0} \subset \mathcal{C}$ to be a core if it is closed under small colimits and extensions. 
		
		We will refer to $\pr^{+}_{\op{st}}$ as the $\infty$-category of cored stable $\infty$-categories. The objects of $\pr^{+}_{\op{st}}$ are pairs $(\mathcal{C}, \mathcal{C}_{ \geq 0})$, where $\mathcal{C}$ is a presentable stable $\infty$-category and $\mathcal{C}_{\geq 0} \subset \mathcal{C}$ is a core. A morphism from $(\mathcal{C}, \mathcal{C}_{ \geq 0})$ to $(\mathcal{D}, \mathcal{D}_{ \geq 0})$ is given by a colimit-preserving functor $f: \mathcal{C} \rightarrow \mathcal{D}$ satisfying $f\left(\mathcal{C}_{\geq 0}\right) \subset \mathcal{D}_{\geq 0}$.
	\end{de}
	\begin{warn}
		$\mathcal{C}_{\geq 0}$ in this definition does not necessarily form a $t$-structure, but if $\mathcal{C}_{\geq 0}$ is presentable, then it does. In that case, we call the pair $(\mathcal{C}, \mathcal{C}_{ \geq 0})$ a \emph{presentably $t$-structured stable $\infty$-category}.
	\end{warn}
	\begin{rem}
		\enu{
			
			\item Our $\pr^{+}_{\op{st}}$ actually refers to $\mathrm{Groth}_{\infty}^{+}$ in \textup{\cite[Remark C.3.1.3]{sag}}.
			\item In fact, we only care about the full subcategory $\pr^{t\text{-rex}}_{\op{st}}\subset\pr^{+}_{\op{st}}$ spanned by those presentably $t$-structured stable $\infty$-categories with right $t$-exact functors. However, the technical advantage of $\pr^{+}_{\op{st}}$ is that it admits good colimits and limits \textup{\cite[Remark C.3.1.7]{sag}}.
			
		}
	\end{rem}
	
	\begin{de}
		There is a natural symmetric monoidal structure on $\pr^{+}_{\op{st}}$ given by the construction:
		$$
		\left(\mathcal{C}, \mathcal{C}_{\geq 0}\right) \otimes\left(\mathcal{D}, \mathcal{D}_{\geq 0}\right)=\left(\mathcal{C} \otimes \mathcal{D}, m_{!}\left(\mathcal{C}_{\geq 0}, \mathcal{D}_{\geq 0}\right)\right)
		$$
		where $\mathcal{C} \otimes \mathcal{D}$ is the Lurie tensor product and $m_{!}\left(\mathcal{C}_{\geq 0}, \mathcal{D}_{\geq 0}\right)$ is the smallest full subcategory of $\mathcal{C} \otimes \mathcal{D}$ which is closed under colimits and extensions and contains the objects $m(C, D)$ for each $C \in \mathcal{C}_{\geq 0}$ and $D \in \mathcal{D}_{\geq 0}$.
	\end{de}
	
	\begin{rem}\label{prt-rex}
		\enu{
			\item The full subcategory $\pr^{t\text{-rex}}_{\op{st}}\subset\pr^{+}_{\op{st}}$ is closed under tensor products, since $m_{!}\left(\mathcal{C}_{\geq 0}, \mathcal{D}_{\geq 0}\right)$ is presentable if both $\mathcal{C}_{\geq 0}$ and $\mathcal{D}_{\geq 0}$ are presentable.
			\item An object in $\op{CAlg}(\pr^{t\text{-rex}}_{\op{st}})$ can be identified with a $ttt$-$\infty$-category.
		}
	\end{rem}
	
	We will also use that stabilization identifies colimit-preserving
	functors between Grothendieck prestable $\infty$-categories with
	right $t$-exact colimit-preserving functors between their
	stabilizations; see \cite[Proposition C.3.1.1]{sag}.
	
	\begin{cor}\label{pret}\,
		\enu{
			\item The construction $\mathcal{C} \mapsto\left(\operatorname{Sp}(\mathcal{C}), \operatorname{Sp}(\mathcal{C})_{\geq 0}\right)$ determines a fully faithful embedding $$\op{Groth}_{\infty}\hookrightarrow \pr^{+}_{\op{st}}$$ from the $\infty$-category of Grothendieck prestable $\infty$-categories to the $\infty$-category of cored stable $\infty$-categories.
			\item A pair $(\mathcal{C}, \mathcal{C}_{ \geq 0})$ belongs to the essential image of $\op{Groth}_{\infty}\hookrightarrow \pr^{+}_{\op{st}}$ if and only if it forms an accessible $t$-structure $\left(\mathcal{C}_{\geq 0}, \mathcal{C}_{\leq 0}\right)$ which is compatible with filtered colimits and is right complete.
			\item Furthermore, the embedding $\op{Groth}_{\infty}\hookrightarrow \pr^{+}_{\op{st}}$ is symmetric monoidal, hence induces a fully faithful embedding $$\op{CAlg}(\op{Groth}_{\infty})\hookrightarrow \op{CAlg}(\pr^{+}_{\op{st}}).$$
		}
	\end{cor}
	
	\begin{de}[Grothendieck $t$-structured stable $\infty$-categories]\label{gro}
		We say that a presentably $t$-structured stable $\infty$-category $(\mathcal{C},\mathcal{C}_{\geq0})\in \pr^{t\text{-rex}}_{\op{st}}$ is Grothendieck if it lies in the essential image of the embedding $\op{Groth}_{\infty}\hookrightarrow \pr^{t\text{-rex}}_{\op{st}}$, or equivalently, if the $t$-structure on $\mathcal{C}$ is right complete and compatible with filtered colimits \textup{(see \cite[Remark C.3.1.5]{sag})}.
	\end{de}
	
	\begin{prop}
		Let $\mcc\in \prlst$ is a compactly generated stable \infcat and $S\subset \mcc^\omega$ be an arbitrary subset. Let $\mcc_{\geq0}\subset \mcc$ be the smallest full subcategory generated by $S$ under small colimits and extensions. Then the $t$-structure $(\mcc,\mcc_{\geq0})$ is compatible with filtered colimits\footnote{One can even show that $\mcc_{\geq0}$ is compactly generated and the inclusion to \mcc preserves compact objects; cf. \cite[Remark 2.5]{rossanigo2026quasiaffineschemessinglycompactly}.}.
		
		Furthermore, if $S$ is a set of compact generators (meaning they generate \mcc under small colimits and shifts),
		then $(\mcc,\mcc_{\geq0})$ is a Grothendieck $t$-structured stable \infcat.
	\end{prop}
	\begin{proof}
		By \cite[Proposition 1.4.4.11]{ha}, $\mcc_{\geq 0}$ forms an accessible $t$-structure on $\mcc$. 
		By the orthogonality property of $t$-structures, an object $X \in \mcc$ belongs to $\mcc_{\leq 0}$ if and only if $\map_{\mcc}(Y, X) \simeq 0$ for all $Y \in \mcc_{\geq 1}$. Since $\mcc_{\geq 1}$ is generated under small colimits and extensions by $S[1] = \{K[1] \mid K \in S\}$, it suffices to check orthogonality against the generators. That is, $X \in \mcc_{\leq 0}$ if and only if $\map_{\mcc}(K[1], X) \simeq 0$ for all $K \in S$, so $\mcc_{\leq 0}\hookrightarrow\mcc$ is closed under filtered colimits by compactness. 
		
		Now assume that $S$ is a set of compact generators. To show that this $t$-structure is Grothendieck, it suffices to verify two properties: 
		\enu{\item  compatibility with filtered colimits (i.e., $\mcc_{\leq 0}$ is closed under filtered colimits);
			\item  $\bigcap_{n \in \mathbb{Z}} \mcc_{\leq n} \simeq \{0\}$ (this implies right completeness under the condition (1) by the dual of \cite[Proposition 1.2.1.19]{ha}).
		}		
		We only need to show that the infinite intersection $\bigcap_{m \in \mathbb{Z}} \mcc_{\leq m}=0$. Suppose $X \in \bigcap_{m \in \mathbb{Z}} \mcc_{\leq m}$. For any integer $n \in \mathbb{Z}$, we have $X \in \mcc_{\leq n-1}$. By orthogonality, $\map_{\mcc}(A, X) \simeq 0$ for all $A \in \mcc_{\geq n}$.
		Since $K[n] \in \mcc_{\geq n}$ for all $K \in S$, it follows that $\map_{\mcc}(K[n], X) \simeq 0$ for all $K \in S$ and all $n \in \mathbb{Z}$. 
		Since $X$ is orthogonal to all shifts of the generators, we must have $X \simeq 0$. Therefore, the $t$-structure is right complete.
	\end{proof}
	
	\begin{de}[See \cite{sag} C.1.2.12]
		Let $\mathcal{C}$ be a prestable $\infty$-category which admits finite limits. We say that an object $X \in \mathcal{C}$ is $\infty$-connective if $\tau_{\leq n} X \simeq 0$ for every integer $n$. \\
		We say $\mathcal{C}$ is separated if every $\infty$-connective object of $\mathcal{C}$ is a zero object. \\
		We say $\mathcal{C}$ is left complete if it is a homotopy limit of the tower of $\infty$-categories:
		$$
		\cdots \rightarrow \tau_{\leq 2} \mathcal{C} \xrightarrow{\tau_{\leq 1}} \tau_{\leq 1} \mathcal{C} \xrightarrow{\tau \leq 0} \tau_{\leq 0} \mathcal{C}=\mathcal{C}^{\heartsuit} .
		$$
		In other words, $\mathcal{C}$ is left complete if it is Postnikov complete.
	\end{de}
	
	\begin{rem}
		\enu{
			\item If a prestable $\infty$-category $\mathcal{C}$ is left complete, then it is separated.
			\item Let $\mathcal{C}$ be a stable $\infty$-category. Then $\mathcal{C}$ is separated if and only if $\mathcal{C} \simeq *$.
		}
	\end{rem}
	
	\begin{prop}\label{comple}
		Let $\mathcal{C}$ be a prestable $\infty$-category with finite limits. Then:
		\enu{
			\item The canonical $t$-structure $(\op{Sp}(\mathcal{C})_{\geq0},\op{Sp}(\mathcal{C})_{\leq0})$ is hypercomplete if and only if $\mathcal{C}$ is separated.
			\item The canonical $t$-structure $(\op{Sp}(\mathcal{C})_{\geq0},\op{Sp}(\mathcal{C})_{\leq0})$ is left complete if and only if $\mathcal{C}$ is left complete.
		}
	\end{prop}
	
	\begin{proof}\,
		(1) The ``only if'' direction is obvious. Now assume that $\mathcal{C}$ is separated and that $X\in \op{Sp}(\mathcal{C})$ satisfies $\tau_{\leq i}X=0$ for every integer $i$. We need to show that $X=0$. Since $\op{Sp}(\mathcal{C})$ is right complete, we have $X\simeq \varinjlim \tau_{\geq -n}X$. However, each $\Sigma^n\tau_{\geq -n}X$ is $\infty$-connective as an object in $\mathcal{C}$, so $\tau_{\geq -n}X=0$ by assumption, therefore $X=0$.
		
		(2) The ``only if'' direction is obvious. Now assume that $\mathcal{C}$ is left complete. We wish to show the diagram
		$$\begin{tikzcd}
			& \vdots \arrow[d] \\
			& \op{Sp}(\mathcal{C})_{\leq1} \arrow[d] \\
			\op{Sp}(\mathcal{C}) \arrow[r] \arrow[ru] \arrow[ruu] & \op{Sp}(\mathcal{C})_{\leq0}
		\end{tikzcd}$$
		is a limit diagram of \infcats. Since $\op{Sp}(\mathcal{C})$ is right complete, we have that $$\op{Sp}(\mathcal{C})\xrightarrow{\varprojlim\tau_{\geq -n}} \varprojlim\op{Sp}(\mathcal{C})_{\geq -n}$$ is an equivalence of \infcats. However, the diagram
		$$\begin{tikzcd}
			& \vdots \arrow[d] \\
			& \op{Sp}(\mathcal{C})_{[-n,1]} \arrow[d] \\
			\op{Sp}(\mathcal{C})_{\geq -n} \arrow[r] \arrow[ru] \arrow[ruu] & \op{Sp}(\mathcal{C})_{[-n,0]}
		\end{tikzcd}$$
		is a limit diagram of \infcats by the left completeness of $\mathcal{C}$, so by \cite[Lemma 5.5.2.3]{htt} we are done.
	\end{proof}
	\begin{de}
		Let $\mcc$ be a prestable \infcat with finite limits. We will say that a tower 
		$$X_{\bullet}: (\mathbb{Z}_{\geq 0}^{\triangleright })^{\op{op}} \rightarrow \mcc$$
		is highly connected if, for every $n \geq 0$, there exists an integer $k$ such that the induced map $\tau_{\leq n} X_\infty \rightarrow \tau_{\leq n} X_{k^{\prime}}$ is an equivalence for $k^{\prime} \geq k$. We will say that a pretower 
		$$Y_{\bullet}: \mathbb{Z}_{\geq 0}^{\op{op}} \rightarrow \mcc$$
		is highly connected if, for every $n \geq 0$, there exists an integer $k$ such that the map $\tau_{\leq n} Y_{k^{\prime \prime}} \rightarrow \tau_{\leq n} Y_{k^{\prime}}$ is an equivalence for $k^{\prime \prime} \geq k^{\prime} \geq k$. It is clear that every Postnikov (pre)tower is highly connected.
	\end{de}
	\begin{prop}\label{limtower}
		Assume $\mcc$ is a left complete prestable $\infty$-category with countable limits. Let $Y_\bullet:\mathbb{Z}_{\geq 0}^{\op{op}} \rightarrow \mcc$ be a highly connected pretower. Then:
		\enu{\item For any $n\geq0$, there exists an integer $k$ such that for any $l\geq k$ the natural projection $$\varprojlim Y_\bullet\to Y_{l}$$ has $n$-connective cofiber.
			\item Suppose further that \mcc admits a  monoidal structure $\mcc^\otimes$ which is compatible with finite colimits. For any object $X\in\mcc$, the natural maps $$X\otimes \varprojlim Y_\bullet\to \varprojlim (X\otimes Y_\bullet) \,\text{ and }\,  \varprojlim Y_\bullet\otimes X\to \varprojlim ( Y_\bullet \otimes X)$$ are  equivalences.
		}
	\end{prop}
	\begin{proof}
		(1) Let us analyze the Postnikov towers of the objects $Y_m$. For each integer $d$, let $\tau_{\leq d}$ denote the truncation functor. By the definition of a highly connected pretower, the connectivity of the cofiber of the transition maps $Y_m \to Y_l$ tends to infinity. This means that for any fixed degree $d \geq 0$, there exists an integer $K(d)$ such that for all $m > l \geq K(d)$, the cofiber of $Y_m \to Y_l$ is $(d+1)$-connective. 
		Consequently, applying the truncation functor yields an equivalence $\tau_{\leq d} Y_m \xrightarrow{\sim} \tau_{\leq d} Y_l$ for all $m > l \geq K(d)$. In other words, for any fixed $d$, the tower $\{\tau_{\leq d} Y_m\}_{m}$ is essentially constant (pro-constant), and its limit is achieved at the finite stage $K(d)$.
		
		Since $\mcc$ is left complete, every object is the limit of its Postnikov tower. We can therefore form a bitower and compute the limit of $Y_\bullet$ by commuting the limits:
		$$ \varprojlim_m Y_m \simeq \varprojlim_m \varprojlim_d \tau_{\leq d} Y_m \simeq \varprojlim_d \left( \varprojlim_m \tau_{\leq d} Y_m \right) $$
		Let $X = \varprojlim_m Y_m$. Because the tower $\{\tau_{\leq d} Y_m\}_m$ stabilizes, the limit over $m$ commutes with the truncation $\tau_{\leq d}$. Thus, we obtain a canonical equivalence for the truncations:
		$$ \tau_{\leq d} X \simeq \varprojlim_m \tau_{\leq d} Y_m \simeq \tau_{\leq d} Y_l \quad \text{for any } l \geq K(d). $$
		Now, for any given $n \geq 0$, let $k = K(n-1)$. For any $l \geq k$, the natural projection $X \to Y_l$ induces an equivalence on their $(n-1)$-truncations:
		$$ \tau_{\leq n-1} X \xrightarrow{\sim} \tau_{\leq n-1} Y_l $$
		This precisely means that the cofiber of the projection $\varprojlim Y_\bullet \to Y_l$ has a vanishing $(n-1)$-truncation, i.e., the cofiber is $n$-connective.
		
		(2) Let $C$ be the cofiber of the natural map $X\otimes \varprojlim Y_\bullet\to \varprojlim (X\otimes Y_\bullet)$. Since $\mcc$ is left complete, to show $C \simeq 0$ (and thus the map is an equivalence), it suffices to show that $C$ is $m$-connective for arbitrarily large $m$.
		For any index $l$, consider the following commutative diagram:
		$$\begin{tikzcd}
			X\otimes \varprojlim Y_\bullet \arrow[rr] \arrow[rd, "p_l"'] &               & \varprojlim (X\otimes  Y_\bullet) \arrow[ld, "q_l"] \\
			& X\otimes  Y_l &                                              
		\end{tikzcd}$$
		Given any integer $m \geq 0$, by (1), we can find a sufficiently large $l$ such that the cofiber of $\varprojlim Y_\bullet \to Y_l$ is $m$-connective. Since $X \in \mcc$ is always ($0$-)connective, and taking the tensor product with a connective object preserves $m$-connective objects, the cofiber of $p_l$ is $m$-connective.
		Similarly, since $X$ is connective, the tower $X \otimes Y_\bullet$ is also a highly connected pretower. Applying (1) to this new tower, for $l$ large enough, the cofiber of $q_l$ is also $m$-connective.
		The commutative triangle induces a fiber sequence of cofibers:
		$$ C \to \op{cofib}(p_l) \to \op{cofib}(q_l) $$
		Since both $\op{cofib}(p_l)$ and $\op{cofib}(q_l)$ can be made $m$-connective by choosing $l$ large enough, their fiber $C$ is at least $(m-1)$-connective. Since $m$ was arbitrary, $C$ is $m$-connective for all $m$. By left completeness (hypercompleteness suffices here), we conclude $C \simeq 0$, hence the arrow is an equivalence. 
		
		The argument for $\varprojlim Y_\bullet\otimes X\to \varprojlim ( Y_\bullet \otimes X)$ is similar.
	\end{proof}

	\subsection{Algebra theory in Grothendieck abelian categories}\label{groab}
	
	A standard and powerful technique for studying derived algebraic geometry is to reduce structural questions to their discrete components via the $\pi_0$ functor. In this subsection, we establish the 1-categorical foundations of algebra within a symmetric monoidal Grothendieck abelian category $\mb{A}^\otimes$.
	
	We introduce the notion of 1-projective rigidity and provide 1-categorical analogues for flatness, finitely presented maps, and Lazard's theorem, which will serve as the discrete shadows for our higher categorical generalizations. Some of them have been developed in \cite{toen2009dessous}, \cite{banerjee2012centre} and \cite{banerjee2017noetherian}.
	
	
	\begin{de}
		Let $\op{Groth}_{1} \subset \prladone$ denote the full subcategory of presentable additive 1-categories spanned by those abelian categories such that filtered colimits commute with finite limits in it. We call an object in $\op{Groth}_{1}$ a Grothendieck abelian category.
	\end{de}
	\begin{cov}
		Throughout \cref{groab}, we fix a symmetric monoidal Grothendieck abelian category $\mb{A}^\otimes\in \calg(\op{Groth}_{1})$.
	\end{cov}
	
	\begin{rem}
		For a Grothendieck prestable \infcat $\mc{C}\in \op{Groth}_{\infty}$, the heart $\mc{C^{\heartsuit}}$ is a Grothendieck Abelian category. 
	\end{rem}

	\begin{rem}
		$\groo$ is closed under the Lurie tensor product on $\prl$ by \cite[Theorem C.5.4.16]{sag}. We call an object in  $\calg(\groo)$  a symmetric monoidal Grothendieck abelian category. 
	\end{rem}

	\begin{de} \label{f.p.}
		Let $R \in \alg(\mb{A})$.
		\enu{\item We say a left $R$-module $M$ is finitely presented if $M$ is a compact object in $\modu_R(\mb{A})$.
			\item We say a left $R$-module $M$ is (faithfully) 1-flat if the relative tensor product functor $(-)\otimes_R M: \rmodu_R(\mb{A})\to \mb{A}$ is (conservative) left exact.
			\item We define a left ideal $I$ of $R$ as a left $R$-submodule of $R$.
			
		}
	\end{de}
	
	\begin{prop}\label{cokerff}
		Let $f: A\to B\in \calg(\mb{A})$ be a faithfully 1-flat map. Then 
		\enu{
			\item For any $A$-module $M$, the map $M\simeq M\otimes_{A}A\to M\otimes_{A}B$ is monomorphic. In particular, $A\to B$ is monomorphic.
			\item $\op{Coker}(f)$ is a 1-flat $A$-module.
		}
	\end{prop}
	\begin{proof}
		(1) Let $N$ be the kernel of $M\to M\otimes_A B$. Considering the following diagram.
		$$
		\begin{tikzcd}
			N \arrow[d, "i", hook] \arrow[r] & N\otimes_{A}B \arrow[d, "i\otimes_A B"] \\
			M \arrow[r]                      & M\otimes_{A}B                          
		\end{tikzcd}$$
		Then by the base change adjoint we conclude that the map $N\otimes_{A}B\to M\otimes_{A}B$ is zero. The faithful 1-flatness implies $i=0$ and hence $N=0$.\\
		(2) It follows immediately from (1) and snake lemma. 
	\end{proof}
	The converse of \cref{cokerff} does not hold for an arbitrary
	symmetric monoidal Grothendieck abelian category. Namely if $f: A\to B\in \calg(\mb{A})$ is a monomorphism such that $\coker(f)$ is flat, $f$ need not be faithfully flat\footnote{However, this holds when $\mb{A}$ admits a symmetric monoidal prestable enhancement; see \cref{pi0flat}.}. 
	
	\begin{ex}\label{cokerflatnotff}
		
		Indeed, let
		\[
		R=\mathbb{Q}[t],\qquad K=\mathbb{Q}(t),
		\]
		and let $\mathbf{T}\subset\operatorname{Mod}_R(\operatorname{Ab})$
		be the Grothendieck abelian category of torsion $R$-modules (i.e. those $R$-modules $M$ such that $M\otimes_R K=0$). Consider
		the Grothendieck abelian category
		\[
		\mathbf{A}=\operatorname{Ab}\times\mathbf{T}
		\]
		equipped with the (presentably) symmetric monoidal structure
		\[
		(P,M)\otimes(Q,N)
		=
		\left(
		P\otimes_{\mathbb Z}Q,\,
		(P\otimes_{\mathbb Z}N)
		\oplus
		(Q\otimes_{\mathbb Z}M)
		\oplus
		(M\otimes_RN)
		\right),
		\]
		whose tensor unit is $\mathbf{1}=(\mathbb Z,0)$.
		This is the unital version of a non-tame tensor category due to
		Goodwillie.\footnote{See
			\url{https://mathoverflow.net/questions/57651/tame-abelian-tensor-categories}.}
		
		Set
		\[
		X=(0,R/tR),\qquad
		Y=(0,K/tR),\qquad
		F=(0,K/R).
		\]
		There is a short exact sequence
		\[
		0\longrightarrow X\longrightarrow Y\longrightarrow F\longrightarrow0.
		\]
		The object $F$ is flat in $\mathbf A$. Indeed, for every
		$(P,M)\in\mathbf A$ we have
		\[
		(P,M)\otimes F
		\simeq
		\left(0,P\otimes_{\mathbb Z}(K/R)\right),
		\]
		since $M\otimes_R(K/R)=0$ for every torsion $R$-module $M$.
		Moreover, $K/R$ is a $\mathbb Q$-vector space and hence is flat over
		$\mathbb Z$.
		
		Now form the square-zero commutative algebras
		\[
		A=\mathbf{1}\oplus X,
		\qquad
		B=\mathbf{1}\oplus Y.
		\]
		The inclusion $X\hookrightarrow Y$ induces a monomorphism
		$
		f:A\longrightarrow B
		$
		in $\operatorname{CAlg}(\mathbf A)$, with
		$
		\operatorname{Coker}(f)\simeq F.
		$
		We regard $F$ as an $A$-module via the augmentation
		$A\to\mathbf{1}$. Since
		$
		X\otimes F=0,
		$
		for every right $A$-module $M$ the canonical map induces an
		equivalence
		\[
		M\otimes_A F\simeq M\otimes F,
		\]
		where on the right-hand side we forget the $A$-module structure.
		Consequently, $F$ is a flat $A$-module.
		
		However, $f$ is not faithfully flat. To see this, consider the free
		$A$-module
		$
		P=A\otimes X.
		$
		The unit map for base change along $f$ identifies with
		\[
		A\otimes X\longrightarrow B\otimes X.
		\]
		We have
		\[
		A\otimes X
		\simeq
		X\oplus(X\otimes X),
		\qquad
		B\otimes X
		\simeq
		X\oplus(Y\otimes X).
		\]
		But
		\[
		X\otimes X\simeq X\neq0,
		\qquad
		Y\otimes X
		\simeq
		\left(0,(K/tR)\otimes_RR/tR\right)
		\simeq0,
		\]
		since multiplication by $t$ on $K/tR$ is surjective. Hence
		$
		P\longrightarrow P\otimes_A B
		$
		has the nonzero kernel $X\otimes X$.
		
		If $B$ were faithfully flat over $A$, then
		$M\to M\otimes_A B$ would be monomorphic for every $A$-module $M$:
		after tensoring with $B$, this map becomes split by the multiplication
		$B\otimes_A B\to B$, and exactness and conservativity of
		$(-)\otimes_A B$ would then imply that the original map is
		monomorphic. This gives a contradiction.
		
		Thus a monomorphism of commutative algebras may have flat cokernel
		without being faithfully flat.
	\end{ex}
	\begin{de}Let $\mb{B}\in\op{Groth}_{1}$ be a Grothendieck abelian category.
		\enu{
			\item We say an object $X\in\mb{B}$ is 1-projective if it is a projective object in the ordinary sense of an abelian category.
			\item We say $\mb{B}$ is 1-projectively generated if $\mb{B}$ is generated by a small set of compact 1-projective objects under small colimits.
			\item If $\mb{B}\in\calg(\op{Groth}_{1})$ is a symmetric monoidal Grothendieck abelian category, then we say $\mb{B}^\otimes$ is 1-projectively rigid if $\mb{B}$ is 1-projectively generated and the dualizable objects in $\mb{B}$ coincide with compact 1-projective objects.
		}
		
	\end{de}
	\begin{rem}
		We use the terminology ``1-projective'' to distinguish it from the ``projective'' in the sense of \textup{\cite[Definition 5.5.8.18]{htt}} for an $\infty$-category. Generally they do not agree in an abelian 1-category  \textup{ \cite[see][Example 5.5.8.21]{htt}}.
	\end{rem}
	\begin{prop}
		Suppose that $\mb{A}$ is 1-projectively generated and $R \in \calg(\mb{A})$. Then:
		\begingroup
		\catcode`\&=13
		\enu{\item Any $R$-module $M$ can be written as a pushout in $\modu_R(\mb{A})$ as the following form
			
			$$\begin{tikzcd}
				P_1 \arrow[r] \arrow[d] & 0 \arrow[d] \\
				P_0 \arrow[r]                & M         
			\end{tikzcd}$$ where $P_1,P_0 \in \modu_R(\mb{A})$ are 1-projective $R$-modules.
			
			\item An $R$-module $M$ is finitely presented if and only if $P_1,P_0$ can be promoted to compact 1-projective $R$-modules.
		}\endgroup
	\end{prop}

	\begin{proof}
		Let $\mathcal{C}=\modu_R(\mb{A})$.
		
		(1) Since $\mathcal{C}$ has enough 1-projective objects, for every
		$R$-module $M$ there exists an epimorphism
		$
		P_0\twoheadrightarrow M
		$
		with $P_0$ 1-projective. Let
		$
		K=\ker(P_0\to M).
		$
		Again using the existence of enough 1-projectives, we may choose an
		epimorphism $P_1\twoheadrightarrow K$ with $P_1$ 1-projective.
		Composing with the inclusion $K\hookrightarrow P_0$ gives a right exact
		sequence
		\[
		P_1\longrightarrow P_0\longrightarrow M\longrightarrow 0,
		\]
		or equivalently a pushout square
		\[
		\begin{tikzcd}
			P_1 \arrow[r] \arrow[d] & 0 \arrow[d] \\
			P_0 \arrow[r]           & M .
		\end{tikzcd}
		\]
		
		(2) Suppose first that $P_0$ and $P_1$ can be chosen to be compact
		1-projective. Since compact objects are closed under finite colimits,
		the above pushout square implies that $M$ is compact, hence finitely
		presented.
		
		Conversely, suppose that $M$ is finitely presented, i.e.\ compact in
		$\mathcal{C}$. Since $\mathcal{C}$ is 1-projectively generated, there
		exists an epimorphism
		\[
		p:\bigoplus_{i\in I}G_i\twoheadrightarrow M,
		\]
		where each $G_i$ is compact 1-projective. For each finite subset
		$F\subset I$, let
		\[
		M_F=\operatorname{Im}\left(
		\bigoplus_{i\in F}G_i\longrightarrow M
		\right).
		\]
		The objects $M_F$ form a filtered system of subobjects of $M$, and
		\[
		\colim_{F\subset I,\; F\text{ finite}}M_F\simeq M.
		\]
		Since $M$ is compact, the identity of $M$ factors through $M_F$ for
		some finite subset $F\subset I$. Thus the inclusion $M_F\hookrightarrow
		M$ admits a section. Since it is also a monomorphism, it is an
		isomorphism. Consequently,
		\[
		P_0:=\bigoplus_{i\in F}G_i\twoheadrightarrow M
		\]
		is an epimorphism, and $P_0$ is compact 1-projective.
		
		Let
		$
		K=\ker(P_0\to M).
		$
		Since $\mathcal{C}$ is generated by compact 1-projective objects, we
		may write $K$ as a filtered union
		$
		K=\colim_{\alpha}Q_\alpha
		$
		of finitely generated submodules, where each $Q_\alpha$ admits an
		epimorphism from a compact 1-projective object. Set
		\[
		M_\alpha=P_0/Q_\alpha.
		\]
		Since filtered colimits are exact in the Grothendieck abelian category
		$\mathcal{C}$, we have
		\[
		\colim_\alpha M_\alpha
		\simeq
		P_0/\left(\colim_\alpha Q_\alpha\right)
		\simeq P_0/K
		\simeq M.
		\]
		By compactness of $M$, the identity of $M$ factors through some
		$M_k$:
		\[
		M\xrightarrow{s}M_k\xrightarrow{q_k}M,
		\qquad
		q_k\circ s=\operatorname{id}_M.
		\]
		Hence the short exact sequence
		\[
		0\longrightarrow K/Q_k
		\longrightarrow M_k
		\xrightarrow{q_k}M
		\longrightarrow0
		\]
		splits.
		Choose an epimorphism $P'_1\twoheadrightarrow Q_k$ with $P'_1$
		compact 1-projective. Then
		\[
		P'_1\longrightarrow P_0\longrightarrow M_k\longrightarrow0
		\]
		is right exact. Since $P'_1$ and $P_0$ are compact and compact objects
		are closed under finite colimits, $M_k$ is compact. It follows from the
		splitting above that $K/Q_k$, being a retract of $M_k$, is also
		compact.
		
		After passing to the cofinal subsystem of indices $\alpha\geq k$, we
		have
		\[
		K/Q_k
		\simeq
		\colim_{\alpha\geq k}Q_\alpha/Q_k.
		\]
		Applying compactness once more, the identity of $K/Q_k$ factors through
		some $Q_\ell/Q_k$:
		\[
		K/Q_k\longrightarrow Q_\ell/Q_k
		\longrightarrow K/Q_k,
		\]
		with composite equal to the identity. The second map is a
		monomorphism, since $Q_\ell\subset K$, and it is also a split
		epimorphism; hence it is an isomorphism. Therefore
		\[
		Q_\ell/Q_k\simeq K/Q_k,
		\qquad\text{and consequently}\qquad
		Q_\ell=K.
		\]
		By construction, there is an epimorphism
		\[
		P_1\twoheadrightarrow Q_\ell=K
		\]
		with $P_1$ compact 1-projective. We therefore obtain the desired right
		exact sequence
		\[
		P_1\longrightarrow P_0\longrightarrow M\longrightarrow0,
		\]
		with both $P_0$ and $P_1$ compact 1-projective.
	\end{proof}

	\begin{prop}Suppose that $\mb{A}^\otimes$ is 1-projectively rigid. Let $R$ be in $\calg(\mc{A}_{\geq 0})$. Then $\modu_R(\mb{A})^\otimes$ is 1-projectively rigid too.
	\end{prop}
	\begin{proof}
		Since the symmetric monoidal functor $$\mb{A}^\otimes\xrightarrow{R\otimes(-)} \modu_R(\mb{A})^\otimes$$
		preserves compact 1-projective objects and dualizable objects, we conclude that 
		\begin{enumerate}
			\item 	The unit $R$ is dualizable in $\modu_R(\mb{A})$.
			\item 	If $P\in\mb{A}$ is compact 1-projective, then $R\otimes P$ is dualizable in $\modu_R(\mb{A})$ .	
		\end{enumerate}
		So the full subcategory of dualizable objects $\modu_R(\mb{A})^d$ contains $\{R\otimes X|X\in \mb{A}^{1\text{-}\op{cproj}}\}$. Then combining \cref{l1}(2) and \cref{du}(2)(3), we get $\modu_R(\mb{A})^{1\text{-}\op{cproj}}\subset \modu_R(\mb{A})^d$. Finally, the fact that the unit $R$ is compact 1-projective implies the equality $\modu_R(\mb{A})^{1\text{-}\op{cproj}}= \modu_R(\mb{A})^d$.
	\end{proof}

	Lazard's theorem over  \ein-algebras appeared in \textup{\cite[Prop. 2.2.22]{stefanich2023classification}}. We find the argument there also works in the non-commutative setting.
	
	\begin{thm}[Lazard's theorem]\label{lazard1}
		Suppose that $\mb{A}^\otimes$ is 1-projectively rigid. Let $R\in \alg(\mb{A})$ and $M$ be a left $R$-module of $\mb{A}$. Then:
		\enu{ \item If $M$ is 1-projective, then $M$ is 1-flat.
			\item $M$ is compact 1-projective if and only if it is left dualizable in $\lmodu_{R}(\mb{A})$.
			\item $M$ is 1-flat if and only if it is a filtered colimit of compact 1-projective left $R$-modules.
		} 
	\end{thm}
	
	\begin{proof}
		(1) Since 1-flat modules are closed under small coproducts and retractions, we reduce to the case $M=R\otimes P$ where $P\in\mb{A}^{1\text{-}\op{cproj}}$ is compact 1-projective. It becomes easy because $(-)\otimes_R(R\otimes P)\simeq (-)\otimes P$ reduces to the case $R=\mb{1}$, which follows from the dualizability of $P$ in $\mb{A}$.\\
		(2) By \cref{rigidlygene}, we see that left dualizable objects are closed under finite coproducts and retracts. We observe that every $R\otimes P$ is left dualizable (given by $P^{\vee}\otimes R$), which proves ``only if'' direction. For the ``if'' direction, if $M$ is left dualizable, then it follows from $$\mapp_{\lmodu_R(\mb{A})}(M,-)\simeq \mapp_{\mb{A}}(\mb{1},\,  ^{\vee}M\otimes_R -)$$ and compact 1-projectivity of the unit.
		\\(3) It is a parallel argument with \cref{cocompfil}.
	\end{proof}

	\begin{cor}\label{f.p.flat}
		Suppose that $\mb{A}^\otimes$ is 1-projectively rigid. Let $R\in \alg(\mb{A})$ and let $M$ be a left $R$-module. Then the following are equivalent:
		\enu{
			\item  $M$ is a compact 1-projective left $R$-module.
			\item  $M$ is a finitely presented and 1-flat left $R$-module.
		}
	\end{cor}
	\begin{proof}
		The direction $(1)\Rightarrow (2)$ is obvious. For $(2)\Rightarrow (1)$, by \cref{lazard1} $M$ can be written as the filtered colimit of a set of compact 1-projective left $R$-modules. Then the compactness implies $M$ is the retraction of some compact 1-projective module, and hence compact 1-projective too.
	\end{proof}

	\begin{de}\label{ringfp}
		We say  a map $R\to S \in \calg(\mb{A})$ is of finite presentation if $S$ is a compact object in $\calg(\mb{A})_{R/}$.
	\end{de}
	
	\begin{prop}
		Suppose that $\mb{A}$ is compactly generated. Then a map $R\to S \in \calg(\mb{A})$ is of finite presentation if and only if $S$ can be written as a pushout in $\calg(\mb{A})$ as the following form
		$$\begin{tikzcd}
			\op{Sym}^*_R(N) \arrow[r, "\alpha"] \arrow[d] & R \arrow[d] \\
			\op{Sym}^*_R(M) \arrow[r]                & S          
		\end{tikzcd}$$ where $M,N \in \modu_R(\mb{A})^\omega$ are compact $R$-modules and $\alpha$ is the natural augmentation. (Note that  $\phi$ here is not necessarily induced by a map of $N\to M$.)
		Furthermore, if $\mb{A}$ is 1-projectively generated, then $M,N$ can be chosen as compact 1-projective $R$-modules.
	\end{prop}
	\begin{proof}
		The proof is parallel with the proof of \cref{appro}.
	\end{proof}
	
	\begin{prop}
		Let $R\to S \in\calg(\mb{A})$ be an epimorphism of commutative algebras. Then the forgetful functor $\modu_S(\mb{A})\to\modu_R(\mb{A})$ is fully faithful.
	\end{prop}
	
	\begin{proof}
		Since $R\to S$ is an epimorphism is equivalent to that $S$ is an idempotent commutative $R$-algebra, the result follows immediately from \cite[Prop. 4.8.2.10]{ha}.
	\end{proof}
	\begin{prop}
		A faithfully 1-flat epimorphism in $\calg(\mb{A})$ is an isomorphism.
	\end{prop}
	\begin{proof}
		The epimorphism implies the map $R\otimes_R S\to S\otimes_R S$ is an isomorphism, so by the fully faithful 1-flatness,  $R\to S$ is an isomorphism.
	\end{proof}

	\section{Flatness and faithful flatness}
	The main goal of this section is to study the flatness in the setting of $ttt$-\infcats.
	We are inspired by equivalent conditions of the flatness over structured ring spectra appeared in \cite[Theorem 7.2.2.15]{ha}.
	\subsection{$t$-structure on the category of modules}
	Given a $ttt$-$\infty$-category $\mathcal{A}$ and a connective algebra $R \in \alg(\mathcal{A}_{\geq 0})$, the $\infty$-category of left $R$-modules naturally inherits an associated $t$-structure. In this subsection, we demonstrate that $\lmodu_R(\mathcal{A})$ inherits the completeness and projective generation properties of the base category $\mathcal{A}$. We establish the precise compatibilities between the symmetric monoidal structure, the module structure, and Postnikov truncations, setting the stage for robust homological algebra over $R$.
	\begin{lem}\label{l1}
		Let $\mathcal{C}\underset{G}{\stackrel{F}{\rightleftarrows}}\mathcal{D}$ be an adjoint pair of $\infty$-categories.
		\enu{
			\item Assume $\kappa$ is a regular cardinal, $\mathcal{C}$ is $\kappa$-presentable, and $\mathcal{D}$ is locally small and admits small colimits. If $G$ is conservative and preserves small $\kappa$-filtered colimits, then $\mathcal{D}$ is $\kappa$-presentable. Furthermore, $\mathcal{D}^\kappa$ is the smallest full subcategory generated by $F(\mathcal{C}^\kappa)$ under $\kappa$-small colimits and retractions. Consequently, $\mathcal{D}$ is generated by the image of $F$ under small colimits.
			\item \textup{(See \cite[Corollary 4.7.3.18]{ha}).}\\ Assume $\mathcal{D}$ admits small filtered colimits and geometric realizations, and $G$ preserves both. Also, assume $\mathcal{C}$ is projectively generated\footnote{That means $\mcc\simeq\mcp_\Sigma(\mcc^{\op{cproj}})$; see \cite[Definition 5.5.8.23]{htt}.} If the functor $G$ is conservative, then $\mathcal{D}$ is projectively generated. An object $D \in \mathcal{D}$ is compact and projective if and only if there exists a compact projective object $C \in \mathcal{C}$ such that $D$ is a retract of $F(C)$. Hence, $\mathcal{D}$ is generated by the image of $F$ under small colimits.
		}
	\end{lem}
	
	\begin{proof}
		We first prove (1). Let $\mathcal{D}_0\subset \mathcal{D}$ be the smallest full subcategory generated by $F(\mathcal{C}^\kappa)$ under finite colimits and retractions. Then the inclusion $\mathcal{D}_0\subset \mathcal{D}$ extends to a fully faithful embedding $F_2:\op{Ind}_\kappa(\mathcal{D}_0)\hookrightarrow \mathcal{D}$ (by \cite[5.3.5.10]{htt}). Since $F$ preserves small colimits, it admits a right adjoint $H$ (\cite[Proposition 5.5.1.9]{htt}). Thus, we have the following factorization of adjoint pairs:
		$$
		\mathcal{C}\underset{G_1}{\stackrel{F_1}{\rightleftarrows}}\op{Ind}_\kappa(\mathcal{D}_0)\underset{G_2}{\stackrel{F_2}{\rightleftarrows}}\mathcal{D}
		$$
		It will therefore suffice to show that the functor $G_2$ is conservative. Let $\alpha: X \rightarrow Y$ be a morphism in $\mathcal{D}$ such that $G_2(\alpha)$ is an equivalence. We aim to show that $\alpha$ is an equivalence. For this, since $\mathcal{C}$ is $\kappa$-compactly generated, it will suffice to show that $\alpha$ induces a homotopy equivalence
		$$
		\theta: \operatorname{Map}_{\mathcal{C}}\left(C, G(X)\right) \rightarrow \operatorname{Map}_{\mathcal{C}}\left(C, G(Y)\right)
		$$
		for every $\kappa$-compact object $C \in \mathcal{C}$. This map can be identified with
		$$
		\theta: \operatorname{Map}_{\op{Ind}_\kappa(\mathcal{D}_0)}\left(F_1(C), G_2(X)\right) \rightarrow \operatorname{Map}_{\op{Ind}_\kappa(\mathcal{D}_0)}\left(F_1(C), G_2(Y)\right)
		$$
		Our assumption that $G_2(\alpha)$ is an equivalence guarantees that $\theta$ is a homotopy equivalence, as desired.
		
		For (2), the argument is entirely parallel. See also \cite[Cor. 4.7.3.18]{ha}.
	\end{proof}
	
	\begin{rem}
		The condition in (2) guarantees that $\mathcal{D}$ is locally small because it is monadic over $\mathcal{C}$ by the Barr-Beck-Lurie theorem (see \textup{\cite[Thm. 4.7.3.5]{ha}}).
	\end{rem}
	
	\begin{cor}\label{l1.2}
		Let $R\in \alg(\mc{A})$. Applying \cref{l1} to the adjoint pair $\mc{A} \underset{}{\stackrel{R\otimes-}{\rightleftarrows}} \lmodu_R(\mc{A})$, we obtain:
		\enu{
			\item If $\mc{A}$ is $\kappa$-presentable for some regular cardinal $\kappa$, then so is $\lmodu_R(\mc{A})$.
			\item $\lmodu_R(\mc{A})$ is presentable.
			\item $\lmodu_R(\mc{A})$ is generated by $\{R\otimes X \mid X\in \mc{A}\}$ under small colimits.
		}
	\end{cor}
	
	\begin{rem}
		By \cite[Proposition 7.1.1.4]{ha}, $\lmodu_R(\mc{A})$ is stable for any $R\in \alg(\mc{A})$.
	\end{rem}
	
	\begin{de}
		Let $\mc{C}$ be a stable $\infty$-category equipped with a $t$-structure. We say that it is hypercomplete, if for an object $X\in \mc{C}$, the condition $\tau_{\leq n}X=0$ for every integer $n$ implies $X=0$.
	\end{de}
	
	\begin{ex}
		Let $\mc{X}$ be a hypercomplete $\infty$-topos. Then the natural $t$-structure $$(\op{Shv}(\mc{X},\op{Sp})^\otimes,\op{Shv}(\mc{X},\op{Sp})_{\geq0})$$ developed in \cite[Proposition 1.3.2.7]{sag} is hypercomplete by \cite[Proposition 1.3.3.3]{sag}.
	\end{ex}
	
	\begin{prop}\label{l2}
		Let $R$ be in $\alg(\mc{A}_{\geq 0})$. Then $\lmodu_R(\mc{A})$ is a presentable stable $\infty$-category which admits a natural accessible $t$-structure $(\lmodu_R(\mc{A})_{\geq 0},\lmodu_R(\mc{A})_{\leq 0})$ satisfying the following properties\footnote{
			See also a similar discussion in \textup{\cite[Appendix]{antieau2021cartier}}.}:
		\enu{
			\item $\lmodu_R(\mc{A})_{\geq 0}$ and $\lmodu_R(\mc{A})_{\leq 0}$ are the inverse images of $\mc{A}_{\geq 0}$ and $\mc{A}_{\leq 0}$ under the projection $\theta: \lmodu_R(\mc{A})\to \mc{A}$.
			\item The natural inclusion $\lmodu_R(\mc{A}_{\geq 0})\hookrightarrow \lmodu_R(\mc{A})$ induces an equivalence $\lmodu_R(\mc{A}_{\geq 0})\xrightarrow{\sim}\lmodu_R(\mc{A})_{\geq 0}$.
			\item The functor $\tau_{ \leq n}:\mc{A}_{\geq 0}\to \mc{A}_{[0,n]}$ induces an equivalence $\lmodu_R(\mc{A})_{[0,n]} \xrightarrow{\sim}\lmodu_{\tau_{ \leq n} R}(\mc{A}_{[0,n]})$. In particular, the $\pi_0$ functor induces an equivalence $\lmodu_R(\mc{A})^{\heartsuit} \xrightarrow{\sim}\lmodu_{\pi_0R}(\mc{A}^{\heartsuit})$.
			\item If $\mc{A}$ is left (resp. right, resp. hyper) complete, then so is $\lmodu_R(\mc{A})$.
			\item If the $t$-structure on $\mc{A}$ is compatible with filtered colimits, meaning $\mc{A}_{\leq0}\subset \mc{A}$ is closed under filtered colimits, then so is the induced $t$-structure on $\lmodu_R(\mc{A})$.
			\item If $\mc{A}_{\geq 0}$ is projectively generated, then so is $\lmodu_R(\mc{A})_{\geq 0}$.
		}
	\end{prop}
	
	\begin{proof}
		We first prove (1). It follows immediately from the definitions that the full subcategory $\lmodu_R(\mc{A})_{\geq 0}\subset \lmodu_R(\mc{A})$ is closed under small colimits and extensions. Also, note that $\lmodu_R(\mc{A})_{\geq 0}$ is presentable since the following is a pullback square in $Pr^L$:
		$$\begin{tikzcd}
			\lmodu_R(\mc{A})_{\geq 0} \arrow[r] \arrow[d] \arrow[rd, "\ulcorner", phantom, very near start] & \lmodu_R(\mc{A}) \arrow[d, "\theta"] \\
			\mc{A}_{\geq 0} \arrow[r] & \mc{A}
		\end{tikzcd}$$
		Using \cite[Prop. 1.4.4.11]{ha}, we deduce the existence of an accessible $t$-structure $$(\lmodu_R(\mc{A})_{\geq 0},\lmodu_R(\mc{A})')$$ on $\lmodu_R(\mc{A})$. To complete the proof, it will suffice to show that $\lmodu_R(\mc{A})^{\prime}=\lmodu_R(\mc{A})_{\leq 0}$.
		
		Suppose first that $N \in \lmodu_R(\mc{A})^{\prime}$. Then the mapping space $\operatorname{Map}_{\lmodu_R(\mc{A})}(M, N)$ is discrete for every object $M \in \lmodu_R(\mc{A})_{\geq 0}$. In particular, for every connective object $X \in \mc{A}_{\geq 0}$, the mapping space $\operatorname{Map}_{\lmodu_R(\mc{A})}(R \otimes X, N) \simeq \operatorname{Map}_{\mc{A}}(X, \theta(N))$ is discrete, so that $\theta(N) \in \mc{A}_{\leq 0}$, and therefore $N \in \lmodu_R(\mc{A})_{\leq 0}$.
		
		Conversely, suppose that $N \in \lmodu_R(\mc{A})_{\leq 0}$. We wish to prove that $N \in \lmodu_R(\mc{A})^{\prime}$. Let $\mathcal{C}$ denote the full subcategory of $\lmodu_R(\mc{A})$ spanned by those objects $M \in \lmodu_R(\mc{A})$ for which the mapping space $\operatorname{Map}_{\lmodu_R(\mc{A})}(M, N)$ is discrete. We wish to prove that $\mathcal{C}$ contains $\lmodu_R(\mc{A})_{\geq 0}$. Firstly, we have that $\theta$ induces a functor $\lmodu_R(\mc{A})_{\geq 0} \rightarrow \mc{A}_{\geq 0}$ which is conservative and preserves small colimits; moreover, this functor has a left adjoint $Fr$, given informally by the formula $Fr(X) \simeq R \otimes X$. Using  \cref{l1}, we conclude that $\lmodu_R(\mc{A})_{\geq 0}$ is generated under small colimits by the essential image of $Fr$. Since $\mathcal{C}$ is stable under colimits, it will suffice to show that $\mathcal{C}$ contains the essential image of $Fr$. Unwinding the definitions, we are reduced to proving that the mapping space
		$$
		\operatorname{Map}_{\lmodu_R(\mc{A})}(F(X), N) \simeq \operatorname{Map}_{\mc{A}}(X, \theta(N))
		$$
		is discrete for every connective object $X$ in $\mathcal{A}_{\geq 0}$, which is equivalent to our assumption that $N \in \lmodu_R(\mc{A})_{\leq 0}$. This completes the proof of (1).
		
		For (2), the proof follows directly from the definition.
		
		For (3), we observe that we have a natural factorization:
		$$
		\begin{tikzcd}
			& \lmodu_R(\mc{A})_{[0,n]} \arrow[rd, "F_0", dashed] \\
			\lmodu_R(\mc{A}_{\geq 0}) \arrow[rr, "F"] \arrow[ru] & & \lmodu_{\tau_{ \leq n}R}(\mc{A}_{[0,n]})
		\end{tikzcd}
		$$
		It suffices to prove that $F_0$ is fully faithful and essentially surjective. It is easy to see that $F$ and $F_0$ preserve colimits. We wish to prove that, for a fixed $N \in \lmodu_R(\mc{A})_{[0,n]}$, the full subcategory $\mathcal{D}$ of $\lmodu_R(\mc{A})_{\geq 0}$ spanned by those objects $M$ for which the map $$ \operatorname{Map}_{\lmodu_R(\mc{A})}(M, N)\to \operatorname{Map}_{\lmodu_{\tau_{\leq n}R}(\mc{A}_{[0,n]})}(F( M), F(N))$$ is an equivalence. It is easy to see that $\mathcal{D}$ is stable under colimits and contains $R\otimes X$ for all $X\in \mc{A}_{\geq 0}$. \cref{l1} shows that $\mathcal{D}=\lmodu_R(\mc{A})_{\geq 0}$. In particular, $F_0$ is fully faithful.
		
		It remains to show that $F_0$ is essentially surjective. Since $F_0$ is fully faithful and preserves small colimits, the essential image of $F_0$ is closed under small colimits. By applying Lemma \cref{l1} to $\mc{A}_{[0,n]}\rightleftarrows \lmodu_{\tau_{\leq n}R}(\mc{A}_{[0,n]})$, it will therefore suffice to show that every free left $\tau_{\leq n}R$-module $\tau_{\leq n}R\overline{\otimes} Y$ where $Y\in \mc{A}_{[0,n]}$ belongs to the essential image of $F_0$, where $\overline{\otimes}$ denotes the tensor product in $\mc{A}_{[0,n]}$. We now conclude by observing that $F\left(R\otimes X\right) \simeq \tau_{\leq n}R\overline{\otimes} \tau_{\leq n}X$.
		
		(4) and (5) are concluded by the fact that $\theta: \lmodu_R(\mc{A})\to \mc{A}$ is $t$-exact, conservative, and preserves small colimits and limits.
		
		(6) is concluded by \cref{l1}(2) applied to the adjoint pair $\mc{A}_{\geq 0}\underset{G}{\stackrel{F}{\rightleftarrows}}\lmodu_R(\mc{A})_{\geq 0}$.
	\end{proof}
	\begin{cor}
		By \cref{l2}, we see that if the presentably $t$-structured stable $\infty$-category  $(\mathcal{A},\mathcal{A}_{\geq0})$ is Grothendieck, then so is $(\modu_R(\mathcal{A}),\modu_R(\mathcal{A})_{\geq0})$ for any $R\in \calg(\mathcal{A}_{\geq0})$.
	\end{cor}

	\begin{cor}\label{l3}
		Let $R\in \alg_{\mbb{E}_{k+1}}(\mc{A}_{\geq 0})$ be a connective $\mbb{E}_{k+1}$-algebra where $1\leq k\leq \infty$. Then the presentably $\mbb{E}_k$-monoidal category $\lmodu_R(\mc{A})^\otimes\to \mbb{E}_k^\otimes$ satisfies:
		\enu{
			\item The natural $t$-structure $(\lmodu_R(\mc{A})_{\geq 0},\lmodu_R(\mc{A})_{\leq 0})$ is compatible with the monoidal structure.
			\item The natural inclusion $\lmodu_R(\mc{A}_{\geq 0})\hookrightarrow \lmodu_R(\mc{A})$ is an $\mbb{E}_k$-monoidal functor which induces an equivalence $\lmodu_R(\mc{A}_{\geq 0})^\otimes\xrightarrow{\sim}\lmodu_R(\mc{A})_{\geq 0}^\otimes$ of $\mbb{E}_k$-monoidal categories.
			\item The symmetric monoidal functor $\tau_{ \leq n}^\otimes:\mc{A}_{\geq 0}^\otimes\to \mc{A}_{[0,n]}^\otimes$ induces an equivalence $\lmodu_R(\mc{A})_{[0,n]}^\otimes \xrightarrow{\sim}\lmodu_{\tau_{ \leq n} R}(\mc{A}_{[0,n]})^\otimes$ of $\mbb{E}_k$-monoidal $(n+1)$-categories. In particular, the $\pi_0$ functor induces an equivalence of $\mbb{E}_k$-monoidal 1-categories $\lmodu_R(\mc{A})^{\heartsuit,\otimes} \xrightarrow{\sim}\lmodu_{\pi_0R}(\mc{A}^{\heartsuit})^\otimes$. Note that when $k>1$, $\mbb{E}_k$-algebras in $\mc{A}^\heartsuit$ are $\mbb{E}_\infty$-algebras, so $\lmodu_{\pi_0R}(\mc{A}^{\heartsuit})^\otimes$ is symmetric monoidal in this case.
		}
	\end{cor}
	
	\begin{rem}
		\enu{
			\item When $R\in \calg(\mc{A}_{\geq 0})$ is a connective $\mathbb{E}_\infty$-algebra, the symmetric monoidal \infcat of left modules $\lmodu_R(\mc{A})^\otimes$ with the induced $t$-structure forms a new 
			$ttt$-\infcat $(\lmodu_R(\mc{A})^\otimes,\lmodu_R(\mc{A})_{\leq 0})$. 
			\item \cref{l2} also holds for the right module category $\rmodu_R(\mc{A})$ and the bimodule category $_R\op{BMod}_S(\mc{A})$ when $R,S$ are connective.
		}
	\end{rem}
	
	\begin{cov}
		In the case where $R\in \calg(\mc{A})$ is commutative, we will simply denote $\lmodu_R(\mc{A})^\otimes$ by $\modu_R(\mc{A})^\otimes$.
	\end{cov}
	
	\subsection{Flat modules and algebras}
	Motivated by the classical equivalent conditions of flatness over structured ring spectra, we define flat modules over an algebra $R \in \alg(\mathcal{A})$ via the $t$-exactness of the relative tensor product. This subsection explores the stability of flat modules under base change, composition, etc. Crucially, we extend these definitions from the connective to the non-connective case, demonstrating that the flatness can be reliably tracked through connective covers. 
	
	We first consider the connective case.
	\begin{de}[The connective case]
		Let $R\in \alg(\mc{A}_{\geq 0})$ be a connective $\mbb{E}_1$-algebra. 
		\enu{\item We say a left $R$-module $M$ is flat if the relative tensor product functor $(-)\otimes_R M: \rmodu_R(\mc{A})\to \mc{A}$ is $t$-exact.
			\item We say a left $R$-module $M$ is faithfully flat if the relative tensor product functor $(-)\otimes_R M: \rmodu_R(\mc{A})\to \mc{A}$ is $t$-exact and conservative.
			\item If $R$ is $\mbb{E}_\infty$ and $f: R\to S$ is a morphism in $\calg(\mc{A})$, we say $f$ is (faithfully) flat if $S$ is a (faithfully) flat $R$-module.
		}
	\end{de}
	\begin{rem}
		If $M$ is flat on a connective $\mbb{E}_1$-algebra $R\in \alg(\mc{A}_{\geq 0})$, then $M\simeq R\otimes_R M$ itself is connective.  \end{rem}
	
	\begin{rem}
		If  $\mc{A}$ 
		is Grothendieck, then a connective left $R$-module $M$ for some $R\in \alg(\mc{A}_{\geq0})$ is flat if and only if the tensor product functor $(-)\otimes_R M: \rmodu_R(\mc{A}_{\geq0})\to \mc{A}_{\geq0}$ is left exact by \cite[Proposition C.3.2.1]{sag}.
	\end{rem}

	\begin{prop}\label{l7}
		Assume that $\mc{A}$ is Grothendieck. Let $R\in \alg(\mc{A}_{\geq 0})$ be a connective $\mbb{E}_1$-algebra and $M$ be a connective left $R$-module. Then the following conditions are equivalent:
		\enu{
			\item $M$ is flat.
			\item The tensor product functor $(-)\otimes_R M $ is left $t$-exact (meaning that it sends the coconnective part to the coconnective part).
			\item The tensor product functor $(-)\otimes_R M $ sends discrete objects to discrete objects.	
		}
	\end{prop}
	\begin{proof}
		The direction $(1)\Leftrightarrow  (2)$ and $(2) \Rightarrow (3)$ are obvious. Now we claim that $(3) \Rightarrow (2)$.\\
		Given a coconnective right $R$-module $M\in \rmodu_R(\mc{A})_{\leq 0}$, we wish to show that $N\otimes_R M \in \mc{A}_{\leq 0}$. Since $\mc{A}$ is right complete, we have that $M\simeq \varinjlim \tau_{\geq -i}M$. Now we will prove that $N\otimes_R \tau_{\geq -n}M \in \mc{A}_{[-n,0]}$ inductively. The case $n=0$ is true by the assumption. Now assume that for $n-1\geq 0$ it is true, we need to show that $N\otimes_R \tau_{\geq -n}M \in \mc{A}_{[-n,0]}$, which is by observing that the first and third items in the following exact sequence $$N\otimes_R \tau_{\geq -(n-1)}M\to N\otimes_R \tau_{\geq -n}M \to N\otimes_R \pi_{-n}M$$
		belong to $\mc{A}_{[-n,0]}$. Since the $t$-structure is compatible with filtered colimits, we have that $N\otimes_RM\simeq \varinjlim N\otimes_R\tau_{\geq -i}M$ belong to $\mc{A}_{\leq 0}$.
	\end{proof}
	\begin{prop}\label{l1.3}
		Let $R\in \alg(\mc{A}_{\geq 0})$ be a connective $\mbb{E}_1$-algebra. Then:
		\enu{	\item The full subcategory of flat modules $\lmodu_R(\mc{A})^{fl}\subset \lmodu_R(\mc{A})$ is closed under finite coproducts, retractions, and extensions. If the $t$-structure on $\mc{A}$ is compatible with filtered colimits, then $\lmodu_R(\mc{A})^{fl}\subset \lmodu_R(\mc{A})$ is furthermore closed under filtered colimits.
			\item If $R\in \calg(\mc{A}_{\geq 0})$ is $\mbb{E}_\infty$, then the full subcategory of flat modules $\modu_R(\mc{A})^{fl}\subset \modu_R(\mc{A})$ contains the unit and is closed under tensor product, and hence forms a  symmetric monoidal full subcategory.
			
			\item If $R\in \calg(\mc{A}_{\geq 0})$ is $\mbb{E}_\infty$ and $M\in \modu_R(\mc{A})$ is a dualizable $R$-module, then $M$ is flat if and only if both $M$ and the dual $M^\vee$ are connective $R$-modules.
		}
	\end{prop}
	\begin{proof}
		(1) and (2) are obvious by definition of flatness.\\
		(3) Assume that $M$ is flat, then $M$ is connective by the remark above. Since  $$\mapp_{\modu_R(\mc{A})}(M^\vee,N)\simeq \mapp_{\modu_R(\mc{A})}(R,M\otimes_{R }N)$$ is contractible for any $(-1)$-truncated $N$,  $M^\vee$ is connective too.\\
		Now assume both $M$ and the dual $M^\vee$ are connective. Since $M$ is connective, the tensor product $(-)\otimes_R M $ is right $t$-exact. So it suffices to check the left $t$-exactness of $(-)\otimes_RM$. Let $Q$ be a connective $R$-module and $N$ is be a $(-1)$-truncated $R$-module. Then  $$\mapp_{\modu_R(\mc{A})}(Q\otimes_{R } M^\vee,N)\simeq \mapp_{\modu_R(\mc{A})}(Q,M\otimes_{R }N)$$ is contractible. So the functor $(-)\otimes_RM$ is indeed left $t$-exact.
	\end{proof}
	\begin{prop}\label{l4}
		If $R\to S\in \alg(\mc{A}_{\geq 0})$ be a morphism of connective $\mbb{E}_1$-algebras, then
		\enu{
			\item The relative tensor product $S\otimes_R(-): \lmodu_R(\mc{A})\to \lmodu_S(\mc{A})$ sends (faithfully) flat modules to (faithfully) flat modules.
			\item If $S$ is flat as a left $R$-module, then the forgetful functor $\theta: \lmodu_S(\mc{A})\to \lmodu_R(\mc{A})$ sends flat modules to flat modules. If furthermore $S$ is faithfully flat as a left $R$-module, then the forgetful functor $\theta: \lmodu_S(\mc{A})\to \lmodu_R(\mc{A})$  preserves faithfully flat modules.
			
		}
	\end{prop}
	\begin{proof}
		(1) Let $M\in \lmodu_R(\mc{A})$. We observe that $(-)\otimes_S(S\otimes_RM)\simeq (-)\otimes_RM$. 
		\\(2) Given $N\in \lmodu_S(\mc{A})$, we observe that $(-)\otimes_R \theta(N)\simeq (-)\otimes_RS\otimes_SN$.
	\end{proof}

	We now start to investigate the non-connective flatness.
	\begin{de}[The non-connective case]\label{ncflat}
		Let $R\in \alg(\mc{A})$ and $\theta: \lmodu_R(\mc{A})\to \lmodu_{\tau_{\geq 0}R}(\mc{A})$ be the forgetful functor. 
		\enu{\item We say a left $R$-module $M$ is flat if the counit map $R\otimes_{\tau_{\geq 0}R}\tau_{\geq 0}\theta(M)\to M$ with respect to the following composite adjunction 
			$$\lmodu_{\tau_{\geq 0}R}(\mc{A})_{\geq 0} \underset{\tau_{\geq 0}}{\stackrel{}{\rightleftarrows}} \lmodu_{\tau_{\geq 0}R}(\mc{A}) \underset{\theta}{\stackrel{}{\rightleftarrows}} \lmodu_R(\mc{A})$$
			is an equivalence and $\tau_{\geq 0}\theta(M)$ is flat over $\tau_{\geq 0}R$.
			\item We say a left $R$-module $M$ is faithfully flat if it is flat over $R$ and $\tau_{\geq 0}\theta(M)$ is faithfully flat over $\tau_{\geq 0}R$.
			\item If $f: R\to S$ is a morphism in $\calg(\mc{A})$, we say that $f$  is a (faithfully) flat morphism if $S$ is a (faithfully) flat $R$-module.
		}
	\end{de}
	
	\begin{rem}Let $R\in \alg(\mc{A})$ be an $\mbb{E}_1$-algebra. Then
		\enu{\item  The full subcategory of flat modules $\lmodu_R(\mc{A})^{fl}\subset \lmodu_R(\mc{A})$ is closed under finite coproducts, retractions.
			Note that, unlike the connective case, it is not closed under extensions in general. 
			\item If the $t$-structure on $\mc{A}$ is compatible with filtered colimits, then $\lmodu_R(\mc{A})^{fl}\subset \lmodu_R(\mc{A})$ is closed under filtered colimits.
			\item If $M$ is a flat left $R$-module, then $M$ is faithfully flat implies that the tensor product functor $(-)\otimes_RM$ is conservative. Note that, unlike the connective case, the converse  does not hold in general.  
		}

	\end{rem}

	\begin{ex}\label{nonconnectivefaithfulcounterexample}
		The converse in the preceding remark (3) fails for non-connective algebras,
		even in \opsp.
		Let $k$ be a field and let
		$
		\mc A=\modu_{Hk}(\opsp)
		$
		with its standard $t$-structure. Set
		\[
		P=Hk\oplus \Sigma Hk,
		\qquad
		R=\underline{\operatorname{End}}_{\mc A}(P).
		\]
		Then $R$ is non-connective; more explicitly,
		\[
		R\simeq
		\begin{pmatrix}
			Hk & \Sigma^{-1}Hk\\
			\Sigma Hk & Hk
		\end{pmatrix}.
		\]
		Let $e\in\pi_0R$ denote the idempotent corresponding to the summand
		$Hk\subset P$. Then
		$
		P\simeq Re
		$
		as a left $R$-module. Writing $R_{\geq0}=\tau_{\geq0}R$, we moreover
		have
		$
		P\simeq R_{\geq0}e
		$
		as a left $R_{\geq0}$-module. Hence $P$ is flat over $R_{\geq0}$ and
		\[
		R\otimes_{R_{\geq0}}P
		\simeq
		R\otimes_{R_{\geq0}}R_{\geq0}e
		\simeq Re
		\simeq P.
		\]
		Thus $P$ is flat over $R$ in the sense of \cref{ncflat}.
		
		On the other hand,
		\[
		\pi_0R_{\geq0}\simeq k\times k,
		\qquad
		\pi_0P\simeq(k,0),
		\]
		so $P$ is not faithfully flat over $R_{\geq0}$. Equivalently, if
		$k_2$ denotes the nonzero discrete right $R_{\geq0}$-module obtained
		from the second projection $k\times k\to k$, then
		\[
		k_2\otimes_{R_{\geq0}}P\simeq0.
		\]
		Consequently, $P$ is not faithfully flat over $R$.
		
		Nevertheless, $P$ is a compact generator of $\mc A$, since $Hk$ is a
		retract of $P$. Therefore Morita theory
		\cite[Theorem 7.1.2.1]{ha} gives an equivalence
		\[
		(-)\otimes_RP:
		\rmodu_R(\mc A)\xrightarrow{\sim}\mc A.
		\]
		In particular, $(-)\otimes_RP$ is conservative. Thus, for a
		non-connective algebra, conservativity of tensoring with a flat module
		does not imply faithful flatness.
	\end{ex}
	
	\begin{prop}\label{l5}
		Let $R\to S$ be a map in $\alg(\mc{A})$.
		\enu{
			\item   If $\tau_{\geq 0}R\to \tau_{\geq 0}S$ is an equivalence, then the relative tensor product $S\otimes_R(-): \lmodu_R(\mc{A})\to \lmodu_S(\mc{A})$ restricts to an equivalence $$\lmodu_R(\mc{A})^{fl}\xrightarrow{\sim} \lmodu_S(\mc{A})^{fl}$$ between the full subcategories of flat modules. \textup{(See \cite[Prop. 7.2.2.16]{ha} for the case of spectra.)}

			\item If $R\in \calg(\mc{A})$, then the full subcategory of flat modules $\modu_R(\mc{A})^{fl}\subset \modu_R(\mc{A})$ contains the unit and is closed under tensor product, and hence forms  a  symmetric monoidal full subcategory.
			\item If $R\to S$ is a morphism in $\calg(\mc{A})$ such that $\tau_{\geq 0}R\to \tau_{\geq 0}S$ is an equivalence, then the base change induces  an equivalence of symmetric monoidal categories $$\modu_R(\mc{A})^{fl,\otimes}\xrightarrow{\sim} \modu_S(\mc{A})^{fl,\otimes}.$$
			
		}
	\end{prop}
	
	\begin{proof}
		The (2), (3) are conclusions of (1). So it suffices to prove (1) in the case when $R=\tau_{\geq 0}S$. Since $R$ is connective, the $\infty$-category $\lmodu_R(\mc{A})$ admits a $t$-structure. Let $F^{\prime}$ denote the composite functor
		$$
		\lmodu_R(\mc{A})_{\geq 0} \subset \lmodu_R(\mc{A}) \xrightarrow{F} \lmodu_S(\mc{A})
		$$
		Then $F^{\prime}$ has a right adjoint, given by the composition $G^{\prime}=\tau_{\geq 0} \circ G$. Given $M$ is a flat left $R$-module, we observe that $M\to G'F'(M)$ is equivalent by the flatness of $M$. Now we wish to prove $F'$ preserves flatness, i.e. $F'G'F'(M)\to F'(M)$ is equivalent, which is obvious. Then we wish to prove $G'$ preserves flatness too, which is by definition of flatness in the non-connective case. Consequently, $F^{\prime}$ and $G^{\prime}$ induce adjoint functors
		
		$$
		\lmodu^{fl}_R(\mc{A}) \underset{G^{\prime }}{\stackrel{F^{ \prime}}{\leftrightarrows}} \lmodu^{fl}_S(\mc{A})
		$$
		It now suffices to show that the unit and counit of the adjunction are equivalences. In other words, we must show:\\
		(i) For every flat left $R$-module $M$, the unit map $M\to G'F'(M)$ is an equivalence, which has been done by the argument above.\\
		(ii) For every flat left $S$-module $N$, the counit map $F'G'(N) \rightarrow N$ is an equivalence, which is by definition of flatness in the non-connective case.
	\end{proof}
	\begin{prop}
		\label{l6}
		If $R\to S\in \alg(\mc{A})$ is a morphism of (non-connective) $\mbb{E}_1$-algebras, then
		\enu{
			\item The relative tensor product $S\otimes_R(-): \lmodu_R(\mc{A})\to \lmodu_S(\mc{A})$ sends (faithfully) flat modules to (faithfully) flat modules.
			\item If $S$ is flat as a left $R$-module, then the forgetful functor $\theta: \lmodu_S(\mc{A})\to \lmodu_R(\mc{A})$ sends flat modules to flat modules. If furthermore $S$ is faithfully flat as a left $R$-module, then the forgetful functor $\theta: \lmodu_S(\mc{A})\to \lmodu_R(\mc{A})$  preserves faithfully flat modules.
		}
	\end{prop}
	\begin{proof}
		The (1) is deduced by combination of \cref{l4} and \cref{l5}. \\For (2), we claim the following diagram is right adjointable,
		$$\begin{tikzcd}
			\lmodu_{\tau_{\geq 0}R}(\mc{A}) \arrow[d] \arrow[r, shift left] & \lmodu_{\tau_{\geq 0}S}(\mc{A}) \arrow[d] \arrow[l, shift left] \\
			\lmodu_R(\mc{A}) \arrow[r, shift left]                          & \lmodu_S(\mc{A}) \arrow[l, shift left]                         
		\end{tikzcd}$$
		because $S\otimes_{\tau_{\geq 0}S}(-)\simeq R\otimes_{\tau_{\geq 0}R}\tau_{\geq 0}S\otimes_{\tau_{\geq 0}S}(-)$ by flatness of $S$ over $R$. Then it reduces to the connective case, which is \cref{l4}.
	\end{proof}
	
	\begin{prop}\label{flatpush}
		\begin{enumerate}[label=(\arabic*),font=\normalfont]
			\item If $f: R\to S$ is a morphism in $\calg(\mc{A})$, then $f$ is flat if and only if $f_{\geq0}: \tau_{\geq 0}R\to \tau_{\geq 0}S$ is flat and the following diagram
			$$\begin{tikzcd}
				\tau_{\geq 0}R \arrow[d] \arrow[r] & \tau_{\geq 0}S \arrow[d] \\
				R \arrow[r]                        & S                       
			\end{tikzcd}$$ 
			is a pushout diagram in $\calg(\mc{A})$.
			\item  If $f: R\to S$ is a flat map in $\calg(\mc{A}_{\geq0})$, then the following diagram is a pushout diagram in $\calg(\mc{A}_{\geq0})$.
			$$
			\begin{tikzcd}
				R \arrow[d] \arrow[r]    & S \arrow[d]    \\
				\tau_{\leq n}R \arrow[r] & \tau_{\leq n}S
			\end{tikzcd}
			$$ Hence $\tau_{\leq n}R\to \tau_{\leq n}S$ is also flat for any $n\geq0$.
			\item Let $f: R\to S$ be a (faithful) flat map in $\calg(\mc{A})$ and $R\to A$ be another map in $\calg(\mc{A})$. Then the map $A\to A\otimes_R S$ given by the following pushout diagram is (faithful) flat.
			$$
			\begin{tikzcd}
				R \arrow[d] \arrow[r]    & S \arrow[d]    \\
				A \arrow[r] & A\otimes_R S
			\end{tikzcd}
			$$
		\end{enumerate}

	\end{prop}
	\begin{proof}
		(1) If $f$ is flat, then we have $S\simeq R\otimes_{\tau_{\geq 0}R}\tau_{\geq 0}S$ by the flatness of $S$ over $R$. If the converse is true, then $R\to S$ is flat by \cref{l6}(1).\\
		(2) Since $(-)\otimes_R S$ is $t$-exact, the following diagram is a pushout diagram in $\calg(\mc{A})$.
		$$
		\begin{tikzcd}
			R \arrow[d] \arrow[r]    & S \arrow[d]    \\
			\tau_{\leq n}R \arrow[r] & \tau_{\leq n}S
		\end{tikzcd}
		$$
		So $\tau_{\leq n}R\to \tau_{\leq n}S$ is flat by \cref{l6}(1).\\
		(3) It follows immediately from \cref{l6}.
		
	\end{proof}
	\begin{prop}\label{compoflat}
		Given a diagram  $$\begin{tikzcd}
			& A \arrow[ld, "f"] \arrow[rd, "h"'] &   \\
			B \arrow[rr, "g"] &                                    & C
		\end{tikzcd}$$
		in $\calg(\mc{A})$.
		\begin{enumerate}[label=(\arabic*),font=\normalfont]
			
			\item  
			where $f,g$ are flat morphisms, then so is the composition $h$.
			\item If $h$ is flat and $g$ is faithfully flat, then $f$ is flat.
			
		\end{enumerate}
	\end{prop}
	\begin{proof}
		(1) It follows immediately from definition.\\
		(2) Considering the following diagram 
		$$
		\begin{tikzcd}
			\tau_{ \geq 0}A \arrow[r, "\tau_{ \geq 0}f"] \arrow[d] & \tau_{ \geq 0}B \arrow[r, "\tau_{ \geq 0}g"] \arrow[d] & \tau_{ \geq 0}C \arrow[d] \\
			A \arrow[r, "f"]                                       & B \arrow[r, "g"]                                       & C                        
		\end{tikzcd}$$
		in $\calg(\mc{A})$. Then the right square and outer square are pushouts by \cref{flatpush}. The faithful flatness of $g$ implies the tensor product functor $(-)\otimes_{\taug B}\taug C$ is conservative, which implies the left square is also a pushout. So we reduce to the case when $A,B,C$ are connective.
		
		Now given a coconnective $A$-module $M\in \modu_{A}(\mc{A})_{\leq 0}$, we wish to show that $N=B\otimes_AM$ is also coconnective. However  $C\otimes_BN$ is coconnective by assumption, so $N$ is coconnective by faithful flatness of $g$.
	\end{proof}
	
	\begin{cor}\label{corB1.4.3}
		Given a pushout diagram in $\calg(\mc{A})$
		$$
		\begin{tikzcd}
			A' \arrow[d, "\phi"] \arrow[r, "\psi"] & A \arrow[d, "\phi'"] \\
			B' \arrow[r, "\psi'"]                  & B                   
		\end{tikzcd}
		$$  where $\psi$ is faithfully flat. If $B$
		is flat over $A$, then $B'$ is flat over $A'$.
	\end{cor}
	\begin{proof}
		Since $\psi$ is faithfully flat, the morphism $\psi'$ is also faithfully flat. By virtue of
		\cref{compoflat}(2), it will suffice to show that the composition $\psi' \circ\phi \simeq \phi\circ
		\psi$ is flat. This also follows from \cref{compoflat}, since $\psi$ and $\phi$ are both flat.
	\end{proof}

	We now study the relation with discrete flatness.
	\begin{prop}\label{pi0flat}
		Let $R\in \alg(\mc{A})$. Then:
		\enu{
			\item Let $M\in \lmodu_R(\mc{A})$. If $M$ is (faithfully) flat over $R$, then $\pi_0M \in \lmodu_{\pi_0R}(\mc{A}^{\heartsuit})$ is (faithfully) 1-flat over $\pi_0R$ in the sense of \cref{f.p.}.
			
			\item Assume that $\mc{A}$ is hypercomplete. Let $f: M\to N$ be a map between flat left $R$-modules. If $\pi_0f:\pi_0M\to \pi_0 N$ is an equivalence, then $f: M\to N$ is an equivalence.
			
			\item Let $M\in \lmodu_R(\mc{A})$ be a flat left $R$-module. Then for any $n\in \mathbb{Z}$, the natural map
			\[
			\pi_n(R)\overline{\otimes}_{\pi_0R}\pi_0M \to \pi_nM
			\]
			is an equivalence in $\lmodu_{\pi_0R}(\mc{A}^\heartsuit)$.
			
			\item Assume that $\mc{A}$ is Grothendieck. If $f: R\to S$ is a faithfully flat morphism in $\calg(\mc{A})$, then $\op{cofib}(f)$ is a flat $R$-module; the converse holds provided further that $\mc{A}$ is hypercomplete.
		}
	\end{prop}
	\begin{proof}
		For (1), we have that $\tau_{\geq 0}M$ is flat over $\tau_{\geq 0}R$, so it suffices to show the case when $R,M$ are connective. Therefore by $t$-exactness we have the following factorization 
		$$\begin{tikzcd}[row sep=large, column sep=7em]
			\rmodu_{\tau_{\geq 0}R}(\mc{A}) \arrow[d, "\tau_{\geq0}"] \arrow[r, "(-)\otimes_{R }M"] & \mc{A} \arrow[d, "\tau_{\geq0}"] \\
			\rmodu_{\tau_{\geq 0}R}(\mc{A})_{\geq0} \arrow[r, "(-)\otimes_{R }M", dashed]           & \mc{A}_{\geq0}                  
		\end{tikzcd}$$
		which implies that $(-)\otimes_{R }M: \rmodu_{R}(\mc{A}_{\geq0})\to\mc{A}_{\geq0}$ is left exact.
		Now since $\pi_0: \mc{A}_{\geq 0}^\otimes\to (\mc{A}^{\heartsuit})^\otimes$ is a symmetric monoidal functor preserving geometric realizations, we have the commutative diagram of relative tensor product functors.
		$$\begin{tikzcd}[row sep=large, column sep=7em]
			\rmodu_{R}(\mc{A}_{\geq0}) \arrow[d] \arrow[r, "(-)\otimes_{R }M"] & \mc{A}_{\geq0} \arrow[d] \\
			\rmodu_{\pi_0R}(\mc{A}^{\heartsuit}) \arrow[r, "(-)\otimes_{\pi_0R }\pi_0M"]                                   & \mc{A}^{\heartsuit}             
		\end{tikzcd}$$
		So we conclude that $(-)\overline{\otimes}_{\pi_0R }\pi_0M$ is left exact since the above horizontal functor preserves discrete objects. For the faithfully flat case, it follows directly from definition.\\
		(2) By definition of non-connective flatness, without loss of generalization we can assume that $R$, $M$ and $N$ are connective. 
		Since both $M$ and $N$ are flat and $\mathcal{A}$ is hypercomplete we may reduce to proving that $$\pi_nM\simeq\pi_nR \otimes_RM\xrightarrow{\pi_nR \otimes_Rf}\pi_nR \otimes_RN\simeq \pi_nN$$ is an equivalence for all $n \geq 0$. However, this map agrees with $$\pi_nR \overline{\otimes}_{\pi_0R} \pi_0M\xrightarrow{\pi_nR \overline{\otimes}_{\pi_0R} \pi_0f}\pi_nR \overline{\otimes}_{\pi_0R} \pi_0N,$$ which is an equivalence by virtue of the fact that $\pi_0(f)$ is an equivalence.\\
		(3) By definition we have that $\tau_{\geq 0}M$ is flat over $\tau_{\geq 0}R$ and $R\otimes_{\tau_{\geq 0}R}\tau_{\geq 0}M\simeq M$. Then it follows by combining the $t$-exactness of $(-)\otimes_{\tau_{\geq 0}R}\tau_{\geq 0}M$ and the natural equivalence $\pi_n(R)\otimes_{\tau_{\geq 0}R}\tau_{\geq 0}M\simeq\pi_n(R)\overline{\otimes}_{\pi_0R}\pi_0M$.\\
		(4) Due to \cref{flatpush}(1), we may reduce to the case where $R$ and $S$ are connective. Let $C$ denote  $\op{cofib}(f)$. 
		
		Now assume that $f: R\to S$ is faithfully flat. By \cref{l7}, it suffices to show that $C\otimes_R M$ is discrete for any discrete $R$-module $M$. Since $R\otimes_R M$ and $S\otimes_R M$ are discrete, we have that  $C\otimes_R M \in \mc{A}_{\leq 1}$. Therefore it suffices to show that $\pi_i(C\otimes_R M)=0$ when $i\neq 0$, which is deduced by combining the \cref{pi0flat}(1), \cref{cokerff} and the long exact sequence associated with $R\otimes_R M\to S\otimes_R M\to C\otimes_R M$.
		
		For the converse implication, to see that $S$ is flat over $R$, it suffices to show that $S\otimes_R M$ is discrete for any discrete $R$-module $M$, which is obvious by the cofiber sequence $R\otimes_R M\to S\otimes_R M\to C\otimes_R M$ with the first and third terms discrete. To see the faithfulness, due to the hypercompleteness, it is reduced to proving $M=0$ if $M$ is discrete and $S\otimes_R M=0 $, which is again by the cofiber sequence above.
	\end{proof}
	\begin{warn}\label{differenceflat}
		Suppose that $R\in\alg(\mca)$ is discrete. Let $M \in \lmodu_{R}(\mca)$ be a discrete left $R$-module. Beware that, if $\pi_0 M$ is 1-flat over $\pi_0R$ in the sense of \cref{f.p.}, this does not imply that $M$ is flat over $R$ in the (derived) sense \cref{ncflat}. However, under the projective rigidity assumption, this implication does hold (see \cref{pi0proj}). 
	\end{warn}

	\section{Projective modules}\label{proj}
	\subsection{Basic properties}\label{projsubsec}
	\begin{nota}
		Throughout \cref{projsubsec}, we denote the following condition  by (*):\\
		(*) $\mc{A}$ is right complete and
		$\mc{A}_{\geq0}$ is projectively generated, meaning $\mc{A}_{\geq0}\simeq\mc{P}_\Sigma(\mc{A}_{\geq 0}^{\op{cproj}})$.
		
	\end{nota}
	\begin{rem}Under the condition (*), the following hold:
		\enu{
			\item  \mca is Grothendieck.
			\item $\mc{A}$ is left complete by combining \cref{comple} and \textup{\cite[Remark C.1.5.9]{sag}.}
			\item For any $R\in \alg(\mc{A}_{\geq 0})$,  $\lmodu_R(\mc{A})_{\geq 0}$ satisfies the condition (*) too. 
			
		}
		
	\end{rem}
	\begin{prop}\label{compmod}
		Suppose that \mca satisfies the condition (*). Let $R\in \alg(\mc{A}_{\geq 0})$, and let $\mathcal{C}$ be the smallest idempotent complete stable subcategory of $\operatorname{LMod}_R(\mc{A})$ which contains all compact projective left modules. Then $\mathcal{C}=\operatorname{LMod}_R(\mc{A})^{\omega}$ is the full subcategory of compact modules.
	\end{prop}
	\begin{proof}
		Since $\operatorname{LMod}_R(\mc{A})$ is right complete, the collection of connective cover functors $\{\tau_{\geq -n}|n\geq 0\}$ is jointly conservative. Therefore by \cref{l1} $\operatorname{LMod}_R(\mc{A})$ is generated by $\{R\otimes\Sigma^{-n}P|n\geq0, P\in \mc{A}_{\geq 0}^{\op{cproj}} \}$ under small colimits. Then the compact generation of $\mc{A}$ implies $\mathcal{C}=\operatorname{LMod}_R(\mc{A})^{\omega}$.
	\end{proof}
	
	\begin{rem}
		As we will show in \cref{fullstableproj}, the ``idempotent-complete'' condition above can be removed under the stronger assumption of projective rigidity.
	\end{rem}

	We begin our study of projective modules, by examining their behavior under truncations and geometric realizations. 

	\begin{de}
		Suppose that \mca satisfies the condition (*).	Let $R$ be in $\alg(\mc{A}_{\geq 0})$. We say $P\in \lmodu_R(\mc{A})_{\geq 0}$ is a projective left $R$-module if it is a projective object in $\lmodu_R(\mc{A})_{\geq 0}$, meaning that the corepresentable functor $$\mapp_{\lmodu_R(\mc{A})_{\geq 0}}(P,-):\lmodu_R(\mc{A})_{\geq 0}\to \mc{S}$$ preserves geometric realizations.
	\end{de}
	Now we introduce the following lemma, which is a slight strengthening of \cite[Lemma 1.3.3.11(2)]{ha}\footnote{The statement there assumes that both $\mc{C}$ and $\mc{C}^{'}$ are left complete; we show that the assumption on $\mc{C}$ is unnecessary.}.
	\begin{lem}\label{lemma4.3}
		Let $\mc{C}$ and $\mc{C}'$ be stable $\infty$-categories equipped with $t$-structures. Then:
		\enu{
			
			\item If $F:\mc{C}_{\geq0}\to\mc{C}_{\geq0}^{'}$ is a functor preserves finite colimits, then $\tau_{\leq n}\circ F\xrightarrow{\sim}\tau_{\leq n}\circ F\circ\tau_{\leq n}$ is a natural equivalence in $\op{Fun}(\mc{C}_{\geq0},\mc{C}_{[0,n]}^{'})$ for any $n\geq0$.
			\item If $\mc{C}_{\geq0}$ admits geometric realizations and $\mc{C}'$ is left complete, then a functor $F: \mc{C}_{\geq0}\to\mc{C}_{\geq0}^{'}$ preserves finite colimits if and only if it preserves finite coproducts and geometric realizations.
		}
	\end{lem}
	\begin{proof}
		(1) Since $F$ is right exact, it preserves suspension. Given $X\in \mc{C}_{\geq0}$, then we have that the sequence $$F(\tau_{\geq n+1}X)\to F(X)\to F(\tau_{\leq n}X)$$ is a cofiber sequence in $\mc{C}^{'}_{\geq0}$ and that $F(\tau_{\geq n+1}X)\in \mc{C}^{'}_{\geq n+1}$. Taking $\tau_{\leq n}$, we get the natural equivalence $$\tau_{\leq n}F(X)\xrightarrow{\sim} \tau_{\leq n}F(\tau_{\leq n}X).$$
		(2) If $F$ preserves finite coproducts and geometric realizations of simplicial objects, then $F$ is right exact \cite[Lemma  1.3.3.10]{ha}. Conversely, suppose that $F$ is right exact; we wish to prove that $F$ preserves geometric realizations of simplicial objects. It will suffice to show that each composition
		$$
		\mathcal{C}_{\geq 0} \xrightarrow{F} \mc{C}_{\geq 0}^{\prime} \xrightarrow{\tau_{\leq n}}\left(\mathcal{C}_{\geq 0}^{\prime}\right)_{\leq n}
		$$
		By the (1), in virtue of the right exactness of $F$, this functor is equivalent to the composition
		$$
		\mathcal{C}_{\geq 0} \xrightarrow{\tau_{\leq n}}\left(\mathcal{C}_{\geq 0}\right)_{\leq n} \xrightarrow{\tau_{ \leq n} \circ F}\left(\mathcal{C}_{\geq 0}^{\prime}\right)_{\leq n} .
		$$
		It will therefore suffice to prove that $\tau_{\leq n} \circ F$ preserves geometric realizations of simplicial objects, which follows from \cite[Lemma 1.3.3.10]{ha} since both the source and target are equivalent to $n$-categories.
	\end{proof}
	
	We also introduce
	the following strengthening lemma of  \cite[Prop. 7.2.2.6]{ha}\footnote{The statement there assumes that  $\mc{C}$ is left complete; we show that such assumption is unnecessary.}.
	\begin{prop}\label{proj1}
		Let $\mc{C}$ be a stable $\infty$-category with a $t$-structure such that $\mc{C}_{\geq0}$ admits geometric realizations. Given $P\in \mc{C}_{\geq0}$, then the following conditions are equivalent:
		\enu{
			\item  $P$ is projective in $\mc{C}_{\geq0}$.
			\item For any $Q\in \mc{C}_{\geq0}$, the abelian group $\op{Ext}^1(P,Q)=0$.
			\item For any $Q\in \mc{C}_{\geq0}$, the abelian group $\op{Ext}^i(P,Q)=0$ when $i>0$.
			\item  The mapping spectrum functor $\underline{\mapp}_{\mc{C}}(P,-):\mc{C}\to \op{Sp}$ is $t$-exact.
			
		}
	\end{prop}
	\begin{proof}
		The implications $(3) \Rightarrow(2)$ and $(3)\Leftrightarrow (4)$ are obvious. The implication $(2) \Rightarrow(3)$ follows by replacing $Q$ by $Q[i-1]$.
		
		We first show that $(1) \Rightarrow(2)$. Let $f: \mathcal{C} \rightarrow \mathcal{S}$ be the functor corepresented by $P$. Let $M_{\bullet}$ be a Čech nerve for the morphism $0 \rightarrow Q[1]$, so that $M_n \simeq Q^n \in \mathcal{C}_{\geq 0}$. Then $Q[1]$ can be identified with the geometric realization $\left|M_{\bullet}\right|$. Since $P$ is projective, $f(Q[1])$ is equivalent to the geometric realization $\left|f\left(M_{\bullet}\right)\right|$. We have a surjective map $* \simeq \pi_0 f\left(M_0\right) \rightarrow$ $\pi_0\left|f\left(M_{\bullet}\right)\right|$, so that $\pi_0 f(Q[1])=\operatorname{Ext}_{\mathcal{C}}^1(P, Q)=0$.
		
		We now show that $(3) \Rightarrow(1)$. That $\mc{C}$ is stable implies that $f$ is homotopic to a composition
		$$
		\mathcal{C} \xrightarrow{F} \mathrm{Sp} \xrightarrow{\Omega^{\infty}} \mathcal{S},
		$$
		where $F$ is an exact functor. Applying (3), we deduce that $F$ is right $t$-exact (see \cite[ Definition 1.3.3.1]{ha}). The  \cref{lemma4.3} implies that the induced map $\mathcal{C}_{\geq 0} \rightarrow \mathrm{Sp}_{\geq 0}$ preserves geometric realizations of simplicial objects. Applying \cite[ Proposition 1.4.3.9]{ha} that $\op{Sp}\xrightarrow{\Omega^{\infty}} \mathcal{S}$ preserves small sifted colimits, we conclude that $f \mid \mathcal{C}_{\geq 0}$ preserves geometric realizations as well.
	\end{proof}
	Now we show that the $\pi_0$-truncation induces an equivalence between the homotopy categories of the (compact) projective objects in the connective part and the (compact) 1-projective objects in the heart.
	\begin{prop}[See \cite{stefanich2023classification} Proposition 2.4.8]\label{eqproj}
		Let $\mc{C}$ be a projectively generated Grothendieck prestable $\infty$-category\footnote{
			We do not use the notation $\mc{A}$ here, because the proposition does not require a monoidal structure.
		}. Then
		\enu{\item The truncation functor $H_0: \mathcal{C} \rightarrow \mathcal{C}^{\heartsuit}$ sends projective objects to 1-projective objects and compact objects to compact objects.
			\item The $0$-truncations of the compact projective objects of $\mathcal{C}$ provide a family of compact 1-projective generators for $\mathcal{C}^{\heartsuit}$.
			\item The functor $\operatorname{h}(\pi_0): \operatorname{h}(\mathcal{C}) \rightarrow \mathcal{C}^{\heartsuit}$ induced at the level of homotopy categories restricts to an equivalence between the full subcategories of (compact) projective objects and (compact) 1-projective objects.
		}
	\end{prop}
	\begin{proof}
		We first prove (1). The fact that $\pi_0$ sends compact objects to compact objects follows directly from the fact that the inclusion $\mathcal{C}^{\heartsuit} \rightarrow \mathcal{C}$ preserves filtered colimits. The fact that $\pi_0$ sends projective objects to 1-projective objects follows from \cref{proj1}.
		
		Item (2) follows directly from (1) together with the fact that $\pi_0$ is a localization. It remains to establish (3). We first prove fully faithfulness. Let $X, Y$ be a pair of projective objects of $\mathcal{C}$. Then the map $\operatorname{Map}_{\mathcal{C}}(X, Y) \rightarrow \operatorname{Map}_{\mathcal{C}^{\heartsuit}}\left(\pi_0(X), \pi_0(Y)\right)$ induced by $\pi_0$ is equivalent to the map $\eta_*: \operatorname{Map}_{\mathcal{C}}(X, Y) \rightarrow \operatorname{Map}_{\mathcal{C}}\left(X, \pi_0(Y)\right)$ of composition with the unit $\eta: Y \rightarrow \pi_0(Y)$. The fact that $X$ is projective and $\eta$ induces an equivalence on $\pi_0$ implies that $\eta_*$ is an effective epimorphism. Its fiber is given by $\operatorname{Map}_{\mathcal{C}}\left(X, \tau_{\geq 1}(Y)\right)$ which is connected since $X$ is projective. We conclude that $\eta_*$ induces an equivalence on $\pi_0$, and therefore $\operatorname{h}(\pi_0)$ is fully faithful when restricts on the full subcategory of projective objects.
		
		It remains to prove the essential surjectivity. In other words, we have to show that every (compact) 1-projective object of $\mathcal{C}^{\heartsuit}$ is the image under $\pi_0$ of a (compact) projective object of $\mathcal{C}$. We will establish the case of compact projective objects, and the proof in the projective case being similar. Let $Y$ be a compact 1-projective object of $\mathcal{C}^{\heartsuit}$. Applying (2) we may find a compact projective object $X$ in $\mathcal{C}$ such that $Y$ is a retract of $\pi_0(X)$. Let $r: \pi_0(X) \rightarrow \pi_0(X)$ be the induced retraction. The fully faithfulness part of (3) allows us to lift $r$ to an idempotent endomorphism $\rho$ of $X$ inside $\operatorname{h}(\mathcal{C})$. Let $X^{\prime}$ be a representative in $\mathcal{C}$ of the image of $\rho$. Then $X^{\prime}$ is a direct summand of $X$ (see similar argument in \cite[Lem. 1.2.4.6]{ha}) and therefore it is compact projective. The proof finishes by observing that $\pi_0(X^{\prime})=\operatorname{Im}(r)=Y$.
	\end{proof}

	\begin{prop}\label{l5.8}
		Suppose that \mca satisfies the condition (*).	Let $R$ be in $\alg(\mc{A}_{\geq 0})$ and $P\in \lmodu_R(\mc{A})_{\geq0}$. Then:
		\enu{ \item  $P$ is a projective $R$-module if and only if every map $X \to P$ in $\lmodu_R(\mc{A})_{\geq0}$ which induces an epimorphism on $\pi_0$ admits a section.
			\item If $P$ is a projective $R$-module, then $\pi_0 P\in \lmodu_{\pi_0R}(\mc{A}^{\heartsuit})$ is a 1-projective discrete $\pi_0R$-module.
			
		}
	\end{prop}
	\begin{proof}
		(1) This follows from the equivalent characterization between \cref{proj1} (1) and (2).\\
		(2) It follows by combining (1) and the equivalence $\mapp_{\lmodu_R(\mc{A})_{\geq0}}(P,M)\simeq \mapp_{\lmodu_{\pi_0R}(\mc{A}^{\heartsuit})}(\pi_0P,M)$ for a discrete $\pi_0R$-module $M$.
	\end{proof}
	\begin{cor}\label{l5.9}
		Suppose that \mca satisfies the condition (*).	Let $R\in \alg(\mc{A}_{\geq 0})$. Then the heart $\lmodu_{\pi_0R}(\mc{A}^{\heartsuit})$ has enough (1-)projectives.
	\end{cor}
	\begin{prop}\label{cocompinftycoprod}
		Suppose that \mca satisfies the condition (*).	Let $R$ be in $\alg(\mc{A}_{\geq 0})$ and $P\in \lmodu_R(\mc{A})_{\geq0}$. Then $P$ is projective if and only if there exists  a small collection of compact projective modules $\{P_\alpha\}$ in $\lmodu_R(\mc{A})_{\geq0}$ such that $P$ is a retraction of $\oplus_\alpha P_\alpha$.
	\end{prop}
	\begin{proof}
		Suppose first that $P$ is projective. By the projective generation there exists an equivalence of left $R$-modules $$ \op{colim}_\alpha M_\alpha \xrightarrow{\sim} P$$ where each $M_\alpha$ is compact projective. Then the induced map $\oplus_\alpha\pi_0 M_\alpha \rightarrow \pi_0 P$ is epimorphic. Invoking  \cref{proj1}, we deduce that $p$ admits a section (up to homotopy), so that $P$ is a retract of $M$. To prove the converse, we observe that the collection of projective left $R$-modules is stable under small coproducts and retracts by  \cref{proj1}. 
	\end{proof}

	\subsection{Projective rigidity and Lazard's theorem}\label{4.2}
	
	We now introduce the notion of \emph{projective rigidity}, a structural condition requiring that the dualizable objects of $\mathcal{A}_{\geq 0}^\otimes$ coincide exactly with its compact projectives. 
	Under such condition, we prove the Lazard's Theorem, which  classifies flat modules as precisely the filtered colimits of compact projective modules as desired.
	\begin{de}\label{projrigid}
		We say that a presentably symmetric monoidal additive \infcat $\mc{C}^\otimes\in \calg(\prlad)$ is  \textbf{projectively rigid} if it satisfies the following:
		\enu{ 
			\item   $\mc{C}$ is projectively generated (this implies that \mcc is prestable).
			\item  $\mc{C}^d=\mc{C}^{\op{cproj}}$, i.e. the dualizable objects coincide with compact projective objects.}
	\end{de}
	\begin{de}
		We say that the $ttt$-\infcat $(\mc{A}^\otimes, \mc{A}_{\geq0}) $ is projectively rigid if  $\mc{A}$ is right complete and $\mc{A}^\otimes_{\geq 0}$ is projectively rigid. 
	\end{de}
	\begin{rem}\label{projcocomplete} \,
		\enu{

			\item \textbf{Warning:} In general,  $(\mc{A}_{\geq 0})^d\subsetneq \mc{A}^d\cap \mc{A}_{\geq 0}$ because the dual of an object $X\in\mc{A}^d\cap \mc{A}_{\geq 0}$ is not necessarily connective! However, that holds exactly when $X$ is flat; see \cref{l1.3}(3), which claims that $(\mc{A}_{\geq 0})^d=\mc{A}^d\cap \mc{A}^{fl}$.
			\item Suppose that $\mc{A}^\otimes_{\geq 0}$ is projectively rigid. Then 
			$\mc{A}_{\geq0}^\otimes$ can be identified with the (symmetric monoidal) sifted cocompletion  $\mc{P}_\Sigma(\mc{A}_{\geq 0}^{\op{cproj}})^\otimes$. And its heart $(\mc{A}^{\heartsuit})^\otimes$ is 1-projectively rigid.
			
		}
	\end{rem}
	\begin{ex}\label{A.6}
		We list several examples of projectively rigid $ttt$-$\infty$-categories:
		\enu{
			\item The \infcat of spectra \opsp with the standard $t$-structure.
			\item $\op{Gr(\op{Sp})}^\otimes=\op{Fun}(\mathbb{Z}^{disc},\op{Sp})^\otimes$ is the $\infty$-category of graded spectra with the pointwise $t$-structure.
			\item The \infcat of filtered spectra $\op{Fil}(\opsp)$ with the pointwise $t$-structure.
			\item The \infcat of filtered spectra $\op{Fil}(\opsp)$ with the homotopy $t$-structure. Its connective part is defined by
			$
			\op{Fil}(\opsp)_{h\geq0}
			\overset{\mathrm{def}}{=}
			\{\, X_* \mid X_n \in \opsp_{\geq n} \text{ for each } n\geq0 \,\}.
			$
			
			\item The \infcat $\op{Sp}_{G,\geq0}^\otimes=\mc{P}_{\Sigma}(\op{Span}(\op{Fin}_{G});\spgeq)$ of connective genuine $G$-spectra for a finite group $G$.
			
			\item For an $\mathbb{E}_\infty$-ring $R$, the $\infty$-category of (connective) synthetic $R$-modules $\mathrm{Syn}_{R,\geq 0} = \mathcal{P}_\Sigma(\mathrm{Mod}_R^{\mathrm{ff}}(\op{Sp}))$ \cite{hesselholt2023dirac}. Here, $\mathrm{Mod}_R^{\mathrm{ff}}(\op{Sp}) \subset \modu_R(\opsp)$ denotes the smallest full subcategory containing $R$ that is closed under finite direct sums, shifts, and retractions. This forms a projectively rigid $ttt$-$\infty$-category whose heart is $\mathrm{Syn}_R^\heartsuit \simeq \mathrm{Mod}_{R_*}(\mathrm{GrAb})$.
			
			\item The universal example: cobordism $\funct(\mb{Cob}_1^{\op{op}},\op{Sp})^\otimes$ with the pointwise $t$-structure.
			
			\item The \infcat $\op{Shv}(X,\op{Sp})^\otimes$ of sheaves on a stone space with the standard $t$-structure.
			\item $\op{QCoh}(X)$ with the standard $t$-structure when $X$ is an affine quotient stack, i.e.\ a stack of the form $\spec(R)/G$ for a linearly reductive group $G$ acting on $\spec(R)$. In this case the compact-projective objects are generated, under retracts, by pullbacks of $G$-representations, and the dual is given by the pullback of the dual representation.
			\item 
			The \infcat $\op{SH}(k)^{\op{A}}$ of Artin motivic spectra over a perfect field $k$ \cite[Definition 1.5]{burklund2020galois} with the very effective $t$-structure. It is defined as the smallest stable full subcategory of the motivic stable homotopy category $\op{SH}(k)^{\op{A}}\subset\mathrm{SH}(k)$ that is closed under small colimits and is generated by the suspension spectra $\Sigma^\infty_+ \mathrm{Spec}(L)$, where $L$ ranges over all finite \'{e}tale $k$-algebras. 
			The connective part $\op{SH}(k)^{\op{A}}_{\geq 0}$ (also referred to very effective Artin spectra) is generated by $\Sigma^\infty_+ \mathrm{Spec}(L)$ for all finite \'{e}tale $k$-algebras $L$. One can see that $\Sigma^\infty_+ \mathrm{Spec}(L)\in\op{SH}(k)^{\op{A}}_{\geq 0}$ is a self-dual compact projective object by infinite loop space recognition principle.

			\item The  \infcat $\op{DM}(k,\mathbb{Z}[1/p])$ of Voevodsky motives over a perfect field $k$ with coefficients in $\mathbb{Z}[1/p]$ \cite[see][Chapter 14]{bachmann2021norms} (where $p$ is the characteristic of $k$, or $p=1$ if $k$ is a $\mathbb{Q}$-algebra) equipped with the Chow $t$-structure  generated by smooth projective varieties and their $\mathbb{P}^1$-desuspensions \cite{bondarko2010weight}. The mapping spectra between smooth projective varieties are connective, so they form compact projective generators, and they are also dualizable within the retract-closed subcategory generated by them.
		}
	\end{ex}
	\begin{ex}
		Here are some examples whose connective parts are projectively generated but not projectively rigid:
		\enu{
			\item The \infcat of light condensed spectra $\op{Cond}^{\op{light}}(\op{Sp})^\otimes=\op{Sp}(\mc{P}_{\Sigma}(\mb{Stonean}^{\op{light}}))^\otimes$.
			\item The \infcat of light Solid spectra $\op{Solid}^{\op{light}}(\op{Sp})^\otimes$.
			\item 
			$\op{Fun}(X,\op{Sp})^\otimes$ the $\infty$-category of parametrized spectra on a small $\infty$-groupoid $X$, with pointwise tensor product.
		}

	\end{ex}
	
	\begin{prop}\label{projrigidproperties}
		Suppose that the $ttt$-\infcat $(\mc{A}^\otimes, \mc{A}_{\geq0}) $ is projectively rigid. Then:
		
		\enu{ 
			\item  $\mc{A}^\otimes$ is compactly-rigidly generated, meaning the compact objects and dualizable objects in it coincide. Particularly we have $\mc{A}^\otimes\in \calg(\pr_{\op{st},\omega}^L)$.
			\item Let $R$ be in $\calg(\mc{A}_{\geq 0})$. Then $\modu_R(\mc{A}_{\geq 0})^\otimes$ is projectively rigid too.
			\item Let $R$ be in $\calg(\mc{A}_{\geq 0})$. Then $\modu_R(\mc{A})^\otimes$ is compactly-rigidly generated.
		} 
	\end{prop}
	\begin{proof}
		(1) Since $\mc{A}$ is right complete, the collection of connective cover functors $\{\tau_{\geq -n}|n\geq 0\}$ is jointly conservative. Therefore by \cref{l1}, $\mc{A}$ is generated by $\{\Sigma^{-n}P|n\geq0, P\in \mc{A}_{\geq 0}^{\op{cproj}} \}$ under small colimits. By \cref{du} (4), we have $\{\Sigma^{-n}P|n\geq0, P\in \mc{A}_{\geq 0}^{\op{cproj}} \}\subset \mc{A}^d$ and hence $\mc{A}^c=\mc{A}^d$. 
		\\(2) Since the symmetric monoidal functor $$\mc{A}_{\geq 0}^\otimes\xrightarrow{R\otimes(-)} \modu_R(\mc{A}_{\geq 0})^\otimes$$
		preserves compact projective objects and dualizable objects, we conclude that 
		
		\begin{enumerate}
			\item 	The unit $R$ is dualizable in $\modu_R(\mc{A}_{\geq 0})$.
			\item 		  $R\otimes P$ is dualizable in $\modu_R(\mc{A}_{\geq 0})$ if $P\in\mc{A}_{\geq 0}$ is compact projective.
			
		\end{enumerate}
		So the full subcategory of dualizable objects $\modu_R(\mc{A}_{\geq 0})^d$ contains $\{R\otimes X|X\in \mc{A}_{\geq 0}^{\op{cproj}}\}$. Then combining \cref{l1} (2) and \cref{du} (2)(3), we get $\modu_R(\mc{A}_{\geq 0})^{\op{cproj}}\subset \modu_R(\mc{A}_{\geq 0})^d$. And that the unit $R$ is compact projective implies the equality $\modu_R(\mc{A}_{\geq 0})^{\op{cproj}}= \modu_R(\mc{A}_{\geq 0})^d$.\\
		(3) It follows by applying  (1) and (2) to $\modu_R(\mc{A})^\otimes$.
	\end{proof}

	\begin{prop}\label{l4.2}
		Suppose that the $ttt$-\infcat $(\mc{A}^\otimes, \mc{A}_{\geq0}) $ is projectively rigid. Let $R\in \calg(\mc{A}_{\geq 0})$ and $M\in \modu_R(\mc{A})_{\geq 0}$. Then $M$ is compact projective if and only if $M$ is dualizable in the whole stable \infcat $\modu_R(\mc{A})^\otimes$ and $M$ is a flat $R$-module.
	\end{prop}
	\begin{proof}
		It follows by combining the projective rigidity and \cref{l1.3}(3).
	\end{proof}
	
	\begin{prop}\label{pi0proj}
		Suppose that the $ttt$-\infcat $(\mc{A}^\otimes, \mc{A}_{\geq0}) $ is projectively rigid. Let $R\in \alg(\mc{A}_{\geq0})$ and $M$ be a connective left $R$-module. Then:
		\enu{\item If $M$ is projective, then $M$ is flat.
			\item  $M$ is compact projective if and only if it is left dualizable in $\lmodu_{R}(\ageq)$.
			\item  $M$ is (compact) projective if and only if it is flat and $\pi_0M$ is (compact) 1-projective in $\lmodu_{\pi_0R}(\mc{A}^{\heartsuit})$.
			\item Suppose that $R\in \alg(\mc{A}^\heartsuit)$ is discrete. Then $M$ is flat if and only if $M$ is discrete and 1-flat over $\pi_0R$ in the sense of \cref{f.p.}.
			
		}
	\end{prop}
	\begin{proof}
		(1) Since flat modules are closed under small coproducts and retractions, we reduce to the case $M=R\otimes P$ where $P\in\mc{A}_{\geq0}^{\op{cproj}}$. That is easy because $(-)\otimes_R(R\otimes P)\simeq (-)\otimes P$ reduces to the case $R=\mb{1}$, which is deduced by \cref{l4.2}.\\
		(2) By \cref{rigidlygene}, we see that left dualizable objects are closed under finite coproducts and retracts. We observe that every $R\otimes P$ is left dualizable (given by $P^{\vee}\otimes R$), which proves ``only if'' direction. For the ``if'' direction, if $M$ is left dualizable, then it follows from $$\mapp_{\lmodu_R(\ageq)}(M,-)\simeq \mapp_{\ageq}(\mb{1},\,  ^{\vee}M\otimes_R -)$$ and compact projectivity of the unit.
		\\(3) By the (1) we have that every projective left $R$-module is flat. Secondly, the fact that $\pi_0$ sends projective objects to 1-projective objects was already observed in  \cref{l5.8}. This finishes the proof of the ``only if'' direction.
		
		Assume now that $M$ is flat and $\pi_0M$ is (compact) 1-projective. Applying  \cref{eqproj} we may find a (compact) projective $R$-module $M^{\prime}$ and an isomorphism $\pi_0M^{\prime}=\pi_0M$. The fact that $M^{\prime}$ is projective allows us to lift this isomorphism to a map $f: M^{\prime} \rightarrow M$. We observe that $f$ is an equivalence by \cref{pi0flat}(2).\\
		(4) The ``only if'' direction follows from (1) and $M\simeq R\otimes_RM$. For the ``if'' direction, given a discrete right $R$-module $N$, we wish to show that $N\otimes_RM$ is discrete too. Take a map $f:P\to N$ of right $R$-module such that $P$ is projective and $f$ induces an epimorphism on $\pi_0$. Then we have exact sequence $$0\to\pi_0\op{fib}(f)\to \pi_0P\to\pi_0N\to0$$ and hence $\op{fib}(f)$ is discrete.
		Now tensoring with $M$, we get an exact sequence
		$$0\to\pi_0(\op{fib}(f)\otimes_RM)\to \pi_0(P\otimes_RM)\to\pi_0(N\otimes_RM)\to0 $$ by 1-flatness and the commutative diagram of relative tensor product functors.
		$$\begin{tikzcd}[row sep=large, column sep=7em]
			\rmodu_{R}(\mc{A}_{\geq0}) \arrow[d, "\pi_0"] \arrow[r, "(-)\otimes_{R }M"] & \mc{A}_{\geq0} \arrow[d, "\pi_0"] \\
			\rmodu_{\pi_0R}(\mc{A}^{\heartsuit}) \arrow[r, "(-)\otimes_{\pi_0R }\pi_0M"]                                   & \mc{A}^{\heartsuit}             
		\end{tikzcd}$$
		Because $P$ is also flat by (1), we see that $\pi_1(N\otimes_RM)=0$. We actually have proved for any discrete right $R$-module $N$ has the property $\pi_1(N\otimes_RM)=0$. Then by induction on $n$ we get that $\pi_n(\op{fib}(f)\otimes_RM)=\pi_{n+1}(N\otimes_RM)=0$ for all $n>0$, which implies $N\otimes_RM$ is discrete.
	\end{proof}
	
	We now state the main result of this subsection.
	\begin{thm}[Lazard's Theorem]\label{cocompfil}
		Suppose that the $ttt$-\infcat $(\mc{A}^\otimes, \mc{A}_{\geq0}) $ is projectively rigid. Let $R\in \alg(\mc{A}_{\geq 0})$ and $M\in \lmodu_R(\mc{A})$. Then $M$ is a flat left $R$-module if and only if it is equivalent to a filtered colimit of compact projective left $R$-modules.
		
		In particular, the \infcat of flat $R$-modules can be identified with
		$$\operatorname{LMod}_R(\mathcal{A})^{fl}\simeq \operatorname{Ind}(\operatorname{LMod}_R(\mathcal{A}_{\geq 0})^{\mathrm{cproj}}) .$$
	\end{thm}
	\begin{proof}
		We take the strategy in \cite[Prop. 2.2.22]{stefanich2023classification}. The ``if'' direction can be concluded by combining \cref{l4.2} and \cref{l1.3}(1).\\
		For the ``only if'' direction, assume now that $M$ is flat. Let $L\mathcal{M}^{\text {cp }}$ denote the full subcategory of $L\mc{M}=\lmodu_R(\mc{A})_{\geq 0}$ spanned by compact projective objects and consider the functor $F(-):\left(L\mathcal{M}^{\mathrm{cp}}\right)^{\mathrm{op}} \rightarrow \mc{S}$  represented by $M$. We wish to show that this functor defines an ind-object of $L\mathcal{M}^{\mathrm{cp}}$. Let $(-)^{\vee}: R\mathcal{M}^{\mathrm{cp}} \rightarrow\left(L\mathcal{M}^{\mathrm{cp}}\right)^{\mathrm{op}}$ be the dualization equivalence introduced in \cref{ldualeq}. We will prove that $F((-)^{\vee}): R\mathcal{M}^{\mathrm{cp}} \rightarrow\mc{S}$  defines a pro-object of $R\mathcal{M}^{\mathrm{cp}}$.

		Let $p: \mathcal{E} \rightarrow R\mathcal{M}$ be the left fibration associated to the functor $\operatorname{Map}_{\ageq}\left(\mb{1},-\otimes_R M\right): R\mathcal{M}\to\mc{S}$. Then the base change of $p$ to $R\mathcal{M}^{\mathrm{cp}}$ is the left fibration classifying $F((-)^\vee)$. We have to show that every finite diagram $G: \mathcal{I} \rightarrow \mathcal{E} \times_{R\mathcal{M}} R\mathcal{M}^{\text {cp }}$ admits a left cone. The fact that $M$ is flat implies that the functor $$\operatorname{Map}_{\ageq}\left(\mb{1},-\otimes_R M\right): R\mathcal{M}\to\mc{S}$$ is left exact, and therefore $\mce$ is cofiltered and $G$ extends to a left cone $G^{\triangleleft}: \mathcal{I}^{\triangleleft} \rightarrow \mathcal{E}$. Let $\overline{N}=\left(M, \rho: \mb{1} \rightarrow N \otimes_R M\right)$ be the value of $G^{\triangleleft}$ at the cone point. To show that $G$ extends to a left cone in $\mathcal{E} \times_{R\mathcal{M}} R\mathcal{M}^{\mathrm{cp}}$ it is enough to prove that $\overline{N}$ receives a map from an object in $\mathcal{E} \times_{R\mathcal{M}} R\mathcal{M}^{\mathrm{cp}}$. This amounts to showing that there exists a map $N^{\prime} \rightarrow N$ from a compact projective right $R$-module $N^{\prime}$ with the property that $\rho$ factors through $N^{\prime} \otimes_R M$. This follows from the fact that $\mb{1}$ is compact projective in $\ageq$.
	\end{proof}

	\subsection{Modules over discrete algebras}\label{4.3}
	
	Having established Lazard's theorem, we turn our attention to the behavior of discrete algebras within a topological framework. In this subsection, we prove that under the  assumption of projective rigidity, the module category $\lmodu_R(\mca)$ over a discrete algebra $R \in \alg(\mathcal{A}^\heartsuit)$ is naturally (monoidally if $R$ is \ein) equivalent to the (unbounded) derived category $\mathcal{D}(\lmodu_{\pi_0 R}(\mathcal{A}^\heartsuit))$.

	\begin{thm}\label{dicre}
		Suppose that the $ttt$-\infcat $(\mc{A}^\otimes, \mc{A}_{\geq0}) $ is projectively rigid. The following hold: 
		\enu{
			\item For any discrete $R\in \alg(\mc{A}^{\heartsuit})$ there exists a (unique up to contractible choices) equivalence in $\pr_{\op{st}}^{t\text{-}rex}$
			$$\mc{D}(\lmodu_{\pi_0R}(\mc{A}^{\heartsuit}))\xrightarrow{\sim} \lmodu_{R}(\mc{A})$$
			which induces the identity functor on the heart.
			\item  For any discrete commutative algebra  $R\in \calg(\mc{A}^{\heartsuit})$ there exists a (unique up to contractible choices) equivalence in $\calg(\pr_{\op{st}}^{t\text{-}rex})$
			$$\mc{D}(\modu_{\pi_0R}(\mc{A}^{\heartsuit}))^\otimes\xrightarrow{\sim} \modu_{R}(\mc{A})^\otimes$$
			which induces the identity functor on the heart, where the symmetric monoidal structure on left-hand side is induced by the projective model with tensor product of chain complexes.
		}
		
	\end{thm}
	\begin{proof}
		(1) Since both are right complete and we have $\mc{P}_{\Sigma}(\lmodu_{R}(\ageq)^{\op{cproj}})\simeq \lmodu_{R}(\ageq)$ and $\mc{D}(\lmodu_{\pi_0R}(\mc{A}^{\heartsuit}))_{\geq0}\simeq\mcp_\Sigma(\lmodu_{\pi_0R}(\mc{A}^{\heartsuit})^{\op{cproj}})$,  it suffices to show that taking $\pi_0$ induces an equivalence $$\lmodu_{R}(\ageq)^{\op{cproj}}\simeq \lmodu_{\pi_0R}(\mc{A}^{\heartsuit})^{\op{cproj}}.$$
		By \cref{eqproj} it is reduced to showing every compact projective $R$-module $P\in\lmodu_{R}(\ageq)^{\op{cproj}}$ is discrete, but that follows from \cref{pi0proj}.\\
		(2) By remark \cref{projcocomplete}(1), it suffices to show that $\mc{D}(\modu_{\pi_0R}(\mc{A}^{\heartsuit}))_{\geq0}^\otimes\simeq\mc{P}_{\Sigma}(\modu_{\pi_0R}(\mc{A}^{\heartsuit})^{\op{cproj}})^\otimes$ is the symmetric monoidal projective sifted cocompletion \textup{\cite[see ][Prop. 4.8.1.10]{ha}} of $\modu_{\pi_0R}(\mc{A}^{\heartsuit})^{\op{cproj}}$. That is to show the following: 
		\begin{enumerate}[label=(\alph*)]
			\item The natural inclusion $\modu_{\pi_0R}(\mc{A}^{\heartsuit})^{\op{cproj}}\hookrightarrow\mc{D}(\modu_{\pi_0R}(\mc{A}^{\heartsuit}))_{\geq0}$ is a symmetric monoidal functor which preserves finite coproducts.
			\item  $\mc{D}(\modu_{\pi_0R}(\mc{A}^{\heartsuit}))_{\geq0}^\otimes$ is presentably symmetric monoidal.
			\item The inclusion induces an equivalence $\mc{P}_{\Sigma}(\modu_{\pi_0R}(\mc{A}^{\heartsuit})^{1\text{-}\op{cproj}})\simeq \mc{D}(\modu_{\pi_0R}(\mc{A}^{\heartsuit}))_{\geq0}$. 
		\end{enumerate}
		The (a) and (c) follow directly from the construction of projective model on derived category. The (b) follows from \cite[Prop. 1.3.5.21]{ha} and the explicit internal hom construction in $\mc{D}(\modu_{\pi_0R}(\mc{A}^{\heartsuit}))$ $$\underline{\operatorname{Map}}_{\mc{D}}\left(M_*, N_*\right)_p=\prod_{n \in \mathbf{Z}}\underline{\operatorname{Hom}}_{\mathcal{M}}\left(M_n, N_{n+p}\right)$$ for each integer $p$, where we denote $\mc{D}=\mc{D}(\modu_{\pi_0R}(\mc{A}^{\heartsuit}))$ and $\mc{M}=\modu_{\pi_0R}(\mc{A}^{\heartsuit})$. We view $\underline{\operatorname{Map}}_{\mc{D}}\left(M_*, N_*\right)_*$
		as a chain complex with values in $\mc{M}$, with differential given by the formula
		$$
		(d f)(x)=d(f(x))-(-1)^p f(d x)
		$$
		for $f \in \operatorname{Map}_{\mc{D}}\left(M_*, N_*\right)_p$.
	\end{proof}
	\begin{rem} In fact, by our argument the uniqueness in above theorem can be promoted as which induces the identity functor on compact 1-projective $\pi_0R$-modules in the heart.
		
	\end{rem}
	\begin{ex}
		Note that \cref{dicre} need not hold if the $ttt$-$\infty$-category
		$(\mc{A}^{\otimes},\mc{A}_{\geq0})$ is not projectively rigid.
		For example, consider
		\[
		\op{Sp}^{BS^1}:=\funct(BS^1,\op{Sp})
		\]
		with the pointwise symmetric monoidal structure. The object
		$\underline{\mathrm{H}\mathbb{Z}}$, equipped with the trivial
		$S^1$-action, is a discrete commutative algebra in
		$\op{Sp}^{BS^1}$. However,
		\[
		\modu_{\underline{\mathrm{H}\mathbb{Z}}}
		(\op{Sp}^{BS^1})
		\neq
		\mc{D}\bigl(
		\modu_{\underline{\mathrm{H}\mathbb{Z}}}
		(\op{Sp}^{BS^1,\heartsuit})
		\bigr)
		\simeq
		\mcd(\mathbb{Z}).
		\]
		In particular, projective objects need not be flat in this setting.
		For instance, the representable presheaf
		\[
		\Sigma_+^\infty \mapp_{BS^1}(-,*)
		\]
		is compact projective in $\op{Sp}^{BS^1}_{\geq0}$, but is not flat.
	\end{ex}
	
	We next record a K\"unneth spectral sequence and a resulting criterion for
	flatness. These also answer a question posed in the original version of this
	paper.
	
	\begin{lem}\label{gradedprojective}
		Suppose that the $ttt$-$\infty$-category
		$(\mc{A}^{\otimes},\mc{A}_{\geq0})$ is projectively rigid.
		Let $R\in\alg(\mc{A}_{\geq0})$ and let
		$P\in\lmodu_R(\mc{A}_{\geq0})$ be projective. Then
		\[
		P_*\in
		\lmodu_{R_*}\bigl(\operatorname{Gr}(\mc{A}^{\heartsuit})\bigr)
		\]
		is $1$-projective. If $P$ is compact projective, then $P_*$ is
		compact $1$-projective.
	\end{lem}
	\begin{proof}
		By \cref{pi0proj}, $P$ is flat and $\pi_0P$ is a
		$1$-projective $\pi_0R$-module. Hence \cref{pi0flat}(3) assembles
		to an equivalence of graded $R_*$-modules
		\[
		P_*
		\simeq
		R_*\overline{\otimes}_{\pi_0R}\pi_0P,
		\]
		where $\pi_0R$ and $\pi_0P$ are regarded as concentrated in degree
		zero.
		
		The extension-of-scalars functor
		\[
		R_*\overline{\otimes}_{\pi_0R}(-):
		\lmodu_{\pi_0R}
		(\mc{A}^{\heartsuit})
		\longrightarrow
		\lmodu_{R_*}
		\bigl(\operatorname{Gr}(\mc{A}^{\heartsuit})\bigr)
		\]
		preserves $1$-projective objects, since its right adjoint is
		restriction of scalars, which preserves epimorphisms. This proves
		the first assertion. The compact case follows similarly, since
		restriction of scalars also preserves filtered colimits.
	\end{proof}
	\begin{prop}[K\"unneth spectral sequence]\label{kunnethss}
		Suppose that the $ttt$-$\infty$-category
		$(\mc{A}^{\otimes},\mc{A}_{\geq0})$ is projectively rigid.
		Let $R\in\alg(\mc{A}_{\geq0})$, and let
		\[
		N\in\rmodu_R(\mc{A})_{\geq0},
		\qquad
		M\in\lmodu_R(\mc{A})_{\geq0}.
		\]
		Then there is a natural convergent spectral sequence in
		$\mc{A}^{\heartsuit}$
		\[
		E^2_{p,q}
		=
		\operatorname{Tor}^{R_*}_{p}(N_*,M_*)_q
		\Longrightarrow
		\pi_{p+q}(N\otimes_RM).
		\]
		Here $\operatorname{Tor}^{R_*}$ denotes the left derived functors
		of the graded relative tensor product in
		$\operatorname{Gr}(\mc{A}^{\heartsuit})$.
	\end{prop}
	
	\begin{proof}
		We follow the construction of
		\cite[Construction 7.2.1.18 and Proposition 7.2.1.19]{ha}.
		Consider the collection
		\[
		\mathcal P
		=
		\{P[n]\mid P \text{ is compact projective},\ n\geq0\}
		\subset\lmodu_R(\mc{A})_{\geq0}.
		\]
		By projective generation, every connective left $R$-module admits
		an $\mathcal P$-free $\mathcal P$-hypercovering; compare
		\cite[Proposition 7.2.1.4]{ha}. Thus we may choose an augmented
		simplicial object
		$
		P_\bullet\longrightarrow M
		$
		whose terms are coproducts of objects of $\mathcal P$ and whose
		geometric realization is $M$.
		
		By \cref{gradedprojective}, each $P_*$ with $P\in\mathcal P$ is a
		compact $1$-projective graded $R_*$-module. Moreover,
		\cref{proj1} and \cref{eqproj} imply that these objects and their
		graded shifts form a family of projective generators for
		$
		\lmodu_{R_*}
		\bigl(\operatorname{Gr}(\mc{A}^{\heartsuit})\bigr).
		$
		Consequently, exactly as in the proof of
		\cite[Proposition 7.2.1.19]{ha}, the normalized chain complex
		associated to
		$
		(P_\bullet)_*
		$
		is a $1$-projective resolution of $M_*$ as a graded
		$R_*$-module.
		
		It remains only to identify the $E^1$-page. If $P$ is projective,
		then it is flat by \cref{pi0proj}. Hence
		\[
		(-)\otimes_RP:
		\rmodu_R(\mc{A})\longrightarrow\mc{A}
		\]
		is $t$-exact, and its restriction to the heart is
		\[
		(-)\overline{\otimes}_{\pi_0R}\pi_0P.
		\]
		Therefore, for every right $R$-module $X$,
		\[
		\pi_*(X\otimes_RP)
		\simeq
		X_*\overline{\otimes}_{\pi_0R}\pi_0P
		\simeq
		X_*\overline{\otimes}_{R_*}P_*.
		\]
		The same formula holds for coproducts and shifts of projective
		modules.
		
		Now apply the spectral sequence associated to the simplicial object
		$
		N\otimes_RP_\bullet
		$
		as in \cite[Construction 7.2.1.18]{ha}. Its $E^1$-page is the
		normalized chain complex associated to
		\[
		N_*\overline{\otimes}_{R_*}(P_\bullet)_*.
		\]
		Taking homology therefore gives
		\[
		E^2_{p,q}
		\simeq
		\operatorname{Tor}^{R_*}_{p}(N_*,M_*)_q.
		\]
		Since $N$, $M$, and $R$ are connective, the spectral sequence is
		first-quadrant and converges to
		\[
		\pi_{p+q}
		\left(
		N\otimes_R|P_\bullet|
		\right)
		\simeq
		\pi_{p+q}(N\otimes_RM)
		\]
		by \cite[Proposition 1.2.4.5]{ha}.
	\end{proof}
	\begin{prop}\label{flatgradedcriterion}
		Suppose that the $ttt$-$\infty$-category
		$(\mc{A}^{\otimes},\mc{A}_{\geq0})$ is projectively rigid.
		Let $R\in\alg(\mc{A}_{\geq0})$ and let
		$M\in\lmodu_R(\mc{A})_{\geq0}$. Then the following are equivalent:
		\enu{
			\item $M$ is flat.
			\item $\pi_0M$ is $1$-flat over $\pi_0R$, and the natural map
			\[
			R_*
			\overline{\otimes}_{\pi_0R}
			\pi_0M
			\longrightarrow
			M_*
			\]
			is an equivalence of graded $R_*$-modules.
		}
	\end{prop}
	
	\begin{proof}
		The implication $(1)\Rightarrow(2)$ follows from
		\cref{pi0flat}(1),(3).
		
		Conversely, assume (2). Then $M_*$ is $1$-flat as a graded
		$R_*$-module. Indeed, for every graded right $R_*$-module $Q_*$,
		we have
		\[
		Q_*\overline{\otimes}_{R_*}M_*
		\simeq
		Q_*\overline{\otimes}_{\pi_0R}\pi_0M,
		\]
		and the latter functor is exact since $\pi_0M$ is $1$-flat over
		$\pi_0R$. Hence
		\[
		\operatorname{Tor}^{R_*}_{p}(Q_*,M_*)=0
		\qquad (p>0).
		\]
		
		Let now $N$ be a discrete right $R$-module. Applying
		\cref{kunnethss}, the K\"unneth spectral sequence degenerates at
		the $E^2$-page and gives
		\[
		\pi_*(N\otimes_RM)
		\simeq
		N_*\overline{\otimes}_{R_*}M_*.
		\]
		Since $N$ is discrete, $N_*$ is concentrated in degree zero.
		Using the assumed description of $M_*$, we obtain
		\[
		\begin{aligned}
			N_*\overline{\otimes}_{R_*}M_*
			&\simeq
			N_*\overline{\otimes}_{R_*}
			\left(
			R_*\overline{\otimes}_{\pi_0R}\pi_0M
			\right)\\
			&\simeq
			N\overline{\otimes}_{\pi_0R}\pi_0M,
		\end{aligned}
		\]
		which is concentrated in degree zero. Thus
		$N\otimes_RM$ is discrete for every discrete right $R$-module
		$N$. By \cref{l7}, $M$ is flat.
	\end{proof}
	\subsection{Cohn localizations of $\mathbb{E}_\infty$-algebras}
	
	While ordinary commutative localization strictly inverts elements, Cohn localization forces more general morphisms between finitely generated projective modules to become invertible. In this subsection, we generalize Neeman and Schofield's derived Cohn localization \cite{cohn1985free,schofield1985representation} to the $ttt$-$\infty$-categorical setting. We establish the existence and universal property of Cohn localizations for \ein-algebras in a projectively rigid base.
	
	Let us first recall the following classical result of Cohn localizations.
	\begin{thm}[\cite{schofield1985representation} Theorem 4.1]
		Let $A$ be an associative ring. Let $\Sigma$ be a set of morphisms between finitely generated projective right $A$-modules. Then there are a ring $A_{\Sigma}$ and a morphism of rings $f_{\Sigma}: A \longrightarrow A_{\Sigma}$, called the universal localisation of $A$ at $\Sigma$, such that
		\enu{\item  $f_{\Sigma}$ is $\Sigma$-inverting, i.e. if $\alpha: P \longrightarrow Q$ belongs to $\Sigma$, then $\alpha \otimes_A 1_{A_{\Sigma}}: P \otimes_A A_{\Sigma} \longrightarrow Q \otimes_A A_{\Sigma}$ is an isomorphism of right $A_{\Sigma}$-modules, and
			\item $f_{\Sigma}$ is universal $\Sigma$-inverting, i.e. for any $\Sigma$-inverting ring homomorphism $\psi: A \longrightarrow B$, there is a unique ring homomorphism $\bar{\psi}: A_{\Sigma} \longrightarrow B$ such that $\bar{\psi} f_{\Sigma}=\psi$.
			Moreover, the homomorphism $f_{\Sigma}$ is a ring epimorphism and $\operatorname{Tor}_1^A\left(A_{\Sigma}, A_{\Sigma}\right)=0$.}
		
	\end{thm}
	This theorem also works for commutative rings\footnote{See \cite{angeleri2020flat} for a discussion about the relation between Cohn localizations and epimorphisms.}, which is the case we mainly care.
	Neeman constructed the Cohn localization in the derived category of a commutative ring in \cite[\textsection4]{neeman2006non}. That motivates us to give a higher categorical correspondence.
	The Cohn localization is very useful in our abstract framework because most interesting cases are only projectively generated but not freely generated. 
	
	The main result in this subsection is the following.
	\begin{thm}\label{unilocal}
		Suppose that the $ttt$-\infcat $(\mc{A}^\otimes, \mc{A}_{\geq0}) $ is projectively rigid. Let $R\in \calg(\ageq)$ and $$S=\{P_\beta\xrightarrow{f_\beta} Q_\beta\}$$ be a set of morphisms between compact projective $R$-modules. Then there exists a Cohn localization $R\to R[S^{-1}]\in \calg(\ageq)$ satisfying the following universal property:\\
		For any $B\in \calg(\mc{A})$, the induced map $$\mapp_{\calg(\mc{A})}(R[S^{-1}],B)\to \mapp_{\calg(\mc{A})}(R,B)$$ is a $(-1)$-truncated map whose image on $\pi_0$ consists those maps $R\to B$ such that for each $f_\beta\in S$,  $B\otimes_RP_\beta\to B\otimes_RQ_\beta$ is an equivalence of $B$-modules.
	\end{thm}
	
	\begin{rem}
		\enu{
			\item See \textup{\cite[\textsection3]{hoyois2020cdh}  or \cite[\textsection3.4]{mantovani2023localizations}} for a discussion in the case where $Q_\beta=\mb{1}$ for each $\beta$, which is related to Moore objects in the general setting.
			\item The Cohn localization with respect to $S$ is unique up to contractible choices by \cref{unilocal}.
		}
		
	\end{rem}
	Before the proof, we recall some useful lemmas.
	\begin{lem}[See \cite{arakawa2025monoidal} B.5]
		Let $\mc{C}\xrightarrow{p}\mc{B}$ be a cocartesian fibration of $\infty$-categories. Let $\{S_b|b\in\mc{B}\}$ be given collections of morphisms such that $S_b\subset \funct(\Delta^1,\mc{C}_b)$ for each $b\in \mc{B}$. We denote $S=\bigcup_b S_b$. If for any morphism $s\to t\in\mc{B}$ the cocartesian transformation $\mc{C}_s\to\mc{C}_t$ sends $S_s$ into $S_t$, then the induced functor $q:\mc{D}=\mc{C}[S^{-1}]\to \mc{B}$ from the localization of $\mc{C}$ at $S$ is a cocartesian fibration and canonical functor  
		$$
		\begin{tikzcd}
			\mc{C} \arrow[rr] \arrow[rd, "p"] &        & \mc{D} \arrow[ld, "q"'] \\
			& \mc{B} &                        
		\end{tikzcd}$$
		preserves cocartesian edges and exhibits $\mc{D}_b\simeq \mc{C}_b[S_b^{-1}]$ for each $b\in\mc{B}$. And for any cocartesian fibration $\mc{E}\to\mc{B}$, the composition induces a fully faithful embedding $$\funct^{\op{coCar}}_{/\mc{B}}(\mc{D},\mc{E})\to \funct^{\op{coCar}}_{/\mc{B}}(\mc{C},\mc{E})$$ whose image consists of those cocartesian functors over $\mc{B}$ sending $S$ to equivalences in $\mc{E}$.
	\end{lem}
	\begin{rem}
		\enu{\item Note that a cocartesian functor $\mc{C}\to\mc{E}$  over $\mc{B}$ sends $S$ to equivalences in $\mc{E}$ if and only if the induced functor on each fiber $\mc{C}_b\to\mc{E}_b$ sends $S_b$ to equivalences in $\mc{E}_b$.
			\item The lemma above is a generalization of \cite[Prop. 2.2.1.9]{ha}. The statement there only gives a construction in the case of reflective localization.
		}
	\end{rem}
	\begin{cor}\label{monlocal}
		Let $\mc{C}^\otimes$ be a symmetric monoidal $\infty$-category and $S$ be a collection of morphisms in $\mc{C}$ satisfying that $f\otimes g\in S$ if both $f,g\in S$. Then the localization $\mc{D}=\mc{C}[S^{-1}]$ inherits a natural symmetric monoidal structure and the localization can be promoted to a symmetric monoidal functor $\cotimes\to\mc{D}^\otimes$ satisfying the universal property that for any symmetric monoidal \infcat $\mc{E}^\otimes$ the composition induces a fully faithful embedding $$\funct^{\otimes}_{/\mathrm{N}(\op{Fin_*})}(\mc{D}^\otimes,\mc{E}^\otimes)\to \funct^{\otimes}_{/\mathrm{N}(\op{Fin_*})}(\mc{C}^\otimes,\mc{E}^\otimes)$$ whose image consists of those symmetric monoidal functors sending $S$ to equivalences in $\mc{E}$.
	\end{cor}
	\begin{lem}\label{rigidbarr}
		Let $\mc{C}^\otimes \xrightarrow{F^\otimes} \mc{D}^\otimes \in \calg(\prl)$. Let $G^\otimes:\mc{D}^\otimes\to\mc{C}^\otimes$ be the relative right adjoint of $F^\otimes$. If $\mc{C}$ is generated by dualizables under small colimits and $G$ is conservative and small-colimit-preserving, then  $G^\otimes$ is symmetric monoidal monadic, i.e. there exist an $R\in\calg(\mc{C})$ and a symmetric monoidal equivalence $\op{Mod}_R(\mc{C})^\otimes\simeq \mc{D}^\otimes$ such that the following diagram is commutative.
		$$
		\begin{tikzcd}
			\op{Mod}_R(\mc{C})^\otimes \arrow[rr, "\sim"] \arrow[rd] &                & \mc{D}^\otimes \arrow[ld, "G^\otimes"] \\
			& \mc{C}^\otimes &                                       
		\end{tikzcd}$$
	\end{lem}
	\begin{proof}
		By \cite[Cor. 4.8.5.21]{ha}, it suffices to show that $G$ satisfies the projection formula, that is, for every object $C\in \mc{C}$ and $D \in\mc{D}$, the canonical map
		$C \otimes G(D)\to G(F(C) \otimes D)$
		is an equivalence. By the assumption, it suffices to verify the case $C$ is dualizable. In this case, for any $M\in\mc{C}$ we have 
		$$
		\begin{aligned}
			\operatorname{Map}_{\mathcal{C}}(M, C \otimes G(D))  \simeq \operatorname{Map}_{\mathcal{C}}\left(C^{\vee} \otimes M, G(D)\right) 
			\simeq \operatorname{Map}_{\mathcal{D}}\left(F\left(C^{\vee} \otimes M\right), D\right) \\
			\simeq \operatorname{Map}_{\mathcal{D}}\left(C^{\vee} \otimes F(M), D\right) 
			\simeq \operatorname{Map}_{\mathcal{D}}(F(M), C \otimes D) 
			\simeq \operatorname{Map}_{\mathcal{C}}(M, G(C \otimes D)) .
		\end{aligned}
		$$
		That indicates the projection formula holds.
	\end{proof}
	\begin{proof}[Proof of  \cref{unilocal}:]
		Let $S_1\subset \funct(\Delta^1,\modu_R(\ageq))$ be the set of morphisms $$\{X_\alpha\otimes_R f_\beta|X_\alpha\in \modu_R(\ageq)^{\op{cproj}}, f_\beta\in S\}.$$ Then $S_1$ is small and thereby generates a strongly saturated class $\overline{S}_1$ of small generation (see \cite[\textsection5.5.4]{htt}). 
		Then  $\overline{S}_1\subset \funct(\Delta^1,\modu_{R}(\ageq))$ satisfies conditions in \cref{monlocal}, thereby it produces a symmetric monoidal localization $\modu_{R}(\ageq)^\otimes\xrightarrow{F^\otimes} \modu_{R}(\ageq)[\overline{S}_1^{\,-1}]^\otimes=D^\otimes$ such that $F^\otimes\in \calg(\prl)$. Since $D\subset \modu_{R}(\ageq)$ closed under finite products, it lies in $\calg(\prlad)$ and $F^\otimes\in \calg(\prlad)$. 
		
		Now we wish to show that $F^\otimes$ satisfies conditions in \cref{rigidbarr}. It suffices to verify that $D\subset \modu_{R}(\ageq)$ closed under small colimits, i.e. $S_1$-local objects are closed under small colimits. Unwinding the definition, a connective $R$-module $M$ is $S_1$-local if and only if $$\mapp_{\modu_R(\ageq)}(X_\alpha\otimes_R Q_\beta,M)\to \mapp_{\modu_R(\ageq)}(X_\alpha\otimes_R P_\beta,M)$$ is equivalent for any $X_\alpha\otimes_R f_\beta\in S_1$. However, this map can be identified with 
		$$\mapp_{\modu_R(\ageq)}(X_\alpha, Q_\beta^\vee\otimes_R M)\to \mapp_{\modu_R(\ageq)}(X_\alpha, P_\beta^\vee\otimes_R M).$$
		So by the projective generation, $M$ is $S_1$-local if and only if $f_\beta^\vee\otimes_R M: Q_\beta^\vee\otimes_R M\to P_\beta^\vee\otimes_R M$ is equivalent for each $f_\beta\in S$. That implies $S_1$-local objects are closed under small colimits. So there exist an $R[S^{-1}]\in\calg(\ageq)_{R/}$ and an equivalence $\op{Mod}_{R[S^{-1}]}(\ageq)^\otimes\simeq \mc{D}^\otimes$ such that the following diagram is commutative.
		$$
		\begin{tikzcd}
			\op{Mod}_{R[S^{-1}]}(\ageq)^\otimes \arrow[rr, "\sim"] \arrow[rd] &                & \mc{D}^\otimes \arrow[ld, "G^\otimes"] \\
			& \op{Mod}_{R}(\ageq)^\otimes &                                       
		\end{tikzcd}$$
		
		Now given $B\in \calg(\mc{A})$, we need to show that the induced map $$\mapp_{\calg(\mc{A})}(R[S^{-1}],B)\to \mapp_{\calg(\mc{A})}(R,B)$$ is a $(-1)$-truncated map whose image on $\pi_0$ consists those maps $R\to B$ such that for any $\beta\in J$,  $B\otimes_RP_\beta\to B\otimes_RQ_\beta$ is an equivalence of $B$-modules. Without loss of generality, we can assume that $B$ is connective. By \cite[Cor. 4.8.5.21]{ha}, we have the following Morita embedding,
		$$\calg(\ageq)\to \calg(\prlad)_{\ageq^\otimes/} $$
		therefore it suffices to show that  $F^\otimes$ induces a fully faithful embedding $$\funct^{\otimes,L}_{/\mathrm{N}(\op{Fin_*})}(\op{Mod}_{R[S^{-1}]}(\ageq)^\otimes,\op{Mod}_{B}(\ageq)^\otimes)\to \funct^{\otimes,L}_{/\mathrm{N}(\op{Fin_*})}(\op{Mod}_{R}(\ageq)^\otimes,\op{Mod}_{B}(\ageq)^\otimes)$$  whose image consists of those functors sending $S$ to equivalences in $\op{Mod}_{B}(\ageq)$, where $\funct^{\otimes,L}_{/\mathrm{N}(\op{Fin_*})}$ denotes symmetric monoidal functors which preserve small colimits. However, that is implied by \cref{monlocal}.
	\end{proof}	
	\begin{rem}\label{invertmapsofdualizables}
		The argument above works for a set $S\subset\funct(\Delta^1,\mc{C}^d)$ of morphisms between dualizables inside an arbitrary presentably symmetric monoidal \infcat $\cotimes$ which is generated by dualizables under small colimits.
	\end{rem}
	We now start to investigate the properties of Cohn localizations.
	\begin{prop}\label{pi0unilocal}
		Let $f_S: R\to R[S^{-1}]\in \calg(\ageq)$ be the Cohn localization at $S$ in \cref{unilocal}. Then the map $\pi_0f_S: \pi_0R\to \pi_0(R[S^{-1}])$ exhibits $\pi_0(R[S^{-1}])\simeq (\pi_0R)[(\pi_0S)^{-1}]$ as the Cohn localization of $\pi_0R$ at $\pi_0S$ in the sense of \cref{unilocal1}, where $\pi_0S=\{\pi_0P_\beta\xrightarrow{\pi_0f_\beta} \pi_0Q_\beta|f_\beta\in S\}$.
	\end{prop}
	\begin{proof}
		It follows immediately from the universal property of the Cohn localization.
	\end{proof}
	
	\begin{prop}\label{unilocalproperties}
		Let $f_S: R\to R[S^{-1}]$ be the Cohn localization at $S$ in \cref{unilocal}. Then   $R[S^{-1}]$ is an idempotent  \ein-$R$-algebra.
	\end{prop}
	\begin{proof}
		It suffices to show that the following diagram is a pushout in $\calg(\ageq)$,
		$$\begin{tikzcd}
			R \arrow[d] \arrow[r] & {R[S^{-1}]} \arrow[d, double, no head] \\ 
			{R[S^{-1}]} \arrow[r, double, no head] & {R[S^{-1}]} 
		\end{tikzcd}$$
		i.e. to show that $f_S$ is an ($\infty$-categorical) epimorphism in $\calg(\ageq)$. That is implied by the description of mapping spaces in \cref{unilocal}.
	\end{proof}
	
	\begin{de}
		Suppose that the $ttt$-\infcat $(\mc{A}^\otimes, \mc{A}_{\geq0}) $ is projectively rigid. We say  a map $A\to B\in \calg(\ageq)$ is a (finitary) Cohn localization if there exists a (finite) set $S$ of morphisms between compact projective $R$-modules such that $B\simeq A[S^{-1}]$.
	\end{de}
	\begin{rem}
		Note that if $S=\{P_i\xrightarrow{f_i} Q_i\}$ is finite, then $A[S^{-1}]\simeq A[f^{-1}]$ is equivalent to the Cohn localization at the single element $f=\bigoplus_i f_i$.
	\end{rem}
	\begin{prop}\label{unilocalfp}
		Suppose that the $ttt$-\infcat $(\mc{A}^\otimes, \mc{A}_{\geq0}) $ is projectively rigid.	Let $ A\to B\in \calg(\ageq)$ be a finitary Cohn localization. Then $B$ is finitely presented over $A$.
	\end{prop}
	\begin{proof}
		By the remark above, we can assume that $S=\{f\}$ consists of a single element. Now given a filtered colimit of connective \ein-$A$-algebras $\varinjlim_\alpha C_\alpha=C$ we need to show that the natural map $$\varinjlim_\alpha \mapp_{\calg(\ageq)_{A/}}(B,C_\alpha)\to \mapp_{\calg(\ageq)_{A/}}(B,C)$$
		is an equivalence. By \cref{unilocalproperties}(1), each mapping space above is empty or a single point. If  $\mapp_{\calg(\ageq)_{A/}}(B,C)=\emptyset$, then nothing needs to prove. 
		
		Now assume that $\mapp_{\calg(\ageq)_{A/}}(B,C)\simeq\{*\}$, we wish to show that there exists an $\alpha$ such that $\mapp_{\calg(\ageq)_{A/}}(B,C_\alpha)$ is not empty. By assumption, the natural map $f\otimes_A C$ is an equivalence, thereby $\op{cofib}(f)\otimes_A C=0$. Since $\op{cofib}(f)$ is a compact $A$-module, there exists an $\alpha$ such that the natural map $\op{cofib}(f)\to \op{cofib}(f)\otimes_A C_\alpha$ is zero. That implies $\op{cofib}(f)\otimes_A C_\alpha=0$ and we are done.
	\end{proof}
	\begin{rem}
		Note that, unlike the case of ($\mathbb{E}_\infty$-)rings, a Cohn localization $R \to R[S^{-1}]$ is not necessarily flat in general. See \cref{exetalenotflt} for an example of a finitary Cohn localization that fails to be flat.
	\end{rem}
	In fact, we can prove that any projectively rigid \infcat is a smashing localization  of some presheaf category induced by a Cohn localization.
	\begin{prop}
		Let $\mc{I}^\otimes$ be a small rigid  (see \cref{smallrig}) symmetric monoidal \infcat which admits finite coproducts and whose tensor product is compatible with finite coproducts. Then the natural symmetric monoidal functor 
		$$\funct(\mc{I}^{\op{op}},\spgeq)^\otimes\xrightarrow{L^\otimes}\funct^\times(\mc{I}^{\op{op}},\spgeq)^\otimes\simeq \mc{P}_{\Sigma}(\mc{I})^\otimes$$ 
		induced by the universal property of Yoneda embedding is a smashing localization, that is, there exists an idempotent   \ein-algebra $R\in \calg(\funct(\mc{I}^{\op{op}},\spgeq))$  such that $L(-)\simeq R\otimes(-)$.
	\end{prop}
	\begin{proof}
		By \cref{rigidbarr}, it only suffices to show that the inclusion $$\funct^\times(\mc{I}^{\op{op}},\spgeq)\subset\funct(\mc{I}^{\op{op}},\spgeq)$$ is closed under small colimits. That is obvious because both sides are additive and the inclusion is closed under finite products and sifted colimits.
	\end{proof}
	\begin{rem}
		In fact, the idempotent algebra $\mb{1}\to R$ above can be identified with the Cohn localization $\mb{1}\to\mb{1}[S^{-1}]$, where $$S=\Bigl\{\bigoplus_{i=1}^n \Sigma_+^\infty h(x_i) \longrightarrow \Sigma_+^\infty h( \coprod_{i=1}^n x_i )\mid x_i\in \mci \text{ for each } i\Bigr\}\subset\funct\Big(\Delta^1, \funct(\mc{I}^{\op{op}},\spgeq)^{\op{cproj}}\Big).$$
	\end{rem}
	\subsection{Cohn localizations in an abelian base}
	To complement and ground our higher $\infty$-categorical construction, this subsection analyzes Cohn localization strictly within a 1-projectively rigid symmetric monoidal Grothendieck abelian category. We demonstrate that finitary Cohn localizations reliably yield finitely presented algebras, mirroring the higher categorical behavior.

	\begin{thm}\label{unilocal1}
		Let $\mb{A}^\otimes$ be a 1-projectively rigid symmetric monoidal  additive 1-category (it is automatically Grothendieck abelian). Let $R\in \calg(\mb{A})$ and $$S=\{P_\beta\xrightarrow{f_\beta} Q_\beta\}$$ be a set of morphisms between compact 1-projective $R$-modules. Then there exists a Cohn localization $R\to R[S^{-1}]\in \calg(\mb{A})$ satisfying the following universal property:\\
		For any $B\in \calg(\mb{A})$, the induced map $$\Hom_{\calg(\mb{A})}(R[S^{-1}],B)\to \Hom_{\calg(\mb{A})}(R,B)$$ is an injection whose image consists those maps $R\to B$ such that for each $f_\beta\in S$,  $B\otimes_RP_\beta\to B\otimes_RQ_\beta$ is an equivalence of $B$-modules.
	\end{thm}
	
	\begin{proof}
		The proof is parallel with the proof of  \cref{unilocal}. Also see  \cref{invertmapsofdualizables}.
		We just need to replace the Morita embedding
		$\calg(\ageq)\to \calg(\prlad)_{\ageq^\otimes/} $ in the argument
		by $$\calg(\mb{A})\to \calg(\prl_{\op{ad},1})_{\mb{A}^\otimes/}$$ to adapt the 1-categorical setting.
	\end{proof}

	\begin{prop}
		Let $f_S: R\to R[S^{-1}]$ be the Cohn localization at $S$ in \cref{unilocal1}. Then $R[S^{-1}]$ is an idempotent commutative $R$-algebra
	\end{prop}
	\begin{proof} It suffices to show that the following diagram is a pushout in $\calg(\mb{A})$,
		$$\begin{tikzcd}
			R \arrow[d] \arrow[r] & {R[S^{-1}]} \arrow[d, double, no head] \\ 
			{R[S^{-1}]} \arrow[r, double, no head] & {R[S^{-1}]} 
		\end{tikzcd}$$
		i.e. to show that $f_S$ is an epimorphism in $\calg(\mb{A})$. That is implied by the description of the Hom set in \cref{unilocal1}.
	\end{proof}
	\begin{de}
		Let $\mb{A}^\otimes$ be a 1-projectively rigid symmetric monoidal Grothendieck abelian category. We say  a map $A\to B\in \calg(\mb{A})$ is a (finitary) Cohn localization if there exists a (finite) set $S=\{P_\beta\xrightarrow{f_\beta} Q_\beta\}$ of morphisms between compact 1-projective $R$-modules such that $B\simeq A[S^{-1}]$.
	\end{de}
	\begin{rem}
		Note that if $S=\{P_i\xrightarrow{f_i} Q_i\}$ is finite, then $A[S^{-1}]=A[f^{-1}]$ can be written as the Cohn localization at a single element $f=\bigoplus_i f_i$.
	\end{rem}
	\begin{prop}
		Let $\mb{A}^\otimes$ be a 1-projectively rigid symmetric monoidal Grothendieck abelian category. Let $A\to B\in \calg(\mb{A})$ be a finitary Cohn localization. Then $B$ is finitely presented over $A$.
	\end{prop}
	\begin{proof}
		The proof is similar to \cref{unilocalfp} but we need to take a different strategy because the kernel is not preserved by base change. 
		
		Let $\mc{A}^\otimes=\mc{D}(\mb{A})^\otimes$. By \cref{dicre} we have that $\mc{A}^{\heartsuit,\otimes}\simeq \mb{A}^\otimes$ and $\ageq^{\op{cproj}}=\mb{A}^{1\text{-}\op{cproj}}$. Let $A'[S^{-1}]$ be the (higher) Cohn localization at $S$ in the sense of \cref{unilocal}. Then $A'[S^{-1}]$ is compact in $\calg(\ageq)_{A/}$ by \cref{unilocalfp}. Therefore $\pi_0A'[S^{-1}]$ is compact in $\calg(\mb{A})_{A/}$. However by \cref{pi0unilocal}, $A[S^{-1}]=\pi_0A'[S^{-1}]$. We are done.
	\end{proof}

	\section{Finiteness properties}\label{finite}
	In this section, we investigate finiteness conditions in the $ttt$-$\infty$-categorical setting. 
	\subsection{Perfect and almost perfect modules}\label{perf}
	In this subsection, we analyze perfect and almost perfect modules, mapping their relationship directly via Tor-amplitude. We demonstrate that an almost perfect module is perfect if and only if it possesses a finite Tor-amplitude, systematically extending  results in \cite[\textsection7.2]{ha} to arbitrary projectively rigid bases.
	\begin{de}
		Let $R\in \alg(\mc{A})$. We say a left $R$-module $M$ is perfect if it is compact in $\lmodu_R(\mc{A})$.
	\end{de}
	
	\begin{prop}\label{perfbounded}
		Suppose that \mca is right complete.	Let $R\in \alg(\ageq)$ and $M$ be a left $R$-module. If $M$ is perfect, then $M$ is bounded-below.
	\end{prop}
	\begin{proof}
		By the right completeness we have $M\simeq \varinjlim \tau_{\geq -n}M$, then the compactness of $M$ implies that $M$ is a retract of $\tau_{\geq -n}M$ for some $n$.
	\end{proof}
	\begin{de}
		Let $\mathcal{C}$ be a presentable $\infty$-category. We will say  an object $C \in \mathcal{C}$ is almost compact if $\tau_{\leq n} C$ is a compact object of $\tau_{\leq n} \mathcal{C}$ for all $n \geq 0$.
	\end{de}
	
	\begin{rem}\label{rem4.12}
		Let $\mathcal{C}$ be a compactly generated $\infty$-category. Then every compact object of $\mathcal{C}$ is almost compact by \textup{\cite[Corollary 5.5.7.4]{htt}.}
	\end{rem}

	\begin{de}
		Suppose that $\mc{A}_{\geq 0}\in\prl_{\omega}$ and that \mca is right complete.	Let $R\in \alg(\mc{A}_{\geq 0})$ be a connective $\mathbb{E}_1$-algebra. We  say that a left $R$-module $M$ is almost perfect if there exists an integer $k$ such that $M \in \lmodu_R(\mc{A})_{\geq k}$ and is almost compact as an object of $\lmodu_R(\mc{A})_{\geq k}$.
		
		We let $\lmodu_R(\mc{A})^{\text{aperf}}\subset \lmodu_R(\mc{A})$  denote the full subcategory spanned by the almost perfect left $R$-modules.
	\end{de}
	\begin{rem}
		Under the assumption that $\mc{A}_{\geq 0}\in\prl_{\omega}$ and \mca is right complete, if a left $R$-module $M$ is almost perfect in $\lmodu_R(\mca)_{\geq k}$, then it is almost perfect in $\lmodu_R(\mca)_{\geq k-1}$ too.
	\end{rem}
	\begin{prop}\label{aper}
		Suppose that $\mc{A}_{\geq 0}\in\prl_{\omega}$ and that \mca is right complete.	Let $R\in \alg(\mc{A}_{\geq 0})$. Then:
		\enu{
			\item The full subcategory $\lmodu_R(\mc{A})^{\mathrm{aperf}} \subset \lmodu_R(\mc{A})$ is closed under translations and finite colimits, and is therefore a stable subcategory of $\lmodu_R(\mc{A})$.
			\item The full subcategory $\lmodu_R(\mc{A})^{\mathrm{aperf}} \subset \lmodu_R(\mc{A})$ is closed under the formation of retracts.
			\item Every perfect left $R$-module is almost perfect.
			\item The full subcategory $\lmodu_R(\mc{A})^{\mathrm{aperf}}_{\geq 0} \subset \lmodu_R(\mc{A})$ is closed under the formation of geometric realizations of simplicial objects.			
		}
		
	\end{prop}
	
	\begin{proof}
		Proof. Assertions (1) and (2) are obvious, and (3) follows from \cref{rem4.12}. To prove (4), it suffices to show that the collection of compact objects of $ \operatorname{LMod}_R(\mc{A})_{[0,n]}$ is closed under geometric realizations, which follows from \cite[Lemma  1.3.3.10]{ha}.
	\end{proof}
	
	\begin{prop}\label{cocompsimplicial}
		Suppose that $\mc{A}_{\geq0}$ is projectively generated. Let $R\in \alg(\mc{A}_{\geq 0})$ and $M\in \lmodu_R(\mc{A})_{\geq0}^{\mathrm{aperf}}$ be a left $R$-module which is connective and almost perfect. Then $M$ can be obtained as the geometric realization of a simplicial left $R$-module $P_{\bullet}$ such that each $P_n$ is a compact projective left $R$-module in $\lmodu_R(\mc{A})_{\geq0}$.
		
		In particular, we obtain an equivalence
		$$\operatorname{LMod}_R(\mathcal{A}_{\geq 0})^{\mathrm{aperf}} \simeq \mathcal{P}_{\sqcup}^{\sqcup, \Delta^{\mathrm{op}}}(\operatorname{LMod}_R(\mathcal{A}_{\geq 0})^{\mathrm{cproj}}),$$
		where the left-hand side is the relative cocompletion obtained by formally adding geometric realizations and finite coproducts while preserving existing finite coproducts.
		Alternatively, we have
		$$  \operatorname{LMod}_R(\mathcal{A}_{\geq 0})^{\mathrm{aperf}}\simeq \mathcal{P}_{\emptyset}^{\Delta^{\mathrm{op}}}(\operatorname{LMod}_R(\mathcal{A}_{\geq 0})^{\mathrm{cproj}}).$$
	\end{prop}
	\begin{proof}
		We mimic the proof in \cite[Prop. 7.2.4.11]{ha} and carefully replace ``free'' by ``projective''. In view of $\infty$-categorical Dold-Kan correspondence, it will suffice to show that $M$ can be obtained as the colimit of a sequence
		$$
		D(0) \xrightarrow{f_1} D(1) \xrightarrow{f_2} D(2) \rightarrow \ldots
		$$
		where each $\operatorname{cofib}(f_n)[-n]$ is a compact projective left $R$-module; here we agree by convention that $f_0$ denotes the zero map $0 \rightarrow D(0)$. The construction goes by induction. Suppose that the diagram
		$$
		D(0) \rightarrow \ldots \rightarrow D(n) \xrightarrow{g} M
		$$
		has already been constructed, and that $N=\mathrm{fib}(g)$ is $n$-connective. Part (1) of \cref{aper} implies that $N$ is almost perfect, so that the bottom $\pi_n N$ is a compact object in the category of left $\pi_0 R$-modules. It follows that there exists a map $\beta: Q[n] \rightarrow N$, where $Q$ is a compact projective left $R$-module because  $\lmodu_R(\mc{A})_{\geq0}$ is projectively generated. 
		And $\beta$ induces a surjection $\pi_0 Q \rightarrow \pi_n N$. We now define $D(n+1)$ to be the cofiber of the composite map $Q[n] \xrightarrow{\beta} N \rightarrow D(n)$, and construct a diagram
		$$
		D(0) \rightarrow \ldots \rightarrow D(n) \rightarrow D(n+1) \xrightarrow{g^{\prime}} M
		$$
		Using the octahedral axiom, we obtain a fiber sequence
		$$
		Q[n] \rightarrow \operatorname{fib}(g) \rightarrow \operatorname{fib}(g^{\prime})
		$$
		and the associated long exact sequence in $\mc{A}^{\heartsuit}$ proves that $\operatorname{fib}(g^{\prime})$ is $(n+1)$-connective.
		In particular, we conclude that for a fixed $m\geq 0$, the maps $\pi_m D(n) \rightarrow \pi_m M$ are isomorphisms for $n \gg 0$, so that the natural map $\varinjlim D(n) \rightarrow M$ is an equivalence of left $R$-modules by the left completeness, as desired.
	\end{proof}
	\begin{prop}\label{perfflat}
		Suppose that the $ttt$-\infcat $(\mc{A}^\otimes, \mc{A}_{\geq0}) $ is projectively rigid. Let $R\in \alg(\mc{A}_{\geq 0})$ and let $M$ be a connective left $R$-module. Then the following are equivalent:
		\enu{
			\item  $M$ is a compact projective left $R$-module.
			\item  $M$ is a perfect and flat left $R$-module.
			\item  $M$ is an almost perfect and flat left $R$-module.
			\item  $M$ is a flat left $R$-module and $\pi_0M$ is a finitely presented  $\pi_0R$-module in the sense of \cref{f.p.}.
		}
	\end{prop}
	\begin{proof}
		The directions $(1)\Rightarrow (2)$,  $(2)\Rightarrow (3)$ and $(3)\Rightarrow (4)$ are obvious. For $(4)\Rightarrow (1)$, by \cref{f.p.flat}, we conclude that $\pi_0M$ is compact 1-projective over $\pi_0R$. Then by \cref{pi0proj}, we get that $M$ is a compact projective left $R$-module.
	\end{proof}

	\begin{de}
		Let $R\in \alg(\mc{A}_{\geq 0})$. We will say  a left $R$-module $M$ has Tor-amplitude $\leq n$ if, for every discrete right $R$-module $N$,  $\pi_i(N \otimes_R M)$ vanish for $i>n$. We will say  $M$ is of finite Tor-amplitude if it has Tor-amplitude $\leq n$ for some integer $n$.
	\end{de}
	\begin{rem}\label{rem7.2.4.22} Suppose that \mca is Grothendieck. In view of \cref{l7}, a connective left $R$-module $M$ has Tor-amplitude $\leq 0$ if and only if $M$ is flat.
		
	\end{rem}
	\begin{prop}
		Suppose that $\mc{A}$ is hypercomplete. Then a connective left $R$-module $M$ has Tor-amplitude $\leq -1$ if and only if $M=0$.
	\end{prop}
	
	\begin{proof}
		Assume \(M\) is connective and has Tor-amplitude \(\leq-1\); we wish to show \(M\simeq0\).
		Since \(R\) is connective, its \(0\)-th truncation
		\[
		\pi_0R=\tau_{\leq0}R
		\]
		is a discrete right \(R\)-module. By the definition of Tor-amplitude \(\leq-1\), we have
		\[
		\pi_i(\tau_{\leq0}R\otimes_R M)=0\quad\text{for all }i>-1\ (\text{i.e. }i\geq0).
		\]
		Because the \(t\)-structure is compatible with the symmetric monoidal structure, the tensor of two connective objects is connective. Since \(\pi_0R\in\mathcal{A}_{\geq0}\) and \(M\in\mathcal{A}_{\geq0}\), their tensor product is connective:
		\[
		\pi_0R\otimes_R M\in\mathcal{A}_{\geq0}.
		\]
		Hence \(\pi_i(\pi_0R\otimes_R M)=0\) for all \(i<0\). This shows all homotopy groups of \(\pi_0R\otimes_R M\) vanish, so by hypercompleteness
		\[
		\pi_0R\otimes_R M\simeq0.
		\]
		Consider the standard truncation fiber sequence in right \(R\)-modules
		\[
		\tau_{\geq 1}R\longrightarrow R\longrightarrow\pi_0R.
		\]
		Tensoring on the right with \(M\) yields a fiber sequence
		\[
		(\tau_{\geq 1}R)\otimes_R M\longrightarrow R\otimes_R M\longrightarrow\pi_0R\otimes_R M.
		\]
		By step 2 the right-hand term vanishes, and \(R\otimes_R M\simeq M\), so the left-hand map is an equivalence:
		\[
		M\simeq(\tau_{\geq 1}R)\otimes_R M.
		\]
		Now use connectivity estimates for tensor products: since \(M\in\mathcal{A}_{\geq0}\) and \(\tau_{\geq 1}R\in\mathcal{A}_{\geq1}\), their tensor lies in \(\mathcal{A}_{\geq1+0}=\mathcal{A}_{\geq1}\), so \(M\in\mathcal{A}_{\geq1}\). Iterating the same argument gives
		\(M\simeq(\tau_{\geq 1}R)\otimes_R M\in\mathcal{A}_{\geq2}\), and so on. By induction
		\(M\in\mathcal{A}_{\geq n}\) for every \(n\geq0\), hence \(\pi_i(M)=0\) for all \(i\in\mathbb{Z}\).
		The hypercompleteness assumption on \(\mathcal{A}\) implies
		\[
		M\simeq0.
		\]
		
	\end{proof}
	\begin{prop}\label{toram}
		Suppose that the $ttt$-\infcat $(\mc{A}^\otimes, \mc{A}_{\geq0}) $ is projectively rigid. Let $R\in \alg(\mc{A}_{\geq 0})$. Then:
		\enu{
			\item  If $M$ is a left $R$-module of Tor-amplitude $\leq n$, then $M[k]$ has Tor-amplitude $\leq n+k$.
			\item Let
			$$
			M^{\prime} \rightarrow M \rightarrow M^{\prime \prime}
			$$
			be a fiber sequence of left $R$-modules. If $M^{\prime}$ and $M^{\prime \prime}$ have Tor-amplitude $\leq n$, then so does M.
			\item Let $M$ be a left $R$-module of Tor-amplitude $\leq n$. Then any retract of $M$ has Tor-amplitude $\leq n$.
			\item Let $M$ be an almost perfect left module over $R$. Then $M$ is perfect if and only if $M$ has finite Tor-amplitude.
			\item Let $M$ be a left module over $R$ having Tor-amplitude $\leq n$. Then for every $N \in\operatorname{RMod}_R(\mc{A})_{\leq 0}$, $\pi_i\left(N \otimes_R M\right)$ vanishes for each $i>n$.
		}
		
	\end{prop}
	\begin{proof}
		We mimic the proof in \cite[Prop. 7.2.4.23]{ha} but carefully replace ``free'' by ``projective''. The first three assertions follow immediately from the exactness of the functor $N \mapsto N \otimes_R M$. It follows that the collection left $R$-modules of finite Tor-amplitude is stable under retracts and finite colimits and desuspensions, and contains all compact projective left $R$-modules. This proves the ``only if'' direction of (4) by \cref{compmod}. For the converse, let us suppose that $M$ is almost perfect and of finite Tor-amplitude. We wish to show that $M$ is perfect. We first apply (1) to reduce to the case where $M$ is connective. The proof now goes by induction on the Tor-amplitude $n$ of $M$. If $n=0$, then $M$ is flat and we may conclude by applying  \cref{perfflat}. We may therefore assume $n>0$.
		
		Since $M$ is almost perfect, there exists a compact projective left $R$-module $P$ and a fiber sequence
		$$
		M^{\prime} \rightarrow P \xrightarrow{f} M
		$$
		where $f$ induces an epimorphism on $\pi_0$. To prove that $M$ is perfect, it will suffice to show that $P$ and $M^{\prime}$ are perfect. It is clear that $P$ is perfect, and it follows from  \cref{aper} that $M^{\prime}$ is almost perfect. Moreover, since $\pi_0f$ is surjective, $M^{\prime}$ is connective. We will show that $M^{\prime}$ is of Tor-amplitude $\leq n-1$; the inductive hypothesis will then imply that $M$ is perfect, and the proof will be complete.
		
		Let $N$ be a discrete right $R$-module. We wish to prove that $\pi_k(N \otimes_R M^{\prime}) \simeq 0$ for $k \geq n$. Since the functor $N \otimes_R -$ is exact, we obtain for each $k$ an exact sequence
		$$
		\pi_{k+1}(N \otimes_R M) \rightarrow \pi_k(N \otimes_R M^{\prime}) \rightarrow \pi_k(N \otimes_R P)
		$$
		The left entry vanishes in virtue of our assumption that $M$ has Tor-amplitude $\leq n$. We now complete the proof of (4) by observing that $\pi_k\left(N \otimes_R P\right)$ vanishes because $N$ is discrete and $P$ is flat and $k\geq n>0$.
		
		We now prove (5). Assume that $M$ has Tor-amplitude $\leq m$. Let $N \in\operatorname{RMod}_R(\mc{A})_{\leq 0}$; we wish to prove that $\pi_i(N \otimes_R M) \simeq 0$ for $i>n$. Since $N \simeq \colimit \tau_{\geq-m} N$, it will suffice to prove the vanishing after replacing $N$ by $\tau_{\geq-m} N$ for every integer $m$. We may therefore assume that $N \in\operatorname{RMod}_R(\mc{A})_{[-m,0]}$ for some $m \geq 0$. We proceed by induction on $m$. When $m=0$, the desired result follows immediately from our assumption on $M$. If $m>0$, we have a fiber sequence
		
		$$
		\tau_{\geq 1-m} N \rightarrow N \rightarrow\left(\pi_{-m} N\right)[-m]
		$$
		hence an exact sequence
		$$
		\pi_i\left(\left(\tau_{\geq 1-m} N\right) \otimes_R M\right) \rightarrow \pi_i\left(N \otimes_R M\right) \rightarrow \pi_{i+m}\left(\pi_{-m} N \otimes_R M\right)
		$$
		If $i>n$, then the first group vanishes by the inductive hypothesis, and the third by virtue of our assumption that $M$ has Tor-amplitude $\leq n$.
	\end{proof}
	\begin{prop}\label{fullstableproj}
		Suppose that the $ttt$-\infcat $(\mc{A}^\otimes, \mc{A}_{\geq0}) $ is projectively rigid. Let $R\in\alg(\ageq)$, and let $\mathcal{C}$ be the smallest stable subcategory\footnote{We don't need to assume that \mcc is idempotent-complete. See \cite[Remark 7.2.4.24]{ha}  for the case of spectra.} of $\operatorname{LMod}_R(\aaa)$ which contains all compact projective modules. Then $\mathcal{C}=\operatorname{LMod}_R^{\text {perf}}(\aaa)$. 
	\end{prop}
	\begin{proof}
		The inclusion $\mathcal{C} \subset \operatorname{LMod}_R^{perf }(\aaa)$ is obvious. To prove the converse, we must show that every object $M \in \operatorname{LMod}_R^{perf }(\aaa)$ belongs to $\mathcal{C}$. Invoking \cref{perfbounded}, we may reduce to the case where $M$ is connective. We then work by induction on the (necessarily finite) Tor-amplitude of $M$. If $M$ is of Tor-amplitude $\leq 0$, then $M$ is flat and the desired result follows from \cref{perfflat}. In the general case, we choose a compact projective $R$-module $P$ and a map $f: P \rightarrow M$ which induces a surjection $\pi_0 P \rightarrow \pi_0 M$. We may conclude that that fiber $K$ of $f$ is a connective perfect module of smaller Tor-amplitude than that of $M$, so that $K \in \mathcal{C}$ by the inductive hypothesis. Since $P \in \mathcal{C}$ and $\mathcal{C}$ is stable under the formation of cofibers, we conclude that $M \in \mathcal{C}$ as desired.
	\end{proof}

	\subsection{Finite presentation and almost of finite presentation}\label{chap7.1}
	We now transition from the  finiteness of modules to the  finiteness of algebras. 
	We prove that finitely presented algebras can be iteratively assembled via finite pushouts of symmetric algebras evaluated on compact projectives.
	\begin{de}\label{einftyfp}
		Let $f:A\to B$ be a map in $\calg(\mc{A}_{\geq0})$.
		\enu{
			\item We say  $f:A\to B$ is locally of finite presentation (or finitely presented) if  $B$ is a compact object in $$\calg(\mc{A}_{\geq0})_{A/}\simeq \calg(\modu_A(\ageq)).$$
			\item We say  $f:A\to B$ is almost of finite presentation if for each $n\geq0$, $\tau_{ \leq n}B$ is a compact object in $$\calg(\mc{A}_{[0,n]})_{\tau_{ \leq n}A/}\simeq \calg(\modu_A(\ageq))_{\leq n}.$$
		}
	\end{de}
	\begin{warn}
		Suppose that $A\in\calg(\mca^\heartsuit)$ is discrete. Let $B \in \calg(\mca^\heartsuit)_{A/}$ be a discrete  \ein-$A$-algebra. Beware that, if $B$ is finitely presented over $A$ in the sense of \cref{ringfp}, this does not imply that $B$ is finitely presented over $A$ in the (derived) sense \cref{einftyfp}; see \cref{fpcounterex} for a counterexample.
	\end{warn}
	\begin{rem}\label{compof.p}
		Suppose that $\ageq\in\prl_{\omega}$ (then so is $\calg(\ageq)$). Given a commutative diagram in $\calg(\mc{A}_{\geq0})$
		$$\begin{tikzcd}
			& A \arrow[ld] \arrow[rd] &   \\
			B \arrow[rr] &                         & C
		\end{tikzcd}$$ where $B$ is of locally of finite presentation over $A$. Then $C$ is locally of finite presentation over $B$ if and only if $C$ is locally of finite presentation over $A$. This follows immediately from \textup{\cite[Proposition 5.4.5.15]{htt}}.
	\end{rem}
	\begin{prop}\label{appro}
		Suppose  that $\mc{A}_{\geq0}$ is projectively generated. Let $f:A\to B\in \calg(\mc{A}_{\geq0})$ be a map such that $\pi_0A\to\pi_0B$ is finitely presented in the sense of \cref{ringfp}. Then there exists compact projective $A$-modules $M,N$ and a diagram 
		$$\begin{tikzcd}
			\op{Sym}^*_A(N) \arrow[r, "\alpha"] \arrow[d, "\phi"] & A \arrow[d] \\
			\op{Sym}^*_A(M) \arrow[r]                & B          
		\end{tikzcd}$$
		such that the map $B'\to B$ induces an isomorphism on $\pi_0$, where $B'$ is the pushout of above diagram in $ \calg(\mc{A}_{\geq0})$ and $\alpha$ is the natural augmentation. (Note that $\phi$  is not necessarily induced by a map of modules $N\to M$).
	\end{prop}
	\begin{proof}
		By \cref{l5.9}, there exists a collection of compact projective
		$A$-modules $\{P_\alpha\}_{\alpha\in I}$ and a map
		\[
		P=\bigoplus_{\alpha\in I}P_\alpha\longrightarrow \pi_0B
		\]
		which is epimorphic on $\pi_0$. By \cref{l5.8}, this map lifts to an
		$A$-module map $P\to B$, and hence induces an $A$-algebra map
		\[
		\operatorname{Sym}_A^*(P)\longrightarrow B
		\]
		which is epimorphic on $\pi_0$.
		For finite subsets $F\subset I$, the images of
		\[
		\pi_0\operatorname{Sym}_A^*
		\left(\bigoplus_{\alpha\in F}P_\alpha\right)
		\longrightarrow \pi_0B
		\]
		form a filtered family of subalgebras whose colimit is $\pi_0B$.
		Since $\pi_0B$ is compact as a commutative $\pi_0A$-algebra, the
		identity of $\pi_0B$ factors through one of these subalgebras. Since
		this subalgebra is contained in $\pi_0B$, it must therefore be equal
		to $\pi_0B$. Thus, for some finite $F_0\subset I$, the map
		\[
		\pi_0\operatorname{Sym}_A^*
		\left(\bigoplus_{\alpha\in F_0}P_\alpha\right)
		\longrightarrow \pi_0B
		\]
		is epimorphic. Set
		$
		M=\bigoplus_{\alpha\in F_0}P_\alpha.
		$
		
		Let
		\[
		N'=\operatorname{fib}\left(\operatorname{Sym}_A^*(M)\to B\right).
		\]
		Since the above map is epimorphic on $\pi_0$, the module $N'$ is
		connective. Again by projective generation, there exists an
		epimorphism on $\pi_0$
		\[
		\bigoplus_{\beta\in J}Q_\beta\longrightarrow N',
		\]
		where every $Q_\beta$ is compact projective.
		For every finite subset $E\subset J$, put
		$
		N_E=\bigoplus_{\beta\in E}Q_\beta.
		$
		The composite $N_E\to N'\to\operatorname{Sym}_A^*(M)$ induces a map
		\[
		\operatorname{Sym}_A^*(N_E)\longrightarrow
		\operatorname{Sym}_A^*(M),
		\]
		and we let
		\[
		B_E=
		A\otimes_{\operatorname{Sym}_A^*(N_E)}
		\operatorname{Sym}_A^*(M).
		\]
		There are natural maps $B_E\to B$, and, after applying $\pi_0$,
		\[
		\colim_E\pi_0B_E\simeq\pi_0B
		\]
		as commutative $\pi_0\operatorname{Sym}_A^*(M)$-algebras.
		
		We claim that $\pi_0B$ is compact as a commutative
		$\pi_0\operatorname{Sym}_A^*(M)$-algebra. Indeed,
		$\pi_0\operatorname{Sym}_A^*(M)$ is finitely presented over $\pi_0A$, since
		$\pi_0M$ is a compact $\pi_0A$-module, while $\pi_0B$ is finitely presented over $\pi_0A$ by
		assumption. Hence the claim follows from the description of mapping
		spaces in the corresponding undercategory.
		
		Therefore the identity of $\pi_0B$ factors through some finite stage
		as a map of commutative
		$\pi_0\operatorname{Sym}_A^*(M)$-algebras:
		\[
		\pi_0B\longrightarrow\pi_0B_{E_0}\longrightarrow\pi_0B.
		\]
		The second map is induced from the quotient map
		\[
		\pi_0\operatorname{Sym}_A^*(M)\longrightarrow\pi_0B.
		\]
		Since the above section is a map over
		$\pi_0\operatorname{Sym}_A^*(M)$, the two quotient maps
		\[
		\pi_0\operatorname{Sym}_A^*(M)\longrightarrow\pi_0B_{E_0},
		\qquad
		\pi_0\operatorname{Sym}_A^*(M)\longrightarrow\pi_0B
		\]
		have the same kernel. Hence
		\[
		\pi_0B_{E_0}\xrightarrow{\sim}\pi_0B.
		\]
		Taking $N=N_{E_0}$, we obtain the desired diagram
		\[
		\begin{tikzcd}
			\operatorname{Sym}_A^*(N) \arrow[r, "\alpha"] \arrow[d, "\phi"]
			& A \arrow[d] \\
			\operatorname{Sym}_A^*(M) \arrow[r]
			& B ,
		\end{tikzcd}
		\]
		and its pushout $B'$ satisfies
		\[
		\pi_0B'\xrightarrow{\sim}\pi_0B.
		\]
	\end{proof}
	
	\begin{rem}
		A natural question arises: given a projective object $P \in \ageq$, is $\pi_0 \operatorname{Sym}^*(P)$ necessarily 1-projective in the heart $\mca^\heartsuit$? 
		However, the answer is negative. Consider the projectively rigid $ttt$-\infcat $$\ageq =\operatorname{Syn}(\mathbb{Z})_{\geq 0}:=  \mathcal{P}_\Sigma(\operatorname{Gr}(\operatorname{Ab})^{1\text{-cproj}})$$ arising in the context of Dirac geometry as introduced by \cite{hesselholt2023dirac}. Let $P = \mathbb{Z}[1] \in \operatorname{Syn}(\mathbb{Z})_{\geq 0}^{\text{cproj}}$, which is a compact projective object. However, its symmetric algebra satisfies $\pi_0 \operatorname{Sym}^*(\mathbb{Z}[1]) \cong \mathbb{Z}[x] / (2x^2)$, which is not a 1-projective object in the heart $\operatorname{Gr}(\operatorname{Ab})$.
	\end{rem}

	\section{Descendable algebras and idempotent algebras}
	
	The goal of this section it to investigate descendable algebras and idempotent algebras in the $ttt$-$\infty$-categorical setting
	\subsection{Faithful algebras}
	
	Descent theory fundamentally requires a rigorous notion of faithful morphisms to ensure that relative tensor products do not obscure or annihilate non-trivial modules. In this subsection, we define faithful and boundedly faithful maps, proving a generalized Nakayama lemma for hypercomplete $ttt$-$\infty$-categories: any nilpotent thickening is automatically boundedly faithful.
	\begin{de}
		Let $f: R\to S \in \alg(\mc{A})$. 
		\enu{
			\item We say  $f$ is left faithful if the base change functor $f_{!}=S\otimes_{R }(-): \lmodu_R(\mc{A})\to \lmodu_S(\mc{A})$ is conservative.
			\item We say  $f$ is  boundedly left  faithful if $f \in \alg(\ageq)$ and the tensor product functor $f_{!}=S\otimes_{R }(-): \lmodu_{R}(\mc{A})^{-}\to \lmodu_{S}(\mc{A})^{-}$ is conservative when restricting on bounded below modules.
		}
	\end{de}
	\begin{cov}
		When $R \to S$ is a map of $\mathbb{E}_1$-algebras, we simply refer to \emph{left faithful} as \emph{faithful}.
	\end{cov}

	\begin{de}
		Let $\alpha:\tilde{A}\to A$ be a map in $\alg(\mc{A}_{\geq0})$. We say  $\alpha$ is a nilpotent thickening if $\pi_0\alpha$ is an epimorphism in the heart and $I=\op{Ker}(\pi_0\alpha)$ is a nilpotent ideal of $\pi_0\tilde{A}$ (i.e. $I^n=0$ for some $n\geq 1$).
	\end{de}
	\begin{prop}[Nakayama lemma]\label{naka}
		Assume that $\mc{A}$ is hypercomplete. Let $\alpha:\tilde{A}\to A$ be a nilpotent thickening in $\alg(\mc{A}_{\geq0})$. Then $\alpha$ is both boundedly left and right  faithful.
	\end{prop}
	\begin{proof}
		We only prove the left bounded faithfulness, as the argument for the right one is similar. Let $I$ denote the 	$\op{Ker}(\pi_0\alpha)$. Suppose that $M\in \lmodu_{\tilde{A}}(\mc{A})^{-}$ is a bounded below module such that $A\otimes_{\tilde{A}}M=0$. We wish to show $M=0$. Now suppose that $M\neq 0$, without loss of generalization, we can assume that $M$ is connective and $\pi_0M\neq 0$. Then $\pi_0A \overline{\otimes}_{\pi_0\tilde{A}}\pi_0M\simeq\tau_{\leq 0}(A\otimes_{\tilde{A}}M)=0 $ where $\overline{\otimes}$ denotes the tensor product in the heart. Therefore $I\cdot \pi_0M=\pi_0M$. However $I$ is nilpotent so $\pi_0M=0$, which leads to a contradiction.
	\end{proof}
	\begin{rem}
		In general, a nilpotent thickening is not faithful, and hence not descendable. A basic counter-example is the truncation map $\mathbb{S}\to H\mathbb{Z}$ of $\mathbb{E}_\infty$-ring spectra, which is a nilpotent thickening but not faithful. Indeed, given a prime $p$ and a natural number $n$, consider the spectrum $K(n)$ of the corresponding Morava K-theory. Then we have $H_{*}(K(n)) =0$, while $\pi_*(K(n)) \neq 0$ \textup{(see \cite[Chap. IX.7.27]{rudyak1998thom}).} That means the base change functor $$\op{Sp}\simeq \modu_{\mathbb{S}}(\op{Sp})\to \modu_{H\mathbb{Z}}(\op{Sp})\simeq \mc{D}(\mathbb{Z})$$ is not conservative.
	\end{rem}

	\subsection{Descendable algebras}
	Building conceptually on the condition of faithfulness, we explore descendable algebras, which structurally generate the entire module category as a thick ideal. By rigorously analyzing the augmented \v{C}ech nerve, we demonstrate that a faithfully flat map whose $\pi_0$-truncation is an $\aleph_n$-compact module automatically is descendable, generalizing a result in \cite{mathew2016galois}. 
	
	\begin{de}
		Let $\mc{C}^\otimes\in \calg(\prlst)$ and $\mc{I}\subset \mc{C}$ be a full subcategory. We say  $\mc{I}$ is a thick ideal of $\mc{C}$ if $\mc{I}\subset\mc{C}$ is a stable subcategory, closed under retractions and the following condition holds:
		For any $x\in \mc{C}$ and $y\in \mc{I}$ we have $x\otimes y\in \mc{I}$.
	\end{de}
	\begin{de}
		Let $\mc{C}^\otimes\in \calg(\prlst)$ and let $f: A\to B \in \calg(\mc{C})$. We say  $f$ is descendable if the smallest thick ideal of $\modu_{A}(\mc{C})$ such that contains $B$ is $\modu_{A}(\mc{C})$ itself.
	\end{de}
	\begin{prop}[See \cite{mathew2016galois} 3.19]\label{descendableimplyfaith}
		Let $\mc{C}^\otimes\in \calg(\prlst)$ and let $f: A\to B \in \calg(\mc{C})$. If $f$ is descendable, then $f$ is faithful.
	\end{prop}
	\begin{de}
		Let $\mc{C}^\otimes\in \calg(\prlst)$ and let $f: A\to B \in \calg(\mc{C})$. We define the augmented cosimplical object 
		$$B^{\bullet+1}: \Delta^+\to \calg(\mc{C})_{A/} $$ be the Cech nerve in $(\calg(\mc{C})_{A/})^{\op{op}}$.
	\end{de}
	\begin{prop}[See \cite{mathew2016galois} 3.20]\label{descendabletotalization}
		Let $\mc{C}^\otimes\in \calg(\prlst)$ and let $f: A\to B \in \calg(\mc{C})$. Then the following conditions are equivalent:
		\enu{
			\item  $A\to B$ is descendable.
			\item $B^{\bullet+1}$ is a $\Delta$-limit diagram in $\op{Pro}(\modu_A(\mc{C}))$.
		}
	\end{prop}
	\begin{proof}
		We first prove that $(1) \Rightarrow(2)$. Let $\mathcal{B}$ denote the full subcategory of $\operatorname{Mod}_A(\mc{C})$ spanned by those objects $M$ for which the canonical map $\theta_M:M \otimes_A A \rightarrow M \otimes_A B^{\bullet+1} $ is an equivalence in $\operatorname{Pro}(\operatorname{Mod}_A(\mc{C}))$. By \cref{moncomp}, we see that $\operatorname{Pro}(\operatorname{Mod}_A(\mc{C}))$ inherits a symmetric monoidal structure that is compatible with limits, and hence $\mathcal{B}$ is a thick ideal of $\operatorname{Mod}_A(\mc{C})$. It will suffice to show that $B \in \mathcal{B}$. This is clear, since $B \otimes_A B^{\bullet+1} $ can be identified with the split cosimplicial object $B^{\bullet+1}$.
		
		Now suppose that (2) is satisfied. Let $\mathcal{D}$ denote the smallest thick ideal of $\operatorname{Mod}_A(\mc{A})$ which contains $B$. Then $B^{\bullet}$ is a cosimplicial object of $\mathcal{D}$ and each term $\operatorname{Tot}^{n}(B / A)$ in the tower $\operatorname{Tot}^{\bullet}(B / A)$ is in $\mc{D}$ too. Assumption (2) implies that $A\simeq\varprojlim_n\operatorname{Tot}^{n}(B / A)$ in $\op{Pro}(\operatorname{Mod}_A(\mc{A}))$. However, $A$ is cocompact in $\op{Pro}(\operatorname{Mod}_A(\mc{A}))$, so $A$ is equivalent to a retract of $\operatorname{Tot}^n(B / A)$ for some integer $n$, so that $A \in \mathcal{D}$. That implies $\mathcal{D}=\operatorname{Mod}_A(\mc{A})$.
	\end{proof}
	\begin{rem}
		From \cref{descendabletotalization}, we see that the descendable condition is stronger than that $B^{\bullet+1}$ is a $\Delta$-limit merely in $\modu_{A}(\mc{A})$.
	\end{rem}
	\begin{lem}\label{finlimst}
		Let $\mathcal{C}$ be a stable $\infty$-category and let $K$ be a finite simplicial set. Given an arbitrary diagram $F: K \to \mathcal{C}$. The limit $\lim_K F$ can be expressed (canonically) as a finite colimit of shifted finite coproducts of the objects $F(k)$. Similarly, the colimit $\colim_K F$ can be expressed (canonically) as a finite limit of shifted finite products of the objects $F(k)$.
	\end{lem}
	\begin{proof}
		We only prove the first statement, because the proof of the second one is totally parallel.
		We compute the limit of $F$ via its Bousfield-Kan cosimplicial replacement (see \cite{ramzi-BousfieldKan}). Consider the cosimplicial object $C^\bullet \in \operatorname{Fun}(\Delta, \mathcal{C})$ whose $n$-th degree term is given by the product over all $n$-simplices of $K$:
		$$ C^n = \prod_{\sigma \in N(K)_n} F(\sigma(n)), $$
		where $N(K)_n$ denotes the set of $n$-simplices of the nerve of $K$, and $\sigma(n) \in K$ is the final vertex of the simplex $\sigma$.
		Since $K$ is a finite diagram, the set of non-degenerate simplices in $N(K)$ is finite, meaning there exists a maximum dimension $d \geq 0$ such that all simplices in $N(K)_n$ for $n > d$ are degenerate. Furthermore, because $\mathcal{C}$ is a stable $\infty$-category, finite products and finite coproducts canonically coincide. Thus, we can rewrite each term as a finite direct sum:
		$$ C^n \simeq \bigoplus_{\sigma \in N(K)_n} F(\sigma(n)). $$
		By the general theory of $\infty$-categorical limits, the limit of the diagram $F$ is canonically equivalent to the limit (the totalization) of the cosimplicial object $C^\bullet$:
		$$ \varprojlim_{K} F \simeq \operatorname{Tot}(C^\bullet) = \varprojlim_{\Delta} C^\bullet. $$
		Because the normalized chain complex associated to the cosimplicial object $C^\bullet$ vanishes in degrees strictly greater than $d$, $C^\bullet$ is a bounded cosimplicial object. In a stable $\infty$-category, the totalization of a bounded cosimplicial object reduces to a finite limit, which by stability is canonically equivalent to a finite colimit of its shifted terms. Specifically, it can be computed as the geometric realization of the associated shifted complex via iterated cofibers:
		$$ \varprojlim_{K} F \simeq \operatorname{colim} \Big( C^d[-d] \to C^{d-1}[-(d-1)] \to \dots \to C^1[-1] \to C^0 \Big), $$
		where the connecting maps are constructed from the alternating sums of the associated coface maps.
		
		Consequently, the limit $\varprojlim_K F$ is obtained purely as a finite colimit of the objects $C^n[-n]$, each of which is a shifted finite direct sum of the original objects evaluated by $F$ in $K$.
	\end{proof}
	\begin{prop}\label{criteriondescendable}
		Let $\mc{C}^\otimes\in \calg(\prlst)$ and let $f: A\to B \in \calg(\mc{C})$. Then the following are equivalent:
		\enu{
			\item $A\to B$ is descendable.
			\item Let $\mcc\subset \modu_A(\mcc)$ be the smallest thick subcategory which contains the essential image of the forgetful functor $\modu_B(\mcc)\to \modu_A(\mca)$. Then
			$\mcc=\modu_A(\mcc)$.
			\item Let $\mcc\subset \modu_A(\mcc)$ be the smallest thick ideal which contains the essential image of the forgetful functor $\modu_B(\mcc)\to \modu_A(\mcc)$. Then
			$\mcc=\modu_A(\mcc)$.
			\item 
			$A$, regarded as an $A$-module, can be obtained as a retract of a finite colimit of a diagram of $A$-modules, each of which can be lifted to a $B$-module.
			\item 
			$A$, regarded as an $A$-module, can be obtained as a retract of a finite limit of a diagram of $A$-modules, each of which can be lifted to a $B$-module.
		}
	\end{prop}
	\begin{proof}
		The implications $(2)\Rightarrow (3)$ and $(4)\Rightarrow (3)$ are clear. By \cref{finlimst}, we obtain $(4)\Leftrightarrow(5)$. The equivalence $(1)\Leftrightarrow (2)$ follows from the same argument as in \cite[Variant D.3.2.3]{sag}. It therefore suffices to prove $(1)\Rightarrow(5)$ and $(3)\Rightarrow(2)$.
		
		$(1)\Rightarrow(5)$ Suppose that $f: A \to B$ is descendable. By the proof of \cref{descendabletotalization}, $A$ is equivalent to a retract of $\operatorname{Tot}^n(B / A)$ for some integer $n\geq0$.  Since the finite totalization $\operatorname{Tot}^n(B / A)$ can be built as a finite limit of the individual cosimplicial terms $B^{k+1} = B \otimes_A \dots \otimes_A B$ ($0 \leq k \leq n$). Since each $B^{k+1}$ naturally admits the structure of a $B$-module, $A$ is indeed a retract of a finite limit of objects that can be lifted to $B$-modules.
		
		$(3)\Rightarrow(2)$ Let $\mcb\subset \modu_A(\mcc)$ be the smallest thick subcategory which contains the essential image of the forgetful functor $\modu_B(\mcc)\to \modu_A(\mcc)$. We show that $\mcb$ has actually been a thick ideal. Let $\mcb'\subset\mcb$ denote the full subcategory spanned by those objects $M$ such that $N\otimes_A M \in \mcb$ for any $A$-module $N$. Since $\mcb'$ is a thick subcategory too, it suffices to show the case when $M$ admits a $B$-module structure $_BM$. In this case $N\otimes_A M\simeq {}_A((N\otimes_A B) \otimes_{B} M)$ admits a $B$-module structure, hence $M\in \mcb'$. Consequently, $\mcb=\mcb'$ is a thick ideal.
	\end{proof}
	
	\begin{lem}[See \cite{mathew2016galois} 3.21, 3.24]\label{descendable321}
		Let $\mc{C}^\otimes\in \calg(\prlst)$. 
		\begin{enumerate}
			\item Let $f: R\to S$ be a descendable morphism in $\calg(\mc{C})$ and $R\to A$ be another map in $\calg(\mc{C})$. Then the map $A\to A\otimes_R S$ given by the following pushout diagram is descendable.
			$$
			\begin{tikzcd}
				R \arrow[d] \arrow[r]    & S \arrow[d]    \\
				A \arrow[r] & A\otimes_R S
			\end{tikzcd}
			$$
			\item 
			Let $A \rightarrow B \rightarrow C$ be maps in $\operatorname{CAlg}(\mathcal{C})$. Then the following hold:
			\begin{enumerate}
				\item  If $A \rightarrow B$ and $B \rightarrow C$ are descendable, so is $A \rightarrow C$.
				\item  If $A \rightarrow C$ is descendable, so does $A \rightarrow B$.
			\end{enumerate}
		\end{enumerate}

	\end{lem}

	\begin{lem}\label{D.3.3.7}
		Suppose that the $ttt$-\infcat $(\mc{A}^\otimes, \mc{A}_{\geq0}) $ is projectively rigid.  Let $A\in\alg(\ageq)$, let $M$ be a flat left $A$-module, and let $N$ be a connective left $A$-module. Assume that $\pi_0 M$ is an $\aleph_n$-compact object of the category of discrete $\pi_0 A$-modules for some $n\geq0$. Then $\operatorname{Ext}_A^m(M, N) \simeq 0$ for $m>n$.
	\end{lem}
	\begin{proof}
		The Postnikov-tower argument in
		\cite[Lemma D.3.3.7]{sag} applies verbatim and reduces us to the
		case where $N$ is discrete. By base change along
		$A\to\pi_0A$, we may further reduce to the case where $A$ and $M$
		are discrete.
		
		By \cref{lazard1}, $M$ is a filtered colimit of compact
		1-projective $A$-modules. Since $M$ is $\aleph_n$-compact, it is a
		retract of an $\aleph_n$-small filtered colimit
		\[
		\colim_{\alpha\in P}M_\alpha
		\]
		of compact 1-projective modules. Hence
		$\operatorname{Map}_A(M,N[m])$ is a retract of
		\[
		\varprojlim_{\alpha\in P}
		\operatorname{Map}_A(M_\alpha,N[m]).
		\]
		For $m>n$, each mapping space on the right is $n$-connective by
		projectivity. Therefore
		\cite[Lemma D.3.3.6]{sag} implies that the inverse limit is
		connected. It follows that
		\[
		\operatorname{Ext}^m_A(M,N)=0
		\]
		for $m>n$.
	\end{proof}
	\begin{thm}\label{ffdescendable}
		Suppose that the $ttt$-\infcat $(\mc{A}^\otimes, \mc{A}_{\geq0}) $ is projectively rigid.  Let $\phi: A\to B \in \calg(\mc{A})$ be a faithfully flat map such that $\pi_0B$ is a $\aleph_n$-compact $\pi_0A$-module\footnote{It means $\pi_0B\in\modu_{\pi_0A}(\mca^\heartsuit)^{\aleph_n}$.} for some $n\geq 0$. Then $\phi$ is descendable.
	\end{thm}
	\begin{proof}
		The argument is parallel to the proof of
		\cite[Proposition D.3.3.1]{sag}. Since $B$ is flat over $A$,
		\cref{flatpush}(1) gives a canonical equivalence
		$
		B\simeq
		A\otimes_{\tau_{\geq0}A}\tau_{\geq0}B.
		$
		By the stability of descendability under base change
		(\cref{descendable321}), it is therefore enough to prove that
		\[
		\tau_{\geq0}A\longrightarrow\tau_{\geq0}B
		\]
		is descendable. Replacing $\phi$ by this map, we may assume that
		$A$ and $B$ are connective.
		
		Let $\mathcal{C}\subset\modu_A(\mc{A})$ be the smallest thick
		subcategory which contains every object of the
		form $M\otimes_A B$. It suffices to prove that $A\in\mathcal{C}$.
		
		Let
		$
		K=\operatorname{fib}(\phi)
		$
		and denote by $\rho:K\to A$ the canonical map. For every $m\geq0$,
		let
		\[
		\rho(m):K^{\otimes m}\longrightarrow A
		\]
		be the $m$th tensor power of $\rho$ in the symmetric monoidal \infcat
		$\modu_A(\mc{A})$, with $\rho(0)=\operatorname{id}_A$. The map
		$\rho(m+1)$ factors as
		\[
		K^{\otimes(m+1)}
		\simeq K^{\otimes m}\otimes_AK
		\xrightarrow{\operatorname{id}_{K^{\otimes m}}\otimes\rho}
		K^{\otimes m}
		\xrightarrow{\rho(m)}
		A.
		\]
		Since $\operatorname{cofib}(\rho)\simeq B$, the octahedral axiom gives
		a cofiber sequence
		\[
		K^{\otimes m}\otimes_A B
		\longrightarrow
		\operatorname{cofib}(\rho(m+1))
		\longrightarrow
		\operatorname{cofib}(\rho(m)).
		\]
		The first term belongs to $\mathcal{C}$, and
		$\operatorname{cofib}(\rho(0))\simeq0$. It follows by induction that
		\[
		\operatorname{cofib}(\rho(m))\in\mathcal{C}
		\qquad
		\text{for every }m\geq0.
		\]
		Consequently, to prove that $A \in \mathcal{C}$, it suffices to show that $A$ is a retract of $\operatorname{cofib}(\rho(m))$ for some $m \geq 0$. This holds whenever the homotopy class of $\rho(m)$ vanishes, regarded as an element of
		\[
		\operatorname{Ext}_A^0\left(K^{\otimes m}, A\right)
		\simeq
		\operatorname{Ext}_A^m\left((\Sigma K)^{\otimes m}, A\right).
		\]
		By \cref{pi0flat}(4), $\Sigma K \simeq \operatorname{cofib}(\phi)$ is a flat $A$-module. Since flat $A$-modules are closed under tensor products by \cref{l1.3}(2), $(\Sigma K)^{\otimes m}$ is flat for every $m>0$. Moreover, $\pi_0\Sigma K$ is an $\aleph_n$-compact $\pi_0A$-module, since it is the cokernel of a map between $\aleph_n$-compact $\pi_0A$-modules. Since $\aotimes$ is projectively rigid, we have
		\[
		\modu_{\pi_0A}(\mca^\heartsuit)^{\otimes}\in\calg(\prl_{\omega}),
		\]
		and hence, in particular,
		\[
		\modu_{\pi_0A}(\mca^\heartsuit)^{\otimes}\in\calg(\prl_{\aleph_n}).
		\]
		It follows that $\pi_0\bigl((\Sigma K)^{\otimes m}\bigr)\simeq(\pi_0\Sigma K)^{\otimes m}$ is an $\aleph_n$-compact $\pi_0A$-module for every $m>0$. Therefore, by \cref{D.3.3.7},
		\[
		\operatorname{Ext}_A^m\left((\Sigma K)^{\otimes m}, A\right)=0
		\]
		for every $m>n$. Choosing an $m>n$, the homotopy class of $\rho(m)$ therefore vanishes.
		Hence $A$ is a retract of $\operatorname{cofib}(\rho(m))\in\mathcal C$,
		so $A\in\mathcal C$. Thus $\phi$ is descendable.
	\end{proof}
	
	\subsection{Almost algebra theory}

	We present an application of our framework by generalizing the higher almost ring theory introduced by Hebestreit and Scholze \cite{hebestreit2024note}. In this subsection, we completely classify $\pi_0$-epimorphic idempotent  \ein-algebras via idempotent ideals.
	We now introduce the main result of this subsection. 
	\begin{thm}\label{almostalg}
		Assume that $\mca$ is both right complete and left complete. Let $R \in \calg(\ageq)$. Consider the full subcategory $\mathrm{LQ}_R$ of $\mathrm{CAlg}(\ageq)_{R/}$ spanned by the maps $\varphi: R \rightarrow S$ for which
		\enu{ \item the multiplication $S \otimes_R S \rightarrow S$ is an equivalence, i.e. $\varphi$ is idempotent,
			\item  $\pi_0(\varphi): \pi_0 R \rightarrow \pi_0 S$ is epimorphic in $\mca^\heartsuit$.}
		Then the functor
		$$
		\mathrm{LQ}_R \longrightarrow\left\{I \subset \pi_0 R \mid I^2=I\right\}, \quad \varphi \longmapsto \operatorname{Ker}(\pi_0 \varphi)
		$$
		is an equivalence of categories, where we again regard the target as a poset via the inclusion ordering. The inverse image of some $I \subset \pi_0(R)$ can be described more directly as $R / I^{\infty}$, where
		$$
		I^{\infty}=\lim _{n \in \mathbb{N}^{\mathrm{op}}} J_I^{\otimes_R n}
		$$
		with $J_I \rightarrow R$ the fibre of the canonical map $R \rightarrow H(\pi_0(R) / I)$. Furthermore, this inverse system stabilises on $\pi_i$ for $n>i+1$. 
		
		Moreover, the image of the fully faithful restriction functor $\operatorname{Mod}_{R / I^{\infty}}(\mca)\rightarrow \operatorname{Mod}_R(\mca)$ consists exactly of those modules whose homotopy is killed by $I$.
	\end{thm}
	\begin{proof}
		
		We only do slight modifications of original proof in \cite{hebestreit2024note} to fit into our framework. Firstly we observe that $\mathrm{LQ}_R$ is indeed equivalent to a poset \cite[see][Proposition 4.8.2.9]{ha}. Let us also immediately verify that $\operatorname{Ker}(\pi_0 \varphi)$ is indeed idempotent for $\varphi: R \rightarrow S\in \op{LQ}_R$. Tensoring the fibre sequence $F \rightarrow R \rightarrow S$ with $F$ gives
		$$
		F \otimes_R F \longrightarrow F \longrightarrow F \otimes_R S
		$$
		and the right hand term vanishes since one has a fibre sequence
		$$
		F \otimes_R S \longrightarrow R \otimes_R S \longrightarrow S \otimes_R S
		$$
		whose right hand map (after identifying $R \otimes_R S \simeq S$) is a section of the multiplication $S \otimes_R S \rightarrow S$ and thus an equivalence. But $F$ is connective and the map $\pi_0(F) \rightarrow \operatorname{Ker}(\pi_0 \varphi)$ surjective by the long exact sequence of $\varphi$, whence a chase in the diagram
		shows that the multiplication $$\operatorname{Ker}\left(\pi_0 \varphi\right) \otimes_{\pi_0 R} \operatorname{Ker}\left(\pi_0 \varphi\right) \rightarrow \operatorname{Ker}\left(\pi_0 \varphi\right)$$ is surjective as desired.
		
		Next, we verify that the inverse system $J_I^{\otimes_R n}$ stabilises degreewise. In fact we show slightly more, namely that the cofibre $R / J_I \otimes_R J_I^{\otimes{ }_R n}$ of the canonical map $J_I^{\otimes_R n+1} \rightarrow J_I^{\otimes_R n}$ is $n$-connective. Since $R / J_I=\mathrm{H}(\pi_0(R) / I)$ is annihilated by $I$, we immediately deduce that the homotopy groups of this cofibre are annihilated by $I$ (from both sides). Now, for $n=0$, the connectivity claim is clear, and if we inductively assume that $R / J_I \otimes_R J_I^{\otimes{ }_R n}$ is $n$-connective, then
		$$
		R / J_I \otimes_R J_I^{\otimes_R n+1}=\left(R / J_I \otimes_R J_I^{\otimes_R n}\right) \otimes_R J_I
		$$
		is clearly also $n$-connective and its $n$th homotopy group is $\pi_n(R / J_I \otimes_R J_I^{\otimes_R^n}) \otimes_{\pi_0 R} I$. Since the left hand term is annihilated by $I$, we compute
		$$
		\begin{gathered}
			\pi_n\left(R / J_I \otimes_R J_I^{\otimes_R n}\right) \otimes_{\pi_0 R} I=\pi_n\left(R / J_I \otimes_R J_I^{\otimes_R n}\right) \otimes_{\pi_0(R) / I} \pi_0(R) / I \otimes_{\pi_0 R} I \\
			=\pi_n\left(R / J_I \otimes_R J_I^{\otimes_R n}\right) \otimes_{\pi_0(R) / I} I / I^2=0 
		\end{gathered}
		$$ which complete the induction.
		
		Now we claim
		$$
		\left(\lim _{n \in \mathbb{N}^{\text {op }}} J_I^{\otimes_R n}\right) \otimes_R J_I \longrightarrow \lim _{n \in \mathbb{N}^{\text {op }}} J_I^{\otimes_R n}
		$$
		is an equivalence. Since the limit stabilises degreewise and $J_I$ is connective, we can move the limit out of the tensor product by \cref{limtower} (the cofibre of the interchange map is a limit of terms with growing connectivity), and then the statement follows from finality. 	By the same argument, for any $n\geq 1$ the canonical map  $$I^\infty \otimes_R J_I^{\otimes_R n}\to I^\infty$$ is equivalent too and hence $I^\infty \otimes_RI^\infty \to I^\infty$ is an equivalence---that is to say $R\to R/I^\infty$ is idempotent and hence produces an element in $\calg(\ageq)^{\op{idem}}_{R/}$.

		As the next step, we show that the tautological map $M=R \otimes_R M \rightarrow R / I^{\infty} \otimes_R M$ is an equivalence if and only if the homotopy of $M$ is annihilated by $I$, or in other words that $I^{\infty} \otimes_R M \simeq 0$. For the ``only if'' direction, it suffices to observe that $
		\pi_0(R/I^\infty)=\pi_0R/I$. For ``if'' direction, 
		we start with the simplest case $M=R / J_I$, where the claim was proved above. For an arbitrary $R$-module $M$ concentrated in degree 0 and killed by the action of $I$, it naturally inherits a $\op{H}(\pi_0R/I)$-module structure  and hence we get a retraction of $R$-modules $$I^\infty\otimes_RM\to I^\infty\otimes_RM\otimes_R H(\pi_0R/I)\to I^\infty\otimes_RM .$$ However the middle one is zero  by commutativity of $\otimes_R$ (note $\op{H}(\pi_0R/I)=R/J_I$), so $I^\infty\otimes_RM=0$.

		For bounded below $M$, we have
		$$
		I^{\infty} \otimes_R M \simeq I^{\infty} \otimes_R\left(\lim _{k \in \mathbb{N}^{\mathrm{op}}} \tau_{\leq k} M\right) \simeq \lim _{k \in \mathbb{N}^{\circ \mathrm{p}}} I^{\infty} \otimes_R \tau_{\leq k} M \simeq 0
		$$
		by commuting the limit out using the same argument as above. Finally, for arbitrary $M$ whose homotopy groups are killed by $I$, we find
		$$
		I^{\infty} \otimes_R M \simeq I^{\infty} \otimes_R\left(\operatorname{colim}_{k \in \mathbb{N}} \tau_{\geq-k} M\right) \simeq \operatorname{colim}_{k \in \mathbb{N}} I^{\infty} \otimes_R \tau_{\geq-k} M \simeq 0 .
		$$

		So combined with the idempotent property of $R\to R/I^\infty$, we learn that the image of the fully faithful restriction functor $\operatorname{Mod}_{R / I^{\infty}}(\mca)\rightarrow \operatorname{Mod}_R(\mca)$ consists exactly of those modules whose homotopy is killed by $I$, as desired.

		Finally, we are ready to verify that the construction $I \mapsto\left(R \rightarrow R / I^{\infty}\right)$ induces an inverse to taking kernels. The composition starting with an ideal is clearly the identity. So we are left to show that for every $\varphi: R \rightarrow S$  in $\mathrm{LQ}_R$ with $I=\operatorname{Ker}(\pi_0 \varphi)$, the canonical map $\psi: R / I^{\infty} \rightarrow S$, arising from the homotopy of $S$ being annihilated by $I$, is an equivalence. Per construction it induces an equivalence on $\pi_0$. By \cref{naka}, the functor $\psi_{!}=S \otimes_{R / I^{\infty}}-: \operatorname{Mod}_{R / I^{\infty}}(\mca) \rightarrow \operatorname{Mod}_S(\mca)$ is thus conservative when restricted to bounded below modules. But the map
		$$
		S \simeq \psi_{!}\left(R / I^{\infty}\right) \xrightarrow{\psi_{!}(\varphi)} \psi_{!}(S)=S \otimes_{R / I^{\infty}} S \simeq S \otimes_R S
		$$
		is induced by the unit and thus an equivalence since $\varphi$ is idempotent.
	\end{proof}

	We can also prove the non-commutative version of \cref{almostalg}, which involves some different techniques about associative non-unital algebras.
	\begin{lem}\label{initialnualg}
		Let $\cotimes\in\alg(\Cat_\infty^\emptyset)$ be a monoidal \infcat compatible with the empty colimit. Then the forgetful functor $G:\alg^{\op{nu}}(\mcc)\to\mcc$ creates the empty colimit.
	\end{lem}
	
	\begin{proof}
		We can compute the relative left Kan extension of $\emptyset\in \mcc^\otimes$ along the inclusion of the trivial operad into the operad of non-unital algebras $*\simeq\{[1]\}\hookrightarrow \Delta_{s,\geq 1}^\opp$, as follows:
		$$\begin{tikzcd}
			* \arrow[d, hook] \arrow[r, "\emptyset"]              & \mcc^\otimes \arrow[d] \\
			{\Delta_{s,\geq 1}^\opp} \arrow[r] \arrow[ru, dashed, "A"] & \Delta^\opp           
		\end{tikzcd}$$
		Because $\cotimes$ is compatible with the empty colimit, this relative left Kan extension evaluates to the initial object, giving $A(\left<n\right>) = \emptyset$ for each $n\geq1$. By the property of the relative Kan extension, $A$ is the initial object of $\funct_{/\Delta^\opp}^{\op{inert}}(\Delta_{s,\geq 1}^\opp,\cotimes)$. Since $A$ lies in $\alg^{\op{nu}}(\mcc)\subset \funct_{/\Delta^\opp}^{\op{inert}}(\Delta_{s,\geq 1}^\opp,\cotimes)$, we see that $A$ is initial in
		$\alg^{\op{nu}}(\mcc)$, whose underlying object is $\emptyset \in \mcc$. Consequently $G$ creates the empty colimit, because $G$ is conservative and preserves the empty colimit.
	\end{proof}
	\begin{lem}\label{nualgebraker}
		Let $\cotimes\in\alg(\Cat_\infty^\emptyset)$ be a monoidal \infcat compatible with the empty colimit. Suppose that \mcc is pointed and admits finite limits. Let $S\in\alg(\mcc)$ be an $\mathbb{E}_1$-algebra. Then the kernel $$I=\ker(\mb{1}\to S)\xrightarrow{i} \mb{1}$$ is a non-unital $\mathbb{E}_1$-algebra satisfying that $\mu_I\simeq 1\otimes i\simeq i\otimes 1$.
	\end{lem}
	\begin{proof}
		Note that  $\alg^{\op{nu}}(\mcc)$ creates any limit that exists in $\mcc$ (see \cite[\textsection3.2.2]{ha}). By \cref{initialnualg}, $0$ is the initial object in $\alg^{\op{nu}}(\mcc)$ and hence a zero object in it. Also, $I$ admits a natural non-unital algebra structure and can be identified with the fiber of $\mb{1}\to S$.  Considering the following commutative diagram:
		$$
		\begin{tikzcd}
			I\otimes I \arrow[rd] \arrow[rr] \arrow[dd] \arrow[dddd, "\mu_I"', dashed, bend right] &                                                                & 0\otimes 0 \arrow[rd] \arrow[dd] &                                            \\
			& \mb{1}\otimes\mb{1} \arrow[rr] \arrow[dd, Rightarrow, no head] &                                  & S\otimes S \arrow[dd, Rightarrow, no head] \\
			I\otimes \mb{1} \arrow[rr] \arrow[rd] \arrow[dd]                                       &                                                                & 0\otimes S \arrow[rd] \arrow[dd] &                                            \\
			& \mb{1}\otimes\mb{1} \arrow[rr] \arrow[dd, "\mu_{\mb{1}}"]      &                                  & S\otimes S \arrow[dd, "\mu_S"]             \\
			I \arrow[rd] \arrow[rr]                                                                &                                                                & 0 \arrow[rd]                     &                                            \\
			& \mb{1} \arrow[rr]                                              &                                  & S                                         
		\end{tikzcd}$$
		The back face represents the limits defining the zero object over $S$, while the top and front faces involve the limit definition of $I$. Because the target of $I \otimes I$ over $S$ is zero, the universal property of the limit induces the dashed map $\mu_I: I \otimes I \to I$. Furthermore, the commutativity of the left and bottom prisms forces this multiplication to factor through the left and right unit maps. Therefore, we naturally obtain the equivalences $\mu_I \simeq 1 \otimes i \simeq i \otimes 1$, making $I$ a non-unital $\mathbb{E}_1$-algebra with this trivialized multiplication.
	\end{proof}
	\begin{lem}
		Let $\phi:R\to S\in\alg(\mca)$. Then the base change functor $\lmodu_{R}(\mca)\to \lmodu_{S}(\mca)$ is a localization if and only if the multiplication $S \otimes_R S \rightarrow S$ is an equivalence (as $R$-$R$ bimodules).
	\end{lem}
	\begin{proof}
		Let $F: \lmodu_R(\mca) \to \lmodu_S(\mca)$ denote the base change functor given by $M \mapsto S \otimes_R M$, and let $G: \lmodu_S(\mca) \to \lmodu_R(\mca)$ be its right adjoint, the restriction of scalars functor. The left adjoint $F$ is a localization if and only if its right adjoint $G$ is fully faithful, which is equivalent to the condition that the counit of the adjunction, $\epsilon_M: F(G(M)) \to M$, is an equivalence for all $S$-modules $M$.
		
		$(\Longrightarrow)$ If $F$ is a localization, $\epsilon_M$ is an equivalence for all $M$. Evaluating the counit at the free module $M = S$, we obtain the multiplication map:
		$$ S \otimes_R S \simeq F(G(S)) \xrightarrow{\epsilon_S} S $$
		which must therefore be an equivalence.
		
		$(\Longleftarrow)$ Conversely, suppose the multiplication map $S \otimes_R S \to S$ is an equivalence. For any $S$-module $M$, we can express $M$ canonically as the relative tensor product $S \otimes_S M$. The counit map applied to $M$ corresponds to the action map $S \otimes_R M \to M$. We can rewrite the domain as:
		$$ S \otimes_R M \simeq S \otimes_R (S \otimes_S M) \simeq (S \otimes_R S) \otimes_S M $$
		Since $S \otimes_R S \xrightarrow{\simeq} S$ by assumption, this reduces canonically to $S \otimes_S M \simeq M$. Thus, the counit is an equivalence for all $S$-modules $M$, meaning $G$ is fully faithful and $F$ is a localization.
	\end{proof}
	\begin{thm}\label{almostalg2}
		Assume that $\mca$ is both right complete and left complete. Let $R \in \alg(\ageq)$. Consider the full subcategory $\mathrm{LQ}_R$ of $\mathrm{Alg}(\ageq)_{R/}$ spanned by $\mathbb{E}_1$-maps $\varphi: R \rightarrow S$ for which
		\enu{ \item the multiplication $S \otimes_R S \rightarrow S$ is an equivalence, i.e. $\varphi$ is idempotent,
			\item  $\pi_0(\varphi): \pi_0 R \rightarrow \pi_0 S$ is epimorphic in $\mca^\heartsuit$.}
		
		Then the functor
		$$
		\mathrm{LQ}_R \longrightarrow\left\{\text{two-sided ideals } I \subset \pi_0 R \mid I^2=I\right\}, \quad \varphi \longmapsto \operatorname{Ker}(\pi_0 \varphi)
		$$
		is an equivalence of categories, where we again regard the target as a poset via the inclusion ordering. The inverse image of some $I \subset \pi_0(R)$ can be described more directly as $R / I^{\infty}$, where
		$$
		I^{\infty}=\lim _{n \in \mathbb{N}^{\mathrm{op}}} J_I^{\otimes_R n}
		$$
		with $J_I \rightarrow R$ the fibre of the canonical map $R \rightarrow H(\pi_0(R) / I)$. Furthermore, this inverse system stabilises on $\pi_i$ for $n>i+1$. 
		
		The image of the fully faithful restriction functor $\operatorname{LMod}_{R / I^{\infty}}(\mca)\rightarrow \operatorname{LMod}_R(\mca)$ consists exactly of those modules whose homotopy is killed by $I$, as desired.
	\end{thm}
	\begin{proof}
		
		The non-commutative case is a bit more tricky. Firstly we observe that $\mathrm{LQ}_R$ is indeed equivalent to a poset \cite[see][Proposition 4.8.2.9]{ha}. Let us also immediately verify that $\operatorname{Ker}(\pi_0 \varphi)$ is indeed idempotent for $\varphi: R \rightarrow S\in \op{LQ}_R$. Relatively tensoring the cofiber sequence (of $R$-$R$-bimodules) $F \rightarrow R \rightarrow S$ with $F$ gives
		$$
		F \otimes_R F \longrightarrow F \longrightarrow F \otimes_R S
		$$
		and the right hand term vanishes since one has a cofiber sequence (of $R$-$R$-bimodules)
		$$
		F \otimes_R S \longrightarrow R \otimes_R S \longrightarrow S \otimes_R S
		$$
		whose right hand map (after identifying $R \otimes_R S \simeq S$) is a section of the multiplication $S \otimes_R S \rightarrow S$ and thus an equivalence. But $F$ is connective and the map $\pi_0(F) \rightarrow \operatorname{Ker}(\pi_0 \varphi)$ surjective by the long exact sequence from $\varphi$, whence a chase in the diagram
		shows that the multiplication $$\operatorname{Ker}\left(\pi_0 \varphi\right) \otimes_{\pi_0 R} \operatorname{Ker}\left(\pi_0 \varphi\right) \rightarrow \operatorname{Ker}\left(\pi_0 \varphi\right)$$ is surjective as desired.
		
		\footnote{This paragraph is the main difference from the argument of \cref{almostalg}.}Next, we verify that the inverse system $J_I^{\otimes_R n}$ (with the relative tensor product $\otimes_R$) stabilises degreewise. In fact we show slightly more, namely that the cofibre $R / J_I \otimes_R J_I^{\otimes{ }_R n}$ of the canonical map $J_I^{\otimes_R n+1} \rightarrow J_I^{\otimes_R n}$ is $n$-connective. Now, for $n=0$, the connectivity claim is clear, and if we inductively assume that $R / J_I \otimes_R J_I^{\otimes{ }_R n}$ is $n$-connective, then
		$$
		R / J_I \otimes_R J_I^{\otimes_R n+1}=\left(R / J_I \otimes_R J_I^{\otimes_R n}\right) \otimes_R J_I
		$$
		is clearly also $n$-connective and its $n$th homotopy group is $\pi_n(R / J_I \otimes_R J_I^{\otimes_R^n}) \otimes_{\pi_0 R} I$. By \cref{nualgebraker}, $J_I\in \alg^{\op{nu}}(_R\op{BMod}_R(\mca))$ and we have the following commutative diagram
		$$
		\begin{tikzcd}
			J_I\otimes_R J_I \arrow[d, Rightarrow, no head] \arrow[r, "i\otimes1"] & R\otimes_R J_I \arrow[rd, "\sim"] &     \\
			J_I\otimes_R J_I \arrow[r, "1\otimes i"]                               & J_I\otimes_R R \arrow[r, "\sim"]  & J_I
		\end{tikzcd}$$
		hence
		$R/J_I\otimes_R J_I\simeq J_I\otimes_R R/J_I$ in $_R\op{BMod}_R(\mca)$.  Therefore the left hand term $R / J_I \otimes_R J_I^{\otimes_R n}$ is annihilated by $I$ from both sides, we compute
		$$
		\begin{gathered}
			\pi_n\left(R / J_I \otimes_R J_I^{\otimes_R n}\right) \otimes_{\pi_0 R} I=\pi_n\left(R / J_I \otimes_R J_I^{\otimes_R n}\right) \otimes_{\pi_0(R) / I} \pi_0(R) / I \otimes_{\pi_0 R} I \\
			=\pi_n\left(R / J_I \otimes_R J_I^{\otimes_R n}\right) \otimes_{\pi_0(R) / I} I / I^2=0 
		\end{gathered}
		$$ which complete the induction.
		
		Now we claim
		$$
		\left(\lim _{n \in \mathbb{N}^{\text {op }}} J_I^{\otimes_R n}\right) \otimes_R J_I \longrightarrow \lim _{n \in \mathbb{N}^{\text {op }}} J_I^{\otimes_R n}
		$$
		is an equivalence. Since the limit stabilises degreewise and $J_I$ is connective, we can move the limit out of the tensor product by \cref{limtower} (the cofibre of the interchange map is a limit of terms with growing connectivity), and then the statement follows from finality. 	By the same argument, for any $n\geq 1$ the canonical map  $$I^\infty \otimes_R J_I^{\otimes_R n}\to I^\infty$$ is equivalent too and hence $I^\infty \otimes_RI^\infty \to I^\infty$ is an equivalence---that is to say $R\to R/I^\infty$ is idempotent and hence produces an element in $\alg(\ageq)^{\op{idem}}_{R/}$.

		As the next step, we show that for a left $R$-module $M$, the tautological map $M=R \otimes_R M \rightarrow R / I^{\infty} \otimes_R M$ is an equivalence (in other words that $I^{\infty} \otimes_R M \simeq 0$) if and only if the homotopy of $M$ is annihilated by $I$. For the ``only if'' direction, it suffices to observe that $
		\pi_0(R/I^\infty)=\pi_0R/I$. For ``if'' direction, 
		we start with the simplest case $M=R / J_I$, where the claim was proved above. For an arbitrary left $R$-module $M$ concentrated in degree 0 and killed by the action of $I$, it naturally inherits a $\op{H}(\pi_0R/I)$-module structure  and hence we have equivalences of left $R$-modules $$I^\infty\otimes_RM\simeq  I^\infty\otimes_R \op{H}(\pi_0R/I)\otimes_{\op{H}(\pi_0R/I)}M .$$ However $ I^\infty\otimes_R \op{H}(\pi_0R/I)$ is zero  by  $\op{H}(\pi_0R/I)=R/J_I$, so $I^\infty\otimes_RM=0$.

		For bounded below $M$, by left completeness we have
		$$
		I^{\infty} \otimes_R M \simeq I^{\infty} \otimes_R\left(\lim _{k \in \mathbb{N}^{\mathrm{op}}} \tau_{\leq k} M\right) \simeq \lim _{k \in \mathbb{N}^{\circ \mathrm{p}}} I^{\infty} \otimes_R \tau_{\leq k} M \simeq 0
		$$
		by commuting the limit out using the same argument as above. Finally, for arbitrary $M$ whose homotopy groups are killed by $I$, we find
		$$
		I^{\infty} \otimes_R M \simeq I^{\infty} \otimes_R\left(\operatorname{colim}_{k \in \mathbb{N}} \tau_{\geq-k} M\right) \simeq \operatorname{colim}_{k \in \mathbb{N}} I^{\infty} \otimes_R \tau_{\geq-k} M \simeq 0 .
		$$

		So combined with the idempotent property of $R\to R/I^\infty$, we learn that the image of the fully faithful restriction functor $\operatorname{LMod}_{R / I^{\infty}}(\mca)\rightarrow \operatorname{LMod}_R(\mca)$ consists exactly of those left modules whose homotopy is killed by $I$, as desired.

		Finally, we are ready to verify that the construction $I \mapsto\left(R \rightarrow R / I^{\infty}\right)$ induces an inverse to taking kernels. The composition starting with an ideal is clearly the identity. So we are left to show that for every $\varphi: R \rightarrow S$  in $\mathrm{LQ}_R$ with $I=\operatorname{Ker}(\pi_0 \varphi)$, the canonical map $\psi: R / I^{\infty} \rightarrow S$, arising from the homotopy of $S$ being annihilated by $I$, is an equivalence. By construction it induces an equivalence on $\pi_0$. By \cref{naka}, the functor $\psi_{!}=S \otimes_{R / I^{\infty}}-: \operatorname{LMod}_{R / I^{\infty}}(\mca) \rightarrow \operatorname{LMod}_S(\mca)$ is thus conservative when restricted to bounded below modules. But the map
		$$
		S \simeq \psi_{!}\left(R / I^{\infty}\right) \xrightarrow{\psi_{!}(\varphi)} \psi_{!}(S)=S \otimes_{R / I^{\infty}} S \simeq S \otimes_R S
		$$
		is induced by the unit and thus an equivalence since $\varphi$ is idempotent.
	\end{proof}
	\begin{rem}
		The argument above also applies to $\mathbb{E}_k$-algebras. However, since we do not treat $\mathbb{E}_k$-algebras in detail in this paper, we choose to omit it.
	\end{rem}
	\section{Deformation theory and étale rigidity}\label{5}
	
	To effectively study étale maps, we must first port the machinery of derivations and cotangent complexes into our generalized categorical setting. Fortunately, most of them has been developed in \cite[\textsection7.3-7.5]{ha}. 
	\subsection{The cotangent complex formalism}
	
	We use throughout the cotangent-complex and derivation formalism of
	\cite[\S\S7.3--7.4]{ha}. We briefly recall the notation that will be
	used below.
	
	For a presentable $\infty$-category $\mathcal C$, let
	\[
	T_{\mathcal C}\longrightarrow \mathcal C
	\]
	denote its tangent bundle and let $\mathcal M^T(\mathcal C)$ denote
	the associated tangent correspondence; see
	\cite[\S7.3.2 and Definition 7.3.6.9]{ha}. The cotangent complex
	functor is denoted by
	\[
	L:\mathcal C\longrightarrow T_{\mathcal C}.
	\]
	
	We will only use this formalism for
	$\mathcal C=\operatorname{CAlg}(\mathcal A)$. In this case, the fiber
	of $T_{\mathcal C}$ over $A$ is canonically identified with
	$\operatorname{Mod}_A(\mathcal A)$, and we denote by $L_A$ the
	cotangent complex of $A$ and by $L_{B/A}$ the relative cotangent
	complex of a map $A\to B$.
	
	We let
	\[
	\operatorname{Der}
	=
	\operatorname{Der}(\operatorname{CAlg}(\mathcal A))
	\]
	be the $\infty$-category of derivations
	\cite[Definition 7.4.1.1]{ha}. Thus, a derivation of $A$ with values
	in an $A$-module $M$ will be written
	\[
	\eta:A\longrightarrow M[1],
	\]
	or equivalently as a map $L_A\to M[1]$ of $A$-modules. We denote by
	\[
	A^\eta\longrightarrow A
	\]
	the square-zero extension classified by $\eta$.
	
	We also use Lurie's $\infty$-category
	$\widetilde{\operatorname{Der}}$ of extended derivations and the
	natural functor
	\[
	\Phi:
	\widetilde{\operatorname{Der}}
	\longrightarrow
	\operatorname{Fun}(\Delta^1,\operatorname{CAlg}(\mathcal A)),
	\qquad
	\eta\longmapsto(A^\eta\to A),
	\]
	as in \cite[Definition 7.4.1.3]{ha}.
	
	We will use the following connective variants of these categories.
	Let $\operatorname{Der}^{+}\subset\operatorname{Der}$ be the
	subcategory whose objects are derivations
	\[
	\eta:A\longrightarrow M[1]
	\]
	with $A$ and $M$ connective, and whose morphisms
	\[
	(\eta:A\to M[1])\longrightarrow
	(\eta':B\to N[1])
	\]
	are those inducing an equivalence
	\[
	B\otimes_A M\xrightarrow{\sim}N.
	\]
	Equivalently, an object belongs to $\operatorname{Der}^{+}$ precisely
	when $A$ and $A^\eta$ are connective and
	$\pi_0A^\eta\to\pi_0A$ is epimorphic.
	
	Similarly, let
	\[
	\operatorname{Fun}^{+}
	(\Delta^1,\operatorname{CAlg}(\mathcal A))
	\subset
	\operatorname{Fun}
	(\Delta^1,\operatorname{CAlg}(\mathcal A))
	\]
	be the subcategory whose objects are maps
	$\widetilde A\to A$ between connective algebras which are epimorphic
	on $\pi_0$, and whose morphisms are pushout squares.
	We use the analogous notation
	$\widetilde{\operatorname{Der}}^{+}$ for extended derivations.
	
	\begin{rem}
		We will repeatedly use the following two standard properties of
		cotangent complexes. For composable maps
		\[
		A\longrightarrow B\longrightarrow C
		\]
		there is a canonical cofiber sequence
		$$
		C\otimes_B L_{B/A}
		\longrightarrow L_{C/A}
		\longrightarrow L_{C/B}.
		$$
		Moreover, for a pushout square
		\[
		\begin{tikzcd}
			A' \arrow[r] \arrow[d] & A \arrow[d] \\
			B' \arrow[r] & B ,
		\end{tikzcd}
		\]
		there is a canonical equivalence
		$$
		L_{B/A}\simeq B\otimes_{B'}L_{B'/A'}.
		$$
		See \cite[Propositions 7.3.3.6 and 7.3.3.7]{ha}.
	\end{rem}

	\begin{lem}\label{5.6}
		Assume that \mca is hypercomplete. Let $f:(\eta: A \rightarrow M[1]) \rightarrow(\eta^{\prime}: B \rightarrow N[1])$ be a morphism in $\operatorname{Der}^{+}$. If the induced map $A^\eta \rightarrow B^{\eta'}$ is an equivalence in $\operatorname{CAlg}(\mc{A})$, then $f$ is an equivalence. \textup{(See \cite[Lem. 7.4.2.9]{ha} for the case of spectra.)}
	\end{lem}
	\begin{proof}
		The morphism $f$ determines a map of fiber sequences
		$$\begin{tikzcd}
			A^\eta \arrow[d] \arrow[r] & A \arrow[d, "f_0"'] \arrow[r] & {M[1]} \arrow[d, "f_1"'] \\
			B^{\eta'} \arrow[r]        & B \arrow[r]                   & {N[1]}                  
		\end{tikzcd}$$
		Since the left vertical map is an equivalence, we obtain an equivalence $\alpha: \operatorname{cofib}(f_0) \simeq \operatorname{cofib}(f_1)$. To complete the proof, it will suffice to show that $\operatorname{cofib}(f_0)$ vanishes. Suppose otherwise. Since $\operatorname{cofib}(f_0)$ is connective and $\mc{A}$ is hypercomplete, there exists some smallest integer $n$ such that $\pi_n \operatorname{cofib}(f_0) \neq 0$. In particular, $\operatorname{cofib}(f_0)$ is $n$-connective.
		
		Since $f$ induces an equivalence $B \otimes_A M \rightarrow N$, $\op{cofib}(f_1)$ can be identified with $\operatorname{cofib}(f_0) \otimes_A M[1]$. Since $M$ is connective, we deduce that $\operatorname{cofib}(f_1)$ is $(n+1)$-connective. Using the equivalence $\alpha$, we conclude that $\operatorname{cofib}(f_0)$ is $(n+1)$-connective, which contradicts our assumption that $\pi_n \operatorname{cofib}(f_0) \neq 0$.
	\end{proof}
	\begin{de}
		We define a subcategory $\op{Fun}^{+}(\Delta^1, \mathrm{CAlg}(\mc{A}))$ as follows:
		\begin{enumerate}
			\item  An object $f: \widetilde{A} \rightarrow A$ of $\op{Fun}(\Delta^1, \mathrm{CAlg}(\mc{A}))$ belongs to $\op{Fun}^{+}(\Delta^1, \mathrm{CAlg}(\mc{A}))$ if and only if both $A$ and $\widetilde{A}$ are connective, and $f$ induces a surjection $\pi_0 \widetilde{A} \rightarrow \pi_0 A$.
			\item 	 Let $f, g \in \op{Fun}^{+}(\Delta^1, \mathrm{CAlg}(\mc{A}))$, and let $\alpha: f \rightarrow g$ be a morphism in  $\op{Fun}(\Delta^1, \mathrm{CAlg}(\mc{A}))$. Then $\alpha$ belongs to $\op{Fun}^{+}(\Delta^1, \mathrm{CAlg}(\mc{A}))$ if and only if it classifies a pushout square in the $\infty$-category $\mathrm{CAlg}(\mc{A})$.
		\end{enumerate}
		
	\end{de}
	\begin{prop}\label{5.8}
		Assume that \mca is hypercomplete.
		Let $\Phi: \widetilde{\operatorname{Der}} \rightarrow \operatorname{Fun}(\Delta^1, \operatorname{CAlg}(\mc{A}))$ be the functor given by $(\eta: A \rightarrow M[1]) \mapsto (A^\eta\to A)$. Then $\Phi$ induces a functor $\Phi^{+}: \widetilde{\operatorname{Der}}^{+} \rightarrow \operatorname{Fun}^{+}(\Delta^1, \operatorname{CAlg}(\mc{A}))$. Moreover, the functor $\Phi^{+}$ is a left fibration.
		
	\end{prop} 
	\begin{proof}
		It is a parallel proof of \cite[Lem.  7.4.2.7]{ha}.
	\end{proof}
	\begin{rem}
		The hypercomplete condition in \cref{5.8} can not be removed. However, the only part involving the hypercompleteness in the argument of \textup{\cite[Lem.  7.4.2.7]{ha}} is \cref{5.6}.
	\end{rem}
	\begin{cor}\label{cor8.12}
		Assume that \mca is hypercomplete.	Let $A \in \calg(\mc{A}_{\geq0})$, $M$ a connective $A$-module, and $\eta: A \rightarrow M[1]$ a derivation. Then the functor $\Phi$ induces an equivalence of $\infty$-categories
		
		$$
		\operatorname{Der}_{\eta /}^{+} \rightarrow \mathrm{CAlg}_{A^\eta}^{\mathrm{cn}}
		$$
		
		given on objects by $\left(\eta^{\prime}: B \rightarrow N[1]\right) \mapsto B^{\eta^{\prime}}$.
	\end{cor}
	We end with a surprising result that any square zero extension is descendable\footnote{We thank Germán Stefanich for pointing this out to the author.}.
	\begin{prop}\label{sq0descendable}
		Let $\eta: A \rightarrow M \in \operatorname{Der}$ and $\alpha:\tilde{A}\to A$ be the induced square-zero extension  in $\calg(\mc{A})$. Then $\alpha$ is descendable.
	\end{prop}
	\begin{proof}
		By the definition of a square-zero extension, the derivation $\eta: A \to M$ induces a fiber sequence in $\mathrm{Mod}_{\tilde{A}}(\mca)$:
		$$ M \longrightarrow \tilde{A} \xrightarrow{\alpha} A .$$
		To apply \cref{criteriondescendable}, we must verify that both $A$ and $M$ (viewed as $\tilde{A}$-modules) admit $A$-module structures:
		\begin{itemize}
			\item The object $M$ is naturally an $A$-module by the definition of the derivation $\eta \in \mathrm{Der}(A, M)$. Its $\tilde{A}$-module structure is simply obtained via restriction of scalars along $\alpha$.
			\item The object $A$ is trivially an $A$-module. 
		\end{itemize}
		Thus, $\tilde{A}$ is obtained as a finite limit of a diagram of $\tilde{A}$-modules, each of which admits the structure of a module over $A$. Taking $f = \alpha$, \cref{criteriondescendable} directly implies that the morphism $\alpha: \tilde{A} \to A$ is descendable.
	\end{proof}
	\begin{cor}\label{sq0faith}
		Let $\eta: A \rightarrow M \in \operatorname{Der}$ and $\alpha:\tilde{A}\to A$ be the induced square-zero extension  in $\calg(\mc{A})$. Then $\alpha$ is faithful.
	\end{cor}
	\begin{proof}[First proof]
		It follows by combining \cref{descendableimplyfaith} and \cref{sq0descendable}.
	\end{proof}
	We can also check it directly:
	\begin{proof}[Second proof]
		We consider the following pullback diagram in $\calg(\mc{A})$, and hence also a pullback diagram in $\modu_{\tilde{A}}(\mc{A})$.
		$$
		\begin{tikzcd}
			\tilde{A} \arrow[d, "\alpha"] \arrow[r, "\alpha"] & A \arrow[d, "\delta"] \\
			A \arrow[r, "\delta_0"]                           & A\oplus M            
		\end{tikzcd}
		$$
		Now given an $\tilde{A}$-module $X$ such that $X\otimes_{\tilde{A}}A=0$, we wish to show that $X=0$ too. Because the following diagram is pullback in $\modu_{\tilde{A}}(\mc{A})$,
		$$
		\begin{tikzcd}
			X=X\otimes_{\tilde{A}}\tilde{A} \arrow[d] \arrow[r] & X\otimes_{\tilde{A}}A=0 \arrow[d] \\
			0=X\otimes_{\tilde{A}}A \arrow[r]                   & X\otimes_{\tilde{A}}(A\oplus M)=0
		\end{tikzcd}$$
		we get $X=0$.
	\end{proof}
	\subsection{L-étale algebras}
	With the cotangent complex established, we formalize the notion of L-étale (formally étale) morphisms as those explicitly possessing a vanishing relative cotangent complex. We demonstrate that L-étale morphisms correspond precisely to cocartesian edges within the tangent correspondence, providing a robust, operational geometric framework for deformation theory in arbitrary $ttt$-$\infty$-categories.
	\begin{de}
		We say  a map of \ein-algebras $A\to B  \in \calg(\mc{A})$ is L-étale (or formally étale)\footnote{In deformation theory, the L-étale condition seems only interesting in the connective case.} if the relative cotangent complex $L_{B/A}$ vanishes.
	\end{de}

	\begin{lem}\label{lem8.12}
		Let $$\begin{tikzcd}
			& A \arrow[ld, "f"] \arrow[rd, "h"'] &   \\
			B \arrow[rr, "g"] &                                    & C
		\end{tikzcd}$$
		be morphisms in $\calg(\mc{A})$. Then:
		\enu{\item Suppose that $f$ is L-étale. Then $g$ is L-étale if and only if $h$ is L-étale.
			\item If $g$ is L-étale and faithfully flat. Then $f$ is L-étale if and only if $h$ is L-étale.
		}
	\end{lem}
	\begin{proof}
		It follows from the cofiber sequence $$C\otimes_BL_{B / A}\to L_{C / A}\to L_{C / B} $$ in $\modu_C(\mc{A})$.
	\end{proof}
	\begin{prop}
		Let	$f: R\to S \in \calg(\mc{A})$ be map such that $S$ is an idempotent  \ein-$R$-algebra. Then $f$ is L-étale.
	\end{prop}
	\begin{proof}
		Since $S\otimes_R S\simeq S$, by \cite[Proposition 7.3.3.7]{ha} we have $L_{S/R}\simeq S\otimes_S L_{S/R}\simeq L_{S/S}=0$.
	\end{proof}
	\begin{cor}
		Let $f:R\to S \in \calg(\mc{A})$ be a flat map between discrete \ein-algebras such that $f$ is idempotent in the heart $\mca^\heartsuit$. Then $f:R\to S$ is L-étale.
	\end{cor}
	\begin{rem}
		Let $f:A\to B$ be a flat morphism in
		$\operatorname{CAlg}(\mathcal A)$. By \cref{flatpush}(1) and base
		change for cotangent complexes, there is a natural equivalence
		\[
		L_{B/A}\simeq
		B\otimes_{\tau_{\geq0}B}
		L_{\tau_{\geq0}B/\tau_{\geq0}A}.
		\]
		Consequently, if $\tau_{\geq0}f$ is $L$-étale, then $f$ is
		$L$-étale. However, the converse fails, even for spectra.
		
		Indeed, let $C_2$ be the cyclic group of order $2$ and consider
		\[
		f:KU\longrightarrow KU[C_2].
		\]
		Its connective cover is
		\[
		ku\longrightarrow ku[C_2],
		\]
		and $f$ is flat. By the computation of the cotangent complex of
		Thom spectra \cite{MR4217953}, we have
		\[
		L_{R[C_2]/R}\simeq R[C_2]\otimes H\mathbb F_2.
		\]
		Hence
		$
		L_{ku[C_2]/ku}\not\simeq0,
		$
		whereas
		\[
		L_{KU[C_2]/KU}\simeq0,
		\]
		since $KU\otimes H\mathbb F_2\simeq0$.
	\end{rem}
	\begin{lem}
		Given a diagram of $\infty$-categories 
		$$
		\begin{tikzcd}
			\mc{C} \arrow[rr, "F"] \arrow[rd, "p"] &        & \mc{D} \arrow[ld, "q"'] \\
			& \mc{E} &                        
		\end{tikzcd}
		$$
		where $p,q$ are cocartesian fibrations and $F$ preserves cocartesian edges. Let $s\in \mc{E}$ and $\theta:K^{\triangleright}\to \mc{C}_s$ be a diagram. If for any morphism $f:s\to t\in \mc{E}$ the edge $f_!\circ\theta:K^{\triangleright}\to\mc{C}_t$ is an $F_s$-colimit, then $\theta$ is an $F$-colimit.
	\end{lem}
	\begin{proof}
		This is the relative version of \cite[Prop.  4.3.1.10]{htt}.
	\end{proof}
	\begin{cor}\label{cor8.13}
		Given a diagram of $\infty$-categories 
		$$
		\begin{tikzcd}
			\mc{C} \arrow[rr, "F"] \arrow[rd, "p"] &        & \mc{D} \arrow[ld, "q"'] \\
			& \mc{E} &                        
		\end{tikzcd}
		$$
		where $p,q$ are cocartesian fibrations and $F$ preserves cocartesian edges. Let $s\in \mc{E}$ and $\theta:x\to y$ be a morphism in the fiber $\mc{C}_s$. If for any morphism $f:s\to t\in \mc{E}$, the edge $f_!(\theta):f_!(x)\to f_!(y)$ is $F_s$-cocartesian, then $\theta$ is an $F$-cocartesian.
	\end{cor}
	
	\begin{prop}\label{prop8.14}
		Let $f:A\to B \in \calg(\mc{A})\subset \mc{M}^T(\calg(\mc{A}))$ be a morphism of $\mathbb{E}_\infty$-algebras. Then $f$ is $F$-cocartesian if and only if $f$ is L-étale, where $F: \mc{M}^T(\calg(\mc{A}))\to \Delta^1\times \calg(\mc{A})$ is the natural projection.
	\end{prop}
	\begin{proof}
		Applying \cref{cor8.13} to the following diagram, we win.
		$$
		\begin{tikzcd}
			\mc{M}^T(\calg(\mc{A})) \arrow[rr, "F"] \arrow[rd, "p"] &          & \Delta^1\times \calg(\mc{A}) \arrow[ld, "q"'] \\
			& \Delta^1 &                                             
		\end{tikzcd}
		$$
	\end{proof}
	Let $\operatorname{Der}^{L\text{-}et}\subset\operatorname{Der}$
	denote the subcategory whose objects are connective derivations and
	whose morphisms
	\[
	(\eta:A\to M[1])\longrightarrow(\eta':B\to N[1])
	\]
	are those for which $A\to B$ is $L$-étale and
	\[
	B\otimes_A M\xrightarrow{\sim}N.
	\]
	We similarly write
	$\operatorname{CAlg}(\mathcal A_{\geq0})^{L\text{-}et}$ for the
	subcategory of connective commutative algebras and $L$-étale
	morphisms.
	\begin{prop}\label{leftfib2}
		Let $f:\operatorname{Der} \rightarrow \calg(\mc{A})$ denote the forgetful functor $(\eta: A \rightarrow M) \mapsto A$. Then $f$ induces a left fibration $\operatorname{Der}^{L\text{-}et}  \rightarrow \calg(\mc{A}_{\geq0})^{L\text{-}et}$.
	\end{prop}
	\begin{proof}
		Fix $0\leq i < n$; we must show that every lifting problem of the form
		$$\begin{tikzcd}
			\Lambda^n_i \arrow[r] \arrow[d]            & \operatorname{Der}^{L\text{-}et} \arrow[d] \\
			\Delta^n \arrow[r] \arrow[ru, "l", dashed] & \calg(\mc{A}_{\geq0})^{L\text{-}et}        
		\end{tikzcd}$$
		admits a solution $l$. Considering the following diagram,
		
		$$
		\begin{tikzcd}
			\Lambda^n_i \arrow[r] \arrow[d]                                      & \operatorname{Der}^{L\text{-}et} \arrow[d] \arrow[r] & \op{Der} \arrow[d] \arrow[r] \arrow[rd, "\ulcorner", phantom, very near start] & {\op{Fun}(\Delta^1,\mc{M}^T(\calg(\mc{A})))} \arrow[d] \\
			\Delta^n \arrow[r] \arrow[ru, "l", dashed] \arrow[rru, "l'" pos=0.7, dashed] \arrow[rrru, "l^{''}" pos=0.7, dashed]& \calg(\mc{A}_{\geq0})^{L\text{-}et} \arrow[r]         & \calg(\mc{A}) \arrow[r]                           & {\op{Fun}(\Delta^1,\Delta^1\times \calg(\mc{A}))}     
		\end{tikzcd}
		$$
		then there exists a lifting $l^{''}$ by \cref{prop8.14}, and hence there exists a lifting $l^{'}$. We observe that $l'$ actually lies in $\operatorname{Der}^{L\text{-}et} $ by \cref{lem8.12}, hence we find a solution $l$.
	\end{proof}

	\subsection{Étale rigidity}
	
	We arrive at the main deformation theorem of this section: Étale Rigidity. We formally prove that in a projectively rigid $ttt$-$\infty$-category, the $\infty$-category of étale algebras  over a connective \ein-algebra is entirely equivalent to the ordinary category of discrete étale extensions over its $\pi_0$-truncation. This demonstrates that étale morphisms are topologically rigid and admit no higher homotopical deformations along the Postnikov tower.
	\begin{de}
		We say that a map $f:A\to B$ in $\calg(\mc{A})$ is étale if $f$ is flat and the map $\tau_{\geq 0}f:\tau_{\geq 0}A\to \tau_{\geq 0}B$ is L-étale and finitely presented.
		
	\end{de}
	\begin{rem}\label{exetalenotflt}
		One may naturally ask whether a finitely presented  L-étale morphism $f: A \to B\in\calg(\ageq)$ of connective \ein-algebras  is necessarily flat, which holds in the case of spectra (see \cite[Lemma B.1.3.3]{sag}). In other words, can the flatness condition in the definition of an étale morphism be omitted? 
		
		However, the answer is negative, even in the projectively rigid case. For example, consider the smashing localization$$ \mca = \opsp_{C_p} \to \opsp_{C_p}/\opsp^{BC_p} \simeq \opsp, $$where $\opsp_{C_p}$ denotes the \infcat of (genuine) $C_p$-equivariant spectra over the cyclic group of order $p$ for a prime $p$. The corresponding idempotent algebra $\mathbb{S}_{C_p} \to \mathbb{S}$ in $\calg(\opsp_{C_p,\geq0})$ is finitely presented and L-étale, but it is not flat.
	\end{rem}
	
	The main result in this subsection is the following theorem (see \cite[\textsection 7.5]{ha}  for the case of spectra).
	\begin{thm}[Étale rigidity]\label{etrig}
		Let $A\in \operatorname{CAlg}(\mathcal{A})$ be an \ein-algebra. Then:
		\begin{enumerate}[label=(\arabic*),font=\normalfont]
			
			\item Suppose that $\mc{A}$ is Grothendieck and left complete.  Let $\operatorname{CAlg}(\mathcal{A})_{A /}^{\text{fl},L\text{-}et}$ denote the full subcategory of $\operatorname{CAlg}(\mathcal{A})_{A /}$ spanned by the flat L-étale maps $A \rightarrow B$. If  $A$ is connective\footnote{The connectivity of $A$ can be removed if we replace the $L$-étale condition on $A\to B$ by the $L$-étale condition on the connective cover $\tau_{ \geq 0}A\to\tau_{ \geq 0} B$.}, then the functor $\pi_0$ induces an equivalence $$\operatorname{CAlg}(\mathcal{A})_{A /}^{\text{fl},L\text{-}et}\xrightarrow{\sim} \operatorname{CAlg}(\mathcal{A}^{\heartsuit})_{\pi_0A /}^{\text{fl},L\text{-}et}$$ with the (1-)category of the discrete flat L-étale commutative $\pi_0 A$-algebras.
			\item Suppose that the $ttt$-\infcat $(\mc{A}^\otimes, \mc{A}_{\geq0}) $ is projectively rigid. Let $\operatorname{CAlg}(\mathcal{A})_{A /}^{et}$ denote the full subcategory of $\operatorname{CAlg}(\mathcal{A})_{A /}$ spanned by the étale maps $A \rightarrow B$. Then the functor $\pi_0 $ induces an equivalence $$\operatorname{CAlg}(\mathcal{A})_{A /}^{et}\xrightarrow{\sim} \operatorname{CAlg}(\mathcal{A}^{\heartsuit})_{\pi_0A /}^{et}$$ with the (1-)category of the discrete étale commutative  $\pi_0 A$-algebras.
			
		\end{enumerate}
	\end{thm}
	\begin{rem}
		The flat condition in the heart $\calg(\mca^\heartsuit)$ above should be understood as the (derived) flatness in the sense of \cref{ncflat}, which is generally a stronger condition than the 1-flatness (see \cref{differenceflat}).
	\end{rem}
	\begin{ex}[Dirac geometry]
		For any \einfring $R\in \calg(\opsp)$, we have the following commutative diagram:
		$$\begin{tikzcd}
			\operatorname{CAlg}\left(\operatorname{Mod}_R(\mathrm{Sp})\right)^{\op{D}\text{-ét}} \arrow[r] \arrow[d, "\sim", "\nu"'] & \mathrm{CAlg}\left(\operatorname{Mod}_{\pi_*R}(\mathrm{Ab})\right)^{\op{D}\text{-ét}} \arrow[d, "\sim"]         \\
			\operatorname{CAlg}(\syn_R)^{\text{ét}} \arrow[r, "\sim"]    & \operatorname{CAlg}(\syn_R^\heartsuit)^{\text{ét}}
		\end{tikzcd}$$
		where the bottom equivalence follows from \cref{etrig}(2). That recovers the étale rigidity theorem in Dirac geometry \cite[Theorem 4.33]{hesselholt2023dirac}. A more detailed study of synthetic higher algebra will appear in \cite{synha}.
	\end{ex}
	One consequence is the following identification between flat idempotent algebras.
	\begin{cor}
		Suppose that $\mc{A}$ is Grothendieck and left complete. Let $A\in \operatorname{CAlg}(\mathcal{A})$. Then the functor $\pi_0$ induces an equivalence 
		$$ \operatorname{CAlg}(\mathcal{A})_{A /}^{\text{fl},\op{idem}} \xrightarrow{\sim} \operatorname{CAlg}(\mathcal{A}^{\heartsuit})_{\pi_0 A /}^{\text{fl},\op{idem}} $$ 
		of categories of flat idempotent \ein-algebras.
	\end{cor}
	
	\begin{proof}
		Note that flat idempotent $\mathbb{E}_\infty$-algebras over $A$ can be identified with flat idempotent $\mathbb{E}_\infty$-algebras over $\tau_{\geq 0}A$, because by \cref{l5}(3), there is a natural equivalence of symmetric monoidal \infcats $\mathrm{Mod}_{\tau_{\geq 0}A}(\mathcal{A})^{fl,\otimes} \xrightarrow{\sim} \mathrm{Mod}_{A}(\mathcal{A})^{fl,\otimes}$. It implies that their connective covers are idempotent, and hence L-étale over $\tau_{\geq 0}A$. Furthermore, under the flatness assumption, an idempotent algebra in the heart $\mathrm{CAlg}(\mathcal{A}^\heartsuit)$ is idempotent in $\mathrm{CAlg}(\mathcal{A})$ and is therefore also L-étale. By \cref{etrig}(1), it suffices to show that for any flat morphism $A \to B$ in $\mathrm{CAlg}(\mathcal{A})$ such that $\tau_{\geq 0}A \to \tau_{\geq 0}B$ is L-étale and $\pi_0 B \otimes_{\pi_0 A} \pi_0 B \simeq \pi_0 B$, we have $B \otimes_A B \simeq B$. This, however, follows directly from \cref{pi0flat}(2).
	\end{proof}

	The proof of \cref{etrig} will occupy the remainder of this section.

	\begin{prop}\label{prop8.22}
		Let $\eta: A \rightarrow M \in \operatorname{Der}^+$ and $\alpha:\tilde{A}\to A$ be the induced square-zero extension in $\calg(\mc{A}_{\geq0})$. Now given a pushout diagram in $\calg(\mc{A}_{\geq0})$.
		$$
		\begin{tikzcd}
			\tilde{A} \arrow[d, "f_0'"] \arrow[r, "\alpha"] & A \arrow[d, "f_0"] \\
			\tilde{B} \arrow[r]                   & B                 
		\end{tikzcd}
		$$
		Then:
		\enu{\item $f_0'$ is L-étale if and only if $f_0$ is L-étale. 
			\item Assume that $\mc{A}$ is Grothendieck. Then $f_0'$ is flat if and only if $f_0$ is flat. 
			\item If $f_0'$ is L-étale, then $f_0'$ is locally of finite presentation if and only if $f_0$ is locally of finite presentation.
			
		}
		
	\end{prop}
	\begin{proof}
		(1) The ``only if'' direction is obvious. The ``if'' direction follows from the equivalences $$L_{B/A}\simeq B\otimes_{\tilde{B}}L_{\tilde{B}/\tilde{A}}\simeq A\otimes_{\tilde{A}}L_{\tilde{B}/\tilde{A}}$$ and  \cref{sq0faith}.
		\\(2) The ``only if'' direction is obvious. For the converse, suppose that $B$ is flat over $A$ : it suffices to show that for every discrete $\widetilde{A}$-module $N$, the relative tensor product $\widetilde{B} \otimes_{\widetilde{A}} N$ is discrete by \cref{l7}(3). To prove this, let $I \subset \pi_0 \widetilde{A}$ be the kernel of the surjective map $\pi_0 \widetilde{A} \rightarrow \pi_0 A$, so that we have a short exact sequence of modules over $\pi_0 \widetilde{A}$ :
		$$
		0 \rightarrow I N \rightarrow N \rightarrow N / I N \rightarrow 0
		$$
		It will therefore suffice to show that the tensor products $\widetilde{B} \otimes_{\widetilde{A}} I N$ and $\widetilde{B} \otimes_{\widetilde{A}} N / I N$ are discrete. Replacing $N$ by $I N$ or $N / I N$, we can reduce to the case where $I N=0$, so that $N$ has the structure of an $A$-module. Then $\widetilde{B} \otimes_{\widetilde{A}} N \simeq B \otimes_A N$ is discrete by virtue of the assumption that $B$ is flat over $A$.\\
		(3) The proof is parallel to that of \cite[Lem.~2.5.4]{dag13}.
	\end{proof}
	
	\begin{prop}\label{sqet}
		Assume that \mca is hypercomplete.	Let $A \in \calg(\mc{A}_{\geq0})$, $M$ be a connective $A$-module, and $\eta: A \rightarrow M[1]$ be a derivation. Then the square-zero extension $\tilde{A}\to A$ induces an equivalence $$(-)\otimes_{\tilde{A}}A: \calg(\mc{A}_{\geq0})_{\tilde{A}/}^{L\text{-}et}\xrightarrow{\sim} \calg(\mc{A}_{\geq0})_{A/}^{L\text{-}et}$$
		between $\infty$-categories of connective L-étale commutative  $\tilde{A}$-algebras and  $A$-algebras.
	\end{prop}
	\begin{proof}
		Any square-zero extension $\widetilde{A} \rightarrow A$ is associated to some derivation $(\eta: A \rightarrow$ $M) \in \operatorname{Der}^{L\text{-}et}$. Let $\Phi: \operatorname{Der} \simeq \widetilde{\operatorname{Der}} \rightarrow \operatorname{Fun}(\Delta^1, \calg(\mc{A}))$ be the functor defined in \cref{5.8}. Let $\Phi_0, \Phi_1:$ $\operatorname{Der} \rightarrow \calg(\mc{A})$ denote the composition of $\Phi$ with evaluation at the vertices $\{0\},\{1\} \in \Delta^1$. The functors $\Phi_0$ and $\Phi_1$ induce maps
		$$
		\calg(\mc{A}_{\geq0})_{\tilde{A}/}^{L\text{-}et} \xleftarrow{\Phi_0^{\prime}} \operatorname{Der}^{L\text{-}et} \xrightarrow{\Phi_1^{\prime}}\calg(\mc{A}_{\geq0})_{A/}^{L\text{-}et}
		$$
		Moreover, the functor $\Phi$ exhibits $\Phi_1^{\prime}$ as equivalent to the composition of $\Phi_0^{\prime}$ with the relative tensor product $\otimes_{\widetilde{A}} A$. Consequently, it will suffice to prove the following:
		\enu{
			\item  The functor $\Phi_0^{\prime}$ is fully faithful, and its essential image consists precisely of the connective L-étale commutative $\widetilde{A}$-algebras.
			\item  The functor $\Phi_1^{\prime}$ is fully faithful, and its essential image consists precisely of the connective L-étale commutative $A$-algebras.
		}
		The (1) follows from \cref{cor8.12} and \cref{prop8.22} (1). And the (2) follows from \cref{leftfib2}.
	\end{proof}
	We will also use Lurie's classification of small extensions
	\cite[Theorem 7.4.1.26]{ha}. In particular, every $n$-small
	extension of connective commutative algebras is a square-zero
	extension. Since the successive maps in the Postnikov tower are
	small extensions, we obtain:
	
	\begin{lem}[{\cite[Corollary 7.4.1.28]{ha}}]\label{postnikovsquarezero}
		Let $A\in\operatorname{CAlg}(\mathcal A_{\geq0})$. Then every
		map in the Postnikov tower
		\[
		\cdots\longrightarrow\tau_{\leq2}A
		\longrightarrow\tau_{\leq1}A
		\longrightarrow\tau_{\leq0}A
		\]
		is a square-zero extension.
	\end{lem}
	
	\begin{thm}\label{mappetale}
		Assume that \mca is left complete.	Let $f: A \rightarrow B \in\operatorname{CAlg}(\mathcal{A})$ be a flat map such that $\tau_{\geq 0}f: \tau_{\geq 0}A \to \tau_{\geq 0}B$ is L-étale. Given an arbitrary $C\in \operatorname{CAlg}(\mathcal{A})$. The canonical map
		$$
		\operatorname{Map}_{\calg(\mc{A})_{A /}}(B, C) \rightarrow \operatorname{Map}_{\calg(\mc{A})_{\pi_0 A /}}\left(\pi_0 B, \pi_0 C\right)
		$$
		is a homotopy equivalence. In particular, $\operatorname{Map}_{\calg(\mc{A})_{A /}}(B, C)$ is homotopy equivalent to a discrete space.
	\end{thm}
	\begin{proof}
		The argument is the same as that of
		\cite[Proposition 3.4.13]{dag4}; we recall the only point needed in
		our setting.
		
		By \cref{flatpush}(1), after replacing $A$ and $B$ by their
		connective covers, we may assume that $A$ and $B$ are connective.
		Since $B$ is connective, maps from $B$ to a discrete algebra factor
		uniquely through $\pi_0B$, so it remains to prove that
		\[
		\operatorname{Map}_{\calg(\mca)_{A/}}(B,C)
		\longrightarrow
		\operatorname{Map}_{\calg(\mca)_{A/}}(B,\pi_0C)
		\]
		is an equivalence for every connective $A$-algebra $C$.
		
		For a trivial square-zero extension $D\oplus M\to D$, the fiber of
		\[
		\operatorname{Map}_{\calg(\mca)_{A/}}(B,D\oplus M)
		\longrightarrow
		\operatorname{Map}_{\calg(\mca)_{A/}}(B,D)
		\]
		over a map $B\to D$ is naturally equivalent to
		\[
		\operatorname{Map}_{\modu_D(\mca)}
		(L_{B/A}\otimes_BD,M),
		\]
		which is contractible since $A\to B$ is $L$-étale. By base change,
		the same conclusion holds for every square-zero extension.
		
		The transition maps in the Postnikov tower of $C$ are square-zero
		extensions. Since representable functors preserve limits and $\mca$
		is left complete, we therefore obtain
		\[
		\operatorname{Map}_{\calg(\mca)_{A/}}(B,C)
		\simeq
		\operatorname{Map}_{\calg(\mca)_{A/}}(B,\pi_0C).
		\]
		This proves the result.
	\end{proof}
	\begin{rem}
		From the argument of \cref{mappetale}, we see that the flat condition on $f:A\to B$ can be removed if both $A, B$ are connective.
	\end{rem}
	\begin{prop}\label{truneq}
		Assume that $\mc{A}$ is Grothendieck and hypercomplete. Let $A\in \operatorname{CAlg}(\mathcal{A}_{\geq0})_{\leq n+1}$ be $(n+1)$-truncated connective. Then the truncation functor $\tau_{\leq n}: \operatorname{CAlg}(\mathcal{A})_{A /} \rightarrow \operatorname{CAlg}(\mathcal{A})_{\tau_{\leq n} A /}$ restricts to:
		\enu{\item An equivalence $\operatorname{CAlg}(\mathcal{A})_{A /}^{\text{fl},L\text{-}et}\xrightarrow{\sim} \operatorname{CAlg}(\mathcal{A})_{\tau_{\leq n} A /}^{\text{fl},L\text{-}et}$ from the $\infty$-category of flat L-étale  commutative $A$-algebras to the $\infty$-category of flat L-étale  commutative $\tau_{\leq n} A$-algebras.
			\item An equivalence $\operatorname{CAlg}(\mathcal{A})_{A /}^{et}\xrightarrow{\sim} \operatorname{CAlg}(\mathcal{A})_{\tau_{\leq n} A /}^{et}$ from the $\infty$-category of étale  commutative $A$-algebras to the $\infty$-category of étale  commutative $\tau_{\leq n} A$-algebras.
		}
	\end{prop}
	\begin{proof}
		It follows by combining \cref{flatpush}(2), \cref{prop8.22}(2) and \cref{sqet}.
	\end{proof}
	\begin{prop}[See \cite{ha} 7.4.3.17]\label{ha74}
		Let $f: A \rightarrow B$ be a morphism in $\operatorname{CAlg}(\mathcal{A}_{\geq0})$. Suppose that $n \geq 0$ and that $f$ induces an equivalence $\tau_{\leq n} A \rightarrow \tau_{\leq n} B$. Then $\tau_{\leq n} L_{B / A} \simeq0$ .
	\end{prop}
	\begin{cor}\label{ha7432}
		Let $f: A \rightarrow B$ be a map in $\operatorname{CAlg}(\mathcal{A}_{\geq0})$. Assume that  $\op{cofib}(f)$ is $n$-connective, for $n \geq 0$. Then the relative cotangent complex $L_{B / A}$ is $n$-connective. The converse holds provided that $f$ induces an isomorphism $\pi_0 A \rightarrow \pi_0 B$.
	\end{cor}
	\begin{prop}\label{trun1}
		Assume that \mca is hypercomplete.	Let $f:A\to B\in \calg(\mc{A}_{\geq0})$. Then:
		\enu{
			\item $f:A\to B$ is L-étale if $\tau_{\leq n}f: \tau_{\leq n}A\to \tau_{\leq n}B$ is L-étale for any $n\geq0$.
			\item Assume further that $\mc{A}$ is Grothendieck. Then $f:A\to B$ is flat if and only if $\tau_{\leq n}f: \tau_{\leq n}A\to \tau_{\leq n}B$ is flat for any $n\geq0$.
			
		}
		
	\end{prop}
	
	\begin{proof}
		(1) For any $n\geq0$, from the composition $A\to B \to \tau_{ \leq n}B$ we have the cofiber sequence $$\tau_{\leq n}B\otimes_BL_{B / A}\to L_{\tau_{\leq n}B / A}\to L_{\tau_{\leq n}B / B} .$$ 
		Since $\tau_{\leq n}L_{\tau_{\leq n}B / B}=0$ by \cref{ha74}, we get that $\tau_{\leq n-1}(\tau_{\leq n}B\otimes_BL_{B / A})\simeq \tau_{\leq n-1}L_{\tau_{\leq n}B / A}$. Now consider another cofiber sequence induced by the composition $A\to \tau_{ \leq n}A \to \tau_{ \leq n}B$ 
		$$\tau_{\leq n}B\otimes_{\tau_{\leq n}A}L_{\tau_{\leq n}A / A}\to L_{\tau_{\leq n}B / A}\to L_{\tau_{\leq n}B / \tau_{\leq n}A} .$$
		The cotangent complex $L_{\tau_{\leq n}B / \tau_{\leq n}A}$ above vanishes by  assumption and $\tau_{\leq n}L_{\tau_{\leq n}A / A}=0$ by \cref{ha74}, so $\tau_{\leq n} L_{\tau_{\leq n}B / A}=\tau_{\leq n}(\tau_{\leq n}B\otimes_{\tau_{\leq n}A}L_{\tau_{\leq n}A / A})=0$. Then combining \cref{lemma4.3}(1) and equations above we get $$\tau_{\leq n-1}(L_{B / A})\simeq\tau_{\leq n-1}(\tau_{\leq n}B\otimes_BL_{B / A})\simeq \tau_{\leq n-1} L_{\tau_{\leq n}B / A}=0.$$
		By the hypercompleteness, we get $L_{B / A}=0$.
		\\(2) The ``only if'' direction can be deduced by \cref{flatpush}. Now suppose $\tau_{\leq n}f: \tau_{\leq n}A\to \tau_{\leq n}B$ is flat for any $n\geq0$. Since $B$ is connective, it suffices to show that given any discrete $M\in \modu_A(\mc{A})^{\heartsuit}$ we have $B\otimes_A M\in \modu_B(\mc{A})^{\heartsuit}$ is discrete too. Now by \cref{lemma4.3}(1), we have $$\tau_{\leq n}(B\otimes_A M)\simeq \tau_{\leq n}(\tau_{\leq n}B\otimes_A M). $$
		Also we have 
		$$\tau_{\leq n}(\tau_{\leq n}B\otimes_A M)\simeq \tau_{\leq n}(\tau_{\leq n}B\otimes_{\tau_{\leq n}A}\tau_{\leq n}A\otimes_{A} M)\simeq \tau_{\leq n}B\otimes_{\tau_{\leq n}A}\tau_{\leq n}(\tau_{\leq n}A\otimes_{A}M)$$
		where the second equality comes from the flatness of $\tau_{\leq n}f$. Note that  $$\tau_{\leq n}B\otimes_{\tau_{\leq n}A}\tau_{\leq n}(\tau_{\leq n}A\otimes_{A}M)\simeq \tau_{\leq n}B\otimes_{\tau_{\leq n}A}\tau_{\leq n}M. $$
		Combining these we get an equivalence $\tau_{\leq n}(B\otimes_A M)\simeq\tau_{\leq n}B\otimes_{\tau_{\leq n}A}\tau_{\leq n}M$, then by the flatness of $\tau_{\leq n}f$ again we conclude that for any $n\geq 0$, $\tau_{\leq n}(B\otimes_A M)$ is discrete. Hence $B\otimes_A M$ is discrete by the hypercompleteness.
	\end{proof}

	We mimic the proof of \cite[Theorem 7.4.3.18]{ha} with light modification to get the following statements.
	\begin{prop}\label{lfinite}
		Suppose that $\mc{A}$ is right complete and that $\mc{A}_{\geq0}$ is compactly generated. Let $A\in \calg(\mc{A}_{\geq0})$, and let $B$ be a connective $\mathbb{E}_{\infty}$-algebra over $A$. Then:
		\enu{
			\item  If $B$ is locally of finite presentation over $A$, then $L_{B / A}$ is perfect as a $B$-module. The converse holds provided that $\mc{A}_{\geq0}^\otimes$ is projectively rigid and that $\pi_0 B$ is finitely presented as a  commutative $\pi_0 A$-algebra in the sense of \cref{ringfp}.
			\item  If $B$ is almost of finite presentation over $A$, then $L_{B / A}$ is almost perfect as a $B$-module. The converse holds provided that $\mc{A}_{\geq0}$ is  projectively generated\footnote{projective rigidity is not needed here} and that $\pi_0 B$ is finitely presented as a  commutative $\pi_0 A$-algebra in the sense of \cref{ringfp}.
		}
		
	\end{prop}
	\begin{proof}
		We first prove the forward implications. It will be convenient to phrase these results in a slightly more general form. Suppose given a commutative diagram $\sigma$: $$\begin{tikzcd}
			& A \arrow[ld, "f"] \arrow[rd, "h"'] &   \\
			B \arrow[rr, "g"] &                                    & C
		\end{tikzcd}$$
		in $\calg(\mc{A}_{\geq0})$, and let $F(\sigma)=L_{B / A} \otimes_B C$. We will show:\\
		$\left(1^{\prime}\right)$ If $B$ is locally of finite presentation as an $\mathbb{E}_{\infty}$-algebra over $A$, then $F(\sigma)$ is perfect as a $C$-module.\\
		$\left(2^{\prime}\right)$ If $B$ is almost of finite presentation as an $\mathbb{E}_{\infty}$-algebra over $A$, then $F(\sigma)$ is almost perfect as a $C$-module.
		
		We will obtain the forward implications of (1) and (2) by applying these results in the case $B=C$.
		We first observe that the construction $\sigma \mapsto F(\sigma)$ defines a functor $\mathrm{CAlg}(\mc{A})_{A / / C}\to \operatorname{Mod}_C(\mc{A})$. Note that the functor $F$ can be identified with the fiber of the relative adjunction
		$$\begin{tikzcd}
			{\op{Fun}(\Delta^1,\calg(\mc{A})_{A/}) } \arrow[rd] \arrow[r, shift left] & T_{\calg(\mc{A})_{A/}} \arrow[d] \arrow[r, "\sim"] \arrow[l, shift left] & \operatorname{Mod}(\modu_A(\mc{A})) \arrow[ld] \\
			& \calg(\mc{A})_{A/}                                                       &                                      
		\end{tikzcd}$$
		on $C\in \calg(\mc{A})_{A/}$, we deduce that this functor preserves colimits. Since the collection of finitely presented $C$-modules is closed under finite colimits and retracts, it will suffice to prove $(1^{\prime})$ in the case where $B=\operatorname{Sym}_A^* M$ for some connective perfect $A$-module $M$. Using Proposition \cite[Proposition 7.4.3.14]{ha}, we deduce that $F(\sigma) \simeq M \otimes_A C$ is a perfect $C$-module, as desired.
		
		We now prove $\left(2^{\prime}\right)$. By \cite[Corollary 5.5.7.4]{htt}, for any $n\geq 2$ there exists a finitely presented  \ein-$A$-algebra $B'\in \calg(\mc{A}_{\geq0})_{A/}$ such that $\tau_{ \leq n}B$ is a retraction of $\tau_{ \leq n}B'$ as  commutative $A$-algebras. Note that the retraction can be lifted in $\calg(\mc{A}_{\geq0})_{A//\tau_{ \leq n}C}$ by \cite[ \href{https://kerodon.net/tag/04KB}{04KB}]{lurie2025kerodon}, as the following. $$\begin{tikzcd}
			\tau_{ \leq n}B' \arrow[r, "r", shift left] \arrow[rd] & \tau_{ \leq n}B \arrow[d] \arrow[l, "i", shift left] \\
			& \tau_{ \leq n}C                                     
		\end{tikzcd}$$ Now consider the diagram
		$$
		\begin{tikzcd}
			B' \arrow[r] & \tau_{ \leq n}B' \arrow[r, "r"]  & \tau_{ \leq n}B \arrow[r] & \tau_{ \leq n}C \\
			& A \arrow[r] \arrow[u] \arrow[lu] & B \arrow[r] \arrow[u]     & C \arrow[u]    
		\end{tikzcd}
		$$
		We claim that $\tau_{\leq  n-2}(L_{B / A} \otimes_B C)$ is a retraction of
		$\tau_{\leq n-2}(L_{B^{\prime} / A} \otimes_{B^{\prime}} C) 
		$. However, assertion $(1^{\prime})$ implies that $L_{B^{\prime} / A} \otimes_{B^{\prime}} C$ will be perfect because $B^{\prime}$ is locally of finitely presentation as a   \ein-$A$-algebra. Then $L_{B^{\prime} / A} \otimes_{B^{\prime}} C$ is perfect as a retraction of a perfect module and $L_{B / A} \otimes_{B} C$ is almost perfect.
		Now using \cref{ha74}, we see that $L_{\tau_{ \leq n}B/B}$ and $L_{\tau_{ \leq n}B'/B'}$ are $n$-connective, thus we have the natural equivalences $$\tau_{ \leq n-2}(L_{B/A}\otimes_B \tau_{ \leq n}B)\xrightarrow{\sim} \tau_{ \leq n-2}L_{\tau_{ \leq n}B/A} \,\,\,,\,\,\, \tau_{ \leq n-2}(L_{B'/A}\otimes_{B'} \tau_{ \leq n}B')\xrightarrow{\sim} \tau_{ \leq n-2}L_{\tau_{ \leq n}B'/A}.$$ So $$\tau_{ \leq n-2}(L_{B/A}\otimes_B C)\xrightarrow{\sim}\tau_{ \leq n-2}(L_{B/A}\otimes_B \tau_{ \leq n}C)\xrightarrow{\sim}\tau_{ \leq n-2}(L_{\tau_{ \leq n}B/A}\otimes_{\tau_{ \leq n}B}\tau_{ \leq n}C)$$ are equivalences by \cref{lemma4.3} (1). By assumption we have that $\tau_{ \leq n-2}(L_{\tau_{ \leq n}B/A}\otimes_{\tau_{ \leq n}B}\tau_{ \leq n}C)$ is a retraction of $\tau_{ \leq n-2}(L_{\tau_{ \leq n}B'/A}\otimes_{\tau_{ \leq n}B'}\tau_{ \leq n}C)$. Again by \cref{lemma4.3} (1), we get the equivalences $$\tau_{ \leq n-2}(L_{B'/A}\otimes_{B'} C)\xrightarrow{\sim}\tau_{ \leq n-2}(L_{B'/A}\otimes_{B'} \tau_{ \leq n}C)\xrightarrow{\sim}\tau_{ \leq n-2}(L_{\tau_{ \leq n}B'/A}\otimes_{\tau_{ \leq n}B'}\tau_{ \leq n}C).$$ Combining these, we in fact conclude that $\tau_{\leq  n-2}(L_{B / A} \otimes_B C)$ is a retraction of
		$\tau_{\leq n-2}(L_{B^{\prime} / A} \otimes_{B^{\prime}} C) 
		$.
		
		We now prove the reverse implication of (2). Assume that $L_{B / A}$ is almost perfect and that $\pi_0 B$ is a finitely presented as a  commutative $\pi_0 A$-algebra. To prove (2), it will suffice to construct a sequence of maps
		$$
		A \rightarrow B(-1) \rightarrow B(0) \rightarrow B(1) \rightarrow \ldots \rightarrow B
		$$
		such that each $B(n)$ is locally of finite presentation as a commutative  $A$-algebra, and each map $f_n: B(n) \rightarrow B$ is $(n+1)$-connective. We begin by constructing $B(-1)$ with an even stronger property: the map $f_{-1}$ induces an isomorphism $\pi_0 B(-1) \rightarrow \pi_0 B$. By \cref{appro}, there exists compact projective $A$-modules $M,N$ and a diagram 
		
		$$\begin{tikzcd}
			\op{Sym}^*_A(N) \arrow[r, "\alpha"] \arrow[d, "\phi"] & A \arrow[d] \\
			\op{Sym}^*_A(M) \arrow[r]                & B          
		\end{tikzcd}$$
		such that the map $B(-1)\to B$ induces an equivalence on $\pi_0 B$, where we take $B(-1)$ as the pushout of above diagram.
		
		We now proceed in an inductive fashion. Assume that we have already constructed a connective commutative  $A$-algebra $B(n)$ which is of finite presentation over $A$, and an $(n+1)$-connective morphism $f_n$ : $B(n) \rightarrow B$ of commutative $A$-algebras. Moreover, we assume that the induced map $\pi_0 B(n) \rightarrow \pi_0 B$ is an isomorphism (if $n \geq 0$ this is automatic; for $n=-1$ it follows from the specific construction given above). We have a fiber sequence of $B$-modules
		$$
		L_{B(n) / A} \otimes_{B(n)} B \rightarrow L_{B / A} \rightarrow L_{B / B(n)}
		$$
		By assumption, $L_{B / A}$ is almost perfect. Assertion $(1')$ implies that $L_{B(n) / A} \otimes_{B(n)} B$ is perfect. Using \cref{aper}, we deduce that the relative cotangent complex $L_{B / B(n)}$ is almost perfect. Moreover, \cref{ha74} ensures that $L_{B / B(n)}$ is $(n+2)$-connective. It follows that $\pi_{n+2} L_{B / B(n)}$ is a compact module over $\pi_0 B$. Using \cite[ Theorem 7.4.3.12]{ha} and the isomorphism $\pi_0 B(n) \rightarrow \pi_0 B$, we deduce that the canonical map
		$$
		\pi_{n+1} \mathrm{fib}(f_n) \rightarrow \pi_{n+2} L_{B / B(n)}
		$$
		is an isomorphism. Choose a compact projective $B(n)$-module $M$ and a map $M[n+1] \rightarrow \operatorname{fib}(f_n)$ such that the composition
		$$
		\pi_0 M \simeq \pi_{n+1} M[n+1] \rightarrow \pi_{n+1} \mathrm{fib}(f) \simeq \pi_{n+2} L_{B / B(n)}
		$$
		is epimorphic. By construction, we have a commutative diagram of $B(n)$-modules
		$$\begin{tikzcd}
			{M[n+1]} \arrow[d] \arrow[r] & 0 \arrow[d] \\
			B(n) \arrow[r]               & B          
		\end{tikzcd}$$
		Adjoint to this, we obtain a diagram
		in $\calg(\mc{A}_{\geq0})_{A/}$.
		$$\begin{tikzcd}
			{\operatorname{Sym}_{B(n)}^*M[n+1]} \arrow[d] \arrow[r] & B(n) \arrow[d] \\
			B(n) \arrow[r]                                          & B             
		\end{tikzcd}$$
		
		We now define $B(n+1)$ to be the pushout
		$$
		B(n) \otimes_{\operatorname{Sym}_{B(n)}^* M[n+1]} B(n),
		$$
		and $f_{n+1}: B(n+1) \rightarrow B$ to be the induced map. It is clear that $B(n+1)$ is locally of finite presentation over $B(n)$, and therefore locally of finite presentation over $A$ (\cref{compof.p}). To complete the proof of (2), it will suffice to show that the fiber of $f_{n+1}$ is $(n+2)$-connective.
		
		By construction, we have a commutative diagram
		$$
		\begin{tikzcd}
			& \pi_0B(n+1) \arrow[rd, "e''"] &        \\
			\pi_0B(n) \arrow[ru, "e'"] \arrow[rr, "e"] &                               & \pi_0B
		\end{tikzcd}$$
		where the map $e^{\prime}$ is epimorphic and $e$ is isomorphic. It follows that $e^{\prime}$ and $e^{\prime \prime}$ are also isomorphic. In view of \cref{ha7432}, it will now suffice to show $L_{B / B(n+1)}$ is $(n+3)$-connective. We have a fiber sequence of $B$-modules
		$$
		L_{B(n+1) / B(n)} \otimes_{B(n+1)} B \rightarrow L_{B / B(n)} \rightarrow L_{B / B(n+1)}
		$$
		Using \cite[Proposition 7.4.3.14]{ha}  and  \cite[Proposition 7.3.3.7]{ha}, we conclude that $L_{B(n+1) / B(n)}$ is canonically equivalent to $M[n+2] \otimes_{B(n)} B(n+1)$. We may therefore rewrite our fiber sequence as
		$$
		M[n+2] \otimes_{B(n)} B \rightarrow L_{B / B(n)} \rightarrow L_{B / B(n+1)} .
		$$
		The inductive hypothesis and  \cref{ha7432} guarantee that $L_{B / B(n)}$ is $(n+2)$-connective. The $(n+3)$-connectiveness of $L_{B / B(n+1)}$ is therefore equivalent to the surjectivity of the map
		$$
		\pi_0 M \simeq \pi_{n+2}\left(M[n+2] \otimes_{B(n)} B\right) \rightarrow \pi_{n+2} L_{B / B(n)}
		$$
		which is evident from our construction. This completes the proof of (2).
		
		To complete the converse of (1), we use the same strategy but make a more careful choice of $M$. Let us assume that $L_{B / A}$ is perfect. It follows from the above construction that each cotangent complex $L_{B / B(n)}$ is likewise perfect. Using  \cref{toram}, we may assume $L_{B / B(-1)}$ is of Tor-amplitude $\leq k+2$ for some $k \geq 0$. Moreover, for each $n \geq 0$ we have a fiber sequence of $B$-modules
		$$
		L_{B / B(n-1)} \rightarrow L_{B / B(n)} \rightarrow P[n+2] \otimes_{B(n)} B,
		$$
		where $P$ is compact projective by our construction, and therefore of Tor-amplitude $\leq 0$. Using  \cref{toram} and induction on $n$, we deduce that the Tor-amplitude of $L_{B / B(n)}$ is $\leq k+2$ for $n \leq k$. In particular, the $B$-module $\overline{M}=L_{B / B(k)}[-k-2]$ is connective and has Tor-amplitude $\leq 0$. It follows from  \cref{rem7.2.4.22} that $\overline{M}$ is a flat $B$-module. Invoking  \cref{perfflat}\footnote{This is a key fact which requires projective rigidity.}, we conclude that $\overline{M}$ is a compact projective $B$-module. Using  \cref{eqproj}, we can choose a compact projective $B(k)$-module $M$ and an equivalance $M[k+2] \otimes_{B(k)} B \simeq L_{B / B(k)}$. Using this map in the construction outlined above, we guarantee that the relative cotangent complex $L_{B / B(k+1)}$ vanishes. It follows from Corollary 7.4.3.4 (which also works in our general setting) that the map $f_{k+1}: B(k+1) \rightarrow B$ is an equivalence, so that $B$ is locally of finite presentation as an $\mathbb{E}_{\infty}$-algebra over $A$, as desired.
	\end{proof}
	
	\begin{cor}\label{trun2}
		Suppose that the $ttt$-\infcat $(\mc{A}^\otimes, \mc{A}_{\geq0}) $ is projectively rigid.  Let $f: A\to B \in \calg(\ageq)$. Then $f$ is étale if and only if $\tau_{\leq n}f: \tau_{\leq n}A\to \tau_{\leq n}B$ is étale for every $n\geq0$.
	\end{cor}
	\begin{proof}
		It follows immediately by combining \cref{trun1} and \cref{lfinite}.
		
	\end{proof}
	\begin{cor}\label{compafp}
		Suppose that the $ttt$-$\infty$-category $(\mathcal{A}^\otimes, \mathcal{A}_{\geq0})$ is projectively rigid. Let $f: A \to B$ be a morphism of discrete commutative algebras in $\operatorname{CAlg}(\mathcal{A}^\heartsuit)$. If $f$ is $L$-étale, then $f$ is finitely presented in the sense of \cref{ringfp} if and only if it is finitely presented in the sense of \cref{einftyfp}.
	\end{cor}
	\begin{proof}
		This is an immediate consequence of \cref{lfinite}(1).
	\end{proof}
	\begin{rem}\label{fpcounterex}
		Note that the $L$-étale condition in \cref{compafp} can not be removed. 
		For example, let \(A=k\) where \(k\) is any field, and let
		\[
		B = k[x,y]/(x^2,xy,y^2).
		\]
		The \(k\)-algebra \(B\) is finite-dimensional, hence in particular is a finitely presented map of ordinary commutative rings.
		Let \(I=(x^2,xy,y^2)\subset k[x,y]\). At the origin the embedding dimension is \(2\) (since \(\mathfrak m/\mathfrak m^2\) is spanned by \(x,y\)), but the ideal \(I\) requires at least \(3\) generators. Therefore \(k\to B\) is not a local complete intersection (lci) map.
		Consequently, the relative cotangent complex \(L_{B/k}\) is not a perfect \(B\)-module by Avramov theorem \cite{avramov1999locally}, and hence $k\to B$ is \emph{not} finitely presented as a map of \ein-rings by \cref{lfinite}(1). 
		
	\end{rem}
	Now we can give the proof of our étale rigidity. Our proof is parallel with the proof of \cite[Theorem 3.4.1]{dag4}.
	\begin{proof}[Proof of  \cref{etrig}:]
		(1) First, using \cref{l5}(3), we may reduce to the case where $A$ is connective. For each $0 \leq n \leq \infty$, let $\mathcal{C}_n$ denote the full subcategory of $\operatorname{Fun}(\Delta^1, \calg(\ageq))$ spanned by those morphisms $f: B \rightarrow B^{\prime}$ such that $B$ and $B^{\prime}$ are connective and $n$-truncated, and let $\mathcal{C}_n^{{fl},L\text{-}et}$ denote the full
		subcategory of $\mathcal{C}_n$ spanned by those morphisms which are also flat and L-étale. Using the left completeness, we deduce that $\mathcal{C}_{\infty}$ is the homotopy inverse limit of the tower
		$$
		\ldots \rightarrow \mc{C}_2 \xrightarrow{\tau_{ \leq 1}} \mc{C}_1 \xrightarrow{\tau_{ \leq 0}} \mc{C}_0 .
		$$
		Using  \cref{trun1}, we deduce that $\mathcal{C}_{\infty}^{{fl},L\text{-}et}$ is the homotopy inverse limit of the restricted tower
		$$
		\ldots \rightarrow \mathcal{C}_2^{{fl},L\text{-}et} \rightarrow \mathcal{C}_1^{{fl},L\text{-}et} \rightarrow \mathcal{C}_0^{{fl},L\text{-}et}
		$$
		Choose a Postnikov tower
		$$
		A \rightarrow \ldots \rightarrow \tau_{\leq 2} A \rightarrow \tau_{\leq 1} A \rightarrow \tau_{\leq 0} A
		$$
		For $0 \leq n \leq \infty$, let $\mathcal{D}_n$ denote the fiber product $\mathcal{C}_n^{{fl},L\text{-}et} \times_{\calg(\ageq)}\left\{\tau_{\leq n} A\right\}$, so that we can identify $\mathcal{D}_n$ with the full subcategory $\calg(\ageq)_{\tau_{\leq n}A /}^{{fl},L\text{-}et}\subset\calg(\ageq)_{\tau_{\leq n}A /} $ spanned by the flat L-étale morphisms $f: \tau_{\leq n} A \rightarrow B$. It follows from \cref{trun1} that $\mathcal{D}_{\infty}$ is the homotopy inverse limit of the tower
		$$
		\ldots \rightarrow \mathcal{D}_2 \xrightarrow{g_1} \mathcal{D}_1 \xrightarrow{g_0} \mathcal{D}_0
		$$
		We wish to prove that the truncation functor induces an equivalence $\mathcal{D}_{\infty} \rightarrow \mathcal{D}_0$. For this, it will suffice to show that each of the functors $g_i$ is an equivalence. Consequently, it follows from \cref{truneq}.\\
		(2) The proof is totally parallel with (1). We only need to replace  ``flat L-étale'' by ``étale'' and replace ``\cref{trun1}'' by ``\cref{trun2}''.
	\end{proof}

	\section{The \infcat of projectively rigid $ttt$-$\infty$-categories}\label{appendix}
	
	Having exhaustively explored the internal higher algebraic theory in a $ttt$-$\infty$-category, we step back to analyze the moduli of such categories themselves.

	\subsection{The universal example via 1-dimensional cobordism}
	In this subsection, we prove that the $\infty$-category of projectively rigid $ttt$-$\infty$-categories is compactly generated. Astonishingly, the universal (compact) generator is precisely the presheaf category on the 1-dimensional framed cobordism category, connecting our abstract algebraic framework directly to topological field theories.
	
	
	\begin{nota}
		Let $\mc{V}\in\calg(\prl)$. We denote $\calg^{\op{rig,at}}_{\mc{V}}$ to be the full subcategory of $\calg(\prlv)$ spanned by rigid and atomically generated commutative $\mc{V}$-algebras. We refer the reader to \cite{ramzi2024dualizable,ramzi2024locally} for more details about the atomic generation and the rigidity.
	\end{nota}
	\begin{rem}
		By \cite[Lemma 4.50]{ramzi2024locally}, a projectively rigid symmetric monoidal Grothendieck prestable \infcat in the sense of \cref{projrigid} can be identified with a $\spgeq$-atomically generated rigid commutative algebra $\mcw\in\calg^{\op{rig,at}}_{\spgeq}$.
	\end{rem}
	\begin{rem}\label{chap8.3}
		By \cref{pret}, a right complete $ttt$-\infcat $(\mc{B}^\otimes,\mc{B}_{\geq0})$ can be totally recovered from the connective part $\mc{B}_{\geq0}^\otimes$ via the stabilization. So there is an equivalence between the \infcat  of projectively rigid $ttt$-\infcats  and the \infcat  of projectively rigid symmetric monoidal Grothendieck prestable \infcats $$\calg^{\op{alg}}(\pr^{t\text{-}rex}_{\op{st}})\xrightarrow{\sim}\calg^{\op{rig,at}}_{\spgeq}$$ given by $(\mc{B}^\otimes,\mc{B}_{\geq0})\mapsto \mc{B}_{\geq0}^\otimes .$ And the inverse is given by $\mathcal{C}^\otimes \mapsto(\operatorname{Sp}(\mathcal{C})^\otimes, \operatorname{Sp}(\mathcal{C})_{\geq 0})$, see \cref{pret}.
	\end{rem}

	We can see the presentability from the following result.
	\begin{prop}
		For any $\mcv\in\calg(\prl)$, $\calg^{\op{rig,at}}_{\mc{V}}$ is presentable.
	\end{prop}
	\begin{proof}
		To see that, firstly we have that $\calg^{\op{rig,at}}_{\mc{V}}= \calg_{\mc{V}}^{\op{rig}} \times_{\prv^{\op{dbl}}} \prv^{\op{at}}$ is accessible by combining  \textup{\cite[ Corollary 3.15]{ramzi2024dualizable} and \cite[Corollary 5.13, 5.14]{ramzi2024locally}}. Then the presentability follows from that $\calg^{\op{rig,at}}_{\mc{V}}$ admits small colimits, which is obtained by observing the inclusion $\calg^{\op{rig,at}}_{\mc{V}}\subset \calg_{\mc{V}}$ is closed under small colimits.
	\end{proof}
	
	Alternatively, we can give a more straightforward proof of the presentability of $\calg^{\op{rig,at}}_{\mc{V}}$ in the case $\mc{V}^\otimes=\spgeq^\otimes$, and even further give a compact generator which is linked to cobordism hypothesis. 
	
	Before that, let us recall a lemma, which we learned from Germán Stefanich.
	\begin{lem}\label{coreconserv}
		Let $\op{Cat}_{\infty,\op{ad}}^\times$ denote the \infcat of small additive \infcats with finite product preserving functors. Then the core functor $(-)^{\simeq}: \op{Cat}_{\infty,\op{ad}}^\times\to \mc{S}$ is conservative.
	\end{lem}
	
	\begin{proof}
		For any $x,y\in \mcc$, we have a natural map of retractions:
		\[
		\begin{tikzcd}[row sep=2em, column sep=5em]
			\mathrm{Map}_{\mathcal C}(x,y)
			\arrow[r, "\matone"]
			\arrow[d, "F"']
			&
			\mathrm{Iso}_{\mathcal C}(x \oplus y, x \oplus y)
			\arrow[r, "p_{12}"]
			\arrow[d, "F"]
			&
			\mathrm{Map}_{\mathcal C}(x,y)
			\arrow[d, "F"]
			\\
			\mathrm{Map}_{\mathcal D}(Fx,Fy)
			\arrow[r, "\matoneF"]
			&
			\mathrm{Iso}_{\mathcal D}(Fx \oplus Fy, Fx \oplus Fy)
			\arrow[r, "p_{12}"]
			&
			\mathrm{Map}_{\mathcal D}(Fx,Fy)
		\end{tikzcd}
		\]
		Since  $F^\simeq:\mcc^\simeq\to \mcd^\simeq$ (regarded as a functor) is fully faithful, and since $\op{Iso}_\mcc(x\oplus y,x\oplus y)$ can be identified with the mapping anima 
		$\mathrm{Map}_{\mathcal C^\simeq}(x \oplus y, x \oplus y)$, we conclude that $F$ is fully faithful.
	\end{proof}
	\begin{rem}
		Note that the additive condition in the above lemma can not be weakened to the semi-additive, otherwise the endmorphism $\bigl(\begin{smallmatrix}
			1 & f \\
			0& 1 \\
		\end{smallmatrix}\bigr)$ is not necessarily an automorphism.
	\end{rem}
	\begin{de}\label{smallrig}
		We say  a symmetric monoidal \infcat is small rigid if it is small and every object in it is dualizable.
	\end{de}
	\begin{prop}\label{A.4}	 We have a natural equivalence $$\calg^{\op{rig,at}}_{\spgeq}\xrightarrow{\sim}\calg^{\op{rig}}(\op{Cat}^{\op{idem}}_{\infty,\op{ad}})$$ given by $\mc{C}^\otimes\mapsto(\mc{C}^{\op{cproj}})^{\otimes}$, where  $\calg^{\op{rig}}(\op{Cat}^{\op{idem}}_{\infty,\op{ad}})$ denotes the full subcategory of $\calg(\op{Cat}^{\op{idem}}_{\infty,\op{ad}})$ spanned by small rigid idempotent-complete additively symmetric monoidal $\infty$-categories. 	
		
	\end{prop}
	\begin{proof}
		It suffices to observe that any colimit-preserving symmetric monoidal functor between projectively rigid commutative algebras $\botimes_{\geq0}\to \cotimes_{\geq0}$ preserves compact projective objects.
	\end{proof}
	
	Now let us recall the (1-dimensional) cobordism hypothesis, which was originally formulated in \cite{baez1995higher} and was proved by Hopkins--Lurie in \cite{lurie2008classification}.
	\begin{thm}[\cite{lurie2008classification}\label{cobhypo} Cobordism hypothesis of dim=1] Let $\mb{Cob}_1^\otimes$ denote the oriented 1-dimensional cobordism $(\infty,1)$-category with the symmetric monoidal structure given by disjoint union. Then $\mb{Cob}_1^\otimes$ is small rigid and satisfies the following universal property:\\
		Let $\mathcal{C}$ be a symmetric monoidal $(\infty, 1)$-category. Then the evaluation functor $Z \mapsto Z(*)$ induces an equivalence of \infcats
		$$
		\operatorname{Fun}^{\otimes}(\mb{Cob}_1, \mathcal{C}) \rightarrow (\mathcal{C}^d)^{\simeq}  
		$$
		where $\operatorname{Fun}^{\otimes}$ denotes the \infcat of symmetric monoidal functors.
	\end{thm}
	\begin{rem}
		Note that the  1-dimensional oriented and framed cobordism \infcats are equivalent $\mb{Cob}_1 \simeq \mb{Bord}_1^{\mathrm{fr}}$ \cite[see][\textsection4.2]{lurie2008classification}, but that does not hold in higher dimensional cases.
	\end{rem}
	We now prove the main theorem of this subsection.
	\begin{thm}[Universal example]\label{uniexamp}
		The \infcat $\calg^{\op{rig,at}}_{\spgeq}$ is compactly generated by a single element $$\funct(\mb{Cob}_1^{\op{op}},\spgeq)^\otimes\in\calg^{\op{rig,at}}_{\spgeq} $$where the symmetric monoidal structure  $\funct(\mb{Cob}_1^{\op{op}},\spgeq)^\otimes$ is given by Day convolution.
		
	\end{thm}
	\begin{proof}
		By \cref{A.4}, we have an equivalence $$\calg^{\op{rig,at}}_{\spgeq}\xrightarrow{\sim}\calg^{\op{rig}}(\op{Cat}^{\op{idem}}_{\infty,\op{ad}}).$$
		We first observe that by \cref{cobhypo}, the composite $$\calg^{\op{rig,at}}_{\spgeq}\xrightarrow{\sim}\calg^{\op{rig}}(\op{Cat}^{\op{idem}}_{\infty,\op{ad}})\hookrightarrow \calg(\op{Cat}^{\op{idem}}_{\infty,\op{ad}})\xrightarrow{(-)^\simeq} \mcs$$ is represented by $\funct(\mb{Cob}_1^{\op{op}},\spgeq)^\otimes\in\calg^{\op{rig,at}}_{\spgeq}$. 
		Since the second functor preserves small colimits by \cite[Proposition 4.53]{ramzi2024locally} and the third functor preserves filtered colimits, we see that $\funct(\mb{Cob}_1^{\op{op}},\spgeq)^\otimes$ is compact in $\calg^{\op{rig,at}}_{\spgeq}$. Moreover, the representable functor by $\funct(\mb{Cob}_1^{\op{op}},\spgeq)^\otimes$ is conservative by \cref{coreconserv}. Consequently, it is a compact generator of $\calg^{\op{rig,at}}_{\spgeq}$.
	\end{proof}
	
	\subsection{Algebraic functors}
	We  formalize the morphisms between projectively rigid $ttt$-$\infty$-categories, termed algebraic functors. In this subsection, we demonstrate that an algebraic functor preserve all the essential algebraic properties we have defined: flatness, projective modules, finite presentation, and étale maps.
	\begin{de}\label{algttt}
		We call a right $t$-exact colimit-preserving symmetric monoidal functor $(\mc{B}^\otimes,\mc{B}_{\geq0})\to(\mc{C}^\otimes,\mc{C}_{\geq0})$ between projectively rigid $ttt$-\infcats by an algebraic functor.
	\end{de}
	\begin{prop}\label{modulepropertiespreser}
		Let $F:(\mc{B}^\otimes,\mc{B}_{\geq0})\to(\mc{C}^\otimes,\mc{C}_{\geq0})$ be an algebraic functor between projectively rigid $ttt$-$\infty$-categories, and let $R\in\alg(\mc{B}_{\geq0})$. Then the functor $\lmodu_{R}(\mc{B}_{\geq0})\to\lmodu_{F(R)}(\mc{C}_{\geq0})$ preserves 
		\enu{
			\item compact projectives;
			\item compacts; 
			\item projectives; 
			\item flats; 
			\item almost perfects. }
	\end{prop}
	\begin{proof}
		(1)-(3) are obvious. (4) and (5) follow by \cref{cocompfil} and \cref{cocompsimplicial}.
	\end{proof}
	\begin{lem}\label{functorialcotangent}
		Let $F:\botimes\to\cotimes\in\calg(\prlst)$. Then for any $A\to B\in \calg(\mc{B})$, we have a natural equivalence $$F(L_{B/A})\simeq L_{F(B)/F(A)}$$ in $\modu_{F(B)}(\mcc)$.
	\end{lem}
	\begin{proof}
		We first  observe that the equivalence in \cite[Theorem 7.3.4.18]{ha} is natural, i.e., we have the following diagram:
		$$\begin{tikzcd}[row sep=1.5em, column sep=1.5em]
			& {\funct(\Delta^1,\calg(\mcc))} \arrow[ld, "{\funct(\Delta^1,\calg(G))}"'] \arrow[rdd] &                                                                  & T_{\calg(\mcc)} \arrow[ld] \arrow[rr, "\sim"] \arrow[ll, "p_2"] \arrow[ldd] &                           & \modu(\mcc) \arrow[ld, "\modu(G)"] \arrow[llldd] \\
			{\funct(\Delta^1,\calg(\mcb))} \arrow[rdd] &                                                                                       & T_{\calg(\mcb)} \arrow[rr, "\sim"] \arrow[ll, "p_1"] \arrow[ldd] &                                                                             & \modu(\mcb) \arrow[llldd] &                                                  \\
			&                                                                                       & \calg(\mcc) \arrow[ld]                                           &                                                                             &                           &                                                  \\
			& \calg(\mcb)                                                                           &                                                                  &                                                                             &                           &                                                 
		\end{tikzcd}$$
		By taking the left adjoints of the diagram, we obtain the following diagram:
		$$\begin{tikzcd}
			\calg(\mcb) \arrow[r, "\delta"] \arrow[d, "\calg(F)"] & {\funct(\Delta^1,\calg(\mcb))} \arrow[d, "{\funct(\Delta^1,\calg(F))}"] \arrow[r, "p_1^L"] & T_{\calg(\mcb)} \arrow[d, "T_{\calg(G)}^L"] \arrow[r, "\sim"] & \modu(\mcb) \arrow[d, "\modu(F)"] \\
			\calg(\mcc) \arrow[r, "\delta"]                       & {\funct(\Delta^1,\calg(\mcc))} \arrow[r, "p_2^L"]                                          & T_{\calg(\mcc)} \arrow[r, "\sim"]                             & \modu(\mcc)                      
		\end{tikzcd}$$
		which exactly means $F(L_A)\simeq L_{F(A)}$ for any $A\in \calg(\mcb)$. Since the following functor preserves cocartesian edges,
		$$\begin{tikzcd}
			\modu(\mcb)  \arrow[r, "\modu(F) "] \arrow[d] & \modu(\mcc)  \arrow[d] \\
			\calg(\mcb)  \arrow[r, "\calg(F) "]           & \calg(\mcc)           
		\end{tikzcd}$$
		we see that $\modu(F)$ preserves relative colimits by \cite[Proposition 4.3.1.10]{htt}. Since $L_{B/A}$ is defined to be the relative pushout as the following diagram,
		$$\begin{tikzcd}
			& L_A \arrow[ld] \arrow[rr] \arrow[dd, maps to] &                             & L_B \arrow[ld] \arrow[dd, maps to] \\
			0 \arrow[rr] \arrow[dd, maps to] &                                               & L_{B/A} \arrow[dd, maps to] &                                    \\
			& A \arrow[ld] \arrow[rr]                       &                             & B \arrow[ld]                       \\
			A \arrow[rr]                     &                                               & B                           &                                   
		\end{tikzcd}$$
		we conclude that $F(L_{B/A})\simeq L_{F(B)/F(A)}$.
	\end{proof}
	\begin{prop}
		Let $F:(\mc{B}^\otimes,\mc{B}_{\geq0})\to(\mc{C}^\otimes,\mc{C}_{\geq0})$ be an algebraic functor between projectively rigid $ttt$-$\infty$-categories, and let $R\in\calg(\mc{B}_{\geq0})$. Then the functor $\calg(\mc{B}_{\geq0})_{R/}\to\calg(\mc{C}_{\geq0})_{F(R)/}$ preserves 
		\enu{
			\item finitely presented algebras;
			\item almost finitely presented algebras;
			\item flat algebras; 
			\item L-étale algebras;
			\item  étale algebras.}
	\end{prop}
	\begin{proof}
		Let $G$ denote the right adjoint to $F$.
		
		For (1), it follows from the following diagram
		$$\begin{tikzcd}[column sep=large, row sep=large]
			\mcb_{\geq0} \arrow[d, "\op{Sym}^{*}"] \arrow[r, "F_{\geq0}"] & \mcc_{\geq0} \arrow[d, "\op{Sym}^{*}"] \\
			\calg(\mcb_{\geq0}) \arrow[r, "\calg(F_{\geq0})"]             & \calg(\mcc_{\geq0})                   
		\end{tikzcd}$$
		that $\calg(F_{\geq0})$ sends free  \ein-algebras of compact objects in $\mcb_{\geq0}$ to free  \ein-algebras of compact objects in $\mcc_{\geq0}$, which are compact generators of $\calg(\mcb_{\geq0})$.
		
		For (2), it suffices to observe that for any $n\geq 0$, the functor $$\calg(\mcb_{\geq0})_{\leq n}\xrightarrow{\calg(F_{\geq0})\otimes \mcs_{\leq n}}\calg(\mcb_{\geq0})_{\leq n}$$ preserves compact objects, because $\calg(\mcb_{\geq0}) \xrightarrow{\calg(F_{\geq0})  }     \calg(\mcc_{\geq0}) $ does.
		
		For (3), it follows from \cref{modulepropertiespreser}(4).
		
		For (4), it suffices to show that for any $A\to B\in \calg(\mc{B})$, we have $F(L_{B/A})\simeq L_{F(B)/F(A)}$ in $\modu_{F(B)}(\mcc)$, but that follows from \cref{functorialcotangent}.
		
		For (5), it follows from (1), (3) and (4).
		
	\end{proof}

	\section{Questions and future directions}
	
	In this final section, we catalogue several open questions regarding étaleness. We then conclude by outlining a broader vision for future research building upon the framework established in this paper.
	
	\begin{qu}
		Assume that the $ttt$-$\infty$-category
		$(\mathcal{A}^\otimes,\mathcal{A}_{\geq0})$ is projectively rigid.
		Consider a commutative diagram in
		$\operatorname{CAlg}(\mathcal{A}_{\geq0})$
		\[
		\begin{tikzcd}
			& A \arrow[ld, "f"'] \arrow[rd, "h"] & \\
			B \arrow[rr, "g"'] & & C .
		\end{tikzcd}
		\]
		\begin{enumerate}[label=(\arabic*),font=\normalfont]
			\item
			If $f$ and $h$ are étale, is $g$ necessarily étale?
			
			By \cref{lem8.12}(1) and \cref{compof.p}, the map $g$ is
			automatically $L$-étale and finitely presented, so the only
			remaining issue is flatness. Equivalently, one may ask whether
			the diagonal
			\[
			B\otimes_A B\longrightarrow B
			\]
			of every étale morphism $A\to B$ is flat: this diagonal is
			automatically finitely presented and $L$-étale. When
			$\mathcal A=\operatorname{Sp}$, the answer is affirmative, since
			every finitely presented $L$-étale map of $\mathbb E_\infty$-rings
			is flat. In general, however, finitely presented $L$-étale maps
			need not be étale; see \cref{exetalenotflt}.
			
			\item A closely related structural question is whether every étale
			commutative $A$-algebra is separable over $A$, that is, whether the
			multiplication map
			\[
			B\otimes_A B\longrightarrow B
			\]
			admits a section as a map of $B$-bimodules. In the projectively rigid
			setting, such a separability result would provide a natural route to the
			missing flatness of the diagonal.

			\item
			Suppose that $g$ is faithfully flat and étale. Is $f$ étale if
			and only if $h$ is étale?
			
			The implication from $f$ to $h$ follows from the stability of
			étale morphisms under composition. Conversely, if $h$ is étale,
			then \cref{compoflat}(2) and \cref{lem8.12}(2) imply that $f$ is
			flat and $L$-étale. Thus the remaining question is whether finite
			presentation descends along the faithfully flat étale morphism
			$g$.
			
			\item
			Let
			\[
			B\simeq\varinjlim_i B_i
			\]
			be a filtered colimit of finitely presented commutative $A$-algebras,
			and let $B\to C$ be étale. Does there exist some $i$ and an étale
			morphism $B_i\to C_i$ such that
			\[
			C\simeq C_i\otimes_{B_i}B?
			\]
			
			Finite presentation and the $L$-étale condition can be descended to a
			sufficiently large finite stage. Thus the remaining issue is whether
			flatness of a finitely presented morphism can be detected at a finite
			stage: if $B_i\to C_i$ is finitely presented and
			\[
			B\to C_i\otimes_{B_i}B
			\]
			is flat, must
			\[
			B_j\to C_i\otimes_{B_i}B_j
			\]
			be flat for some $j\geq i$?
		\end{enumerate}
	\end{qu}

	\subsection*{Future directions: Derived geometry internal to $\mathcal{A}$}
	
	With the foundational framework of higher algebra over an arbitrary projectively rigid $ttt$-$\infty$-category $\mathcal{A}$ now firmly established, a natural and compelling next step is to systematically develop derived algebraic geometry within this setting. Just as classical derived algebraic geometry is built upon $\mathbb{E}_\infty$-rings in spectra, our scaffolding allows one to define and study derived schemes, stacks, and moduli problems modeled locally on $\mathbb{E}_\infty$-algebras internal to $\mathcal{A}$. This opens the door to importing deep geometric intuitions and intersection theory into new, exotic environments, ranging from equivariant and motivic domains to specialized analytic settings, ultimately broadening the scope and applicability of modern geometric methods.
	
	\appendix
	
	\section{Duality}
	
	For the reader's convenience, we review the general theory of dualizable objects within monoidal $\infty$-categories. 
	\subsection{Dualizable objects}\label{3.1}
	\begin{cov}
		Throughout \cref{3.1}, we fix a symmetric monoidal $\infty$-category $\mathcal{C}^\otimes \to \mathrm{N}(\op{Fin_*})$.
	\end{cov}
	
	\begin{de}
		We say an object $X\in \mathcal{C}$ is dualizable if there exists an object $X^\vee$ and a pair of morphisms
		$$
		c: \mathbf{1} \rightarrow X \otimes X^{\vee} \quad e: X^{\vee} \otimes X \rightarrow \mathbf{1}
		$$
		where $\mathbf{1}$ denotes the unit object of $\mathcal{C}$. These morphisms are required to satisfy the following conditions: The composite maps
		$$
		\begin{gathered}
			X \xrightarrow{c \otimes \mathrm{id}} X \otimes X^{\vee} \otimes X\xrightarrow{\mathrm{id } \otimes e} X \\
			X^{\vee} \xrightarrow{\mathrm{id} \otimes c} X^{\vee} \otimes X \otimes X^{\vee} \xrightarrow{e \otimes \mathrm{id}} X^{\vee}
		\end{gathered}
		$$
		are homotopic to the identity on $X$ and $X^{\vee}$, respectively.
	\end{de}
	
	We call $X\in\mathcal C$ \emph{cotensorable} if
	$(-)\otimes X$ admits a right adjoint, which we denote by
	$\underline{\operatorname{Map}}_{\mathcal C}(X,-)$. Every object of a
	presentably symmetric monoidal $\infty$-category is cotensorable.
	
	\begin{prop}\label{cri}
		Let $X\in\mathcal{C}$ be an object. Then $X$ is dualizable if and only if $X$ is cotensorable and for any $Y\in\mathcal{C}$, the natural map $\underline{\mapp}_{\mathcal{C}}(X,\mathbf{1})\otimes Y\to\underline{\mapp}_{\mathcal{C}}(X,Y)$ is an equivalence in $\mathcal{C}$.
	\end{prop}
	
	\begin{proof}
		Assume that $X$ is dualizable. Then $X$ is cotensorable since we have $\underline{\mapp}_{\mathcal{C}}(X,-)\simeq X^\vee\otimes(-)$. Particularly, $\underline{\mapp}_{\mathcal{C}}(X,\mathbf{1})\simeq X^\vee$. Now let $Y \in \mathcal{C}$. We wish to show that the composite map
		$$
		\phi: \operatorname{Map}_{\mathcal{C}}\left(-, Y \otimes X^{\vee}\right) \rightarrow \operatorname{Map}_{\mathcal{C}}\left(- \otimes X, Y \otimes X^{\vee} \otimes X\right) \xrightarrow{e} \operatorname{Map}_{\mathcal{C}}(- \otimes X, Y)
		$$
		is a homotopy equivalence. Let $\psi$ denote the composition
		$$
		\operatorname{Map}_{\mathcal{C}}(- \otimes X, Y) \rightarrow \operatorname{Map}_{\mathcal{C}}\left(- \otimes X \otimes X^{\vee}, Y \otimes X^{\vee}\right) \xrightarrow{c} \operatorname{Map}_{\mathcal{C}}\left(-, Y \otimes X^{\vee}\right) .
		$$
		Using the compatibility of $e$ and $c$, we deduce that $\phi$ and $\psi$ are homotopy inverses to one another. By the Yoneda lemma, $\phi$ can be identified with the map $\underline{\mapp}_{\mathcal{C}}(X,\mathbf{1})\otimes Y\to\underline{\mapp}_{\mathcal{C}}(X,Y)$.
		
		Assume that $X$ is cotensorable and that for any $Y\in\mathcal{C}$, the natural map $$\underline{\mapp}_{\mathcal{C}}(X,\mathbf{1})\otimes Y\to\underline{\mapp}_{\mathcal{C}}(X,Y)$$ is an equivalence in $\mathcal{C}$. Particularly, we have an equivalence $\underline{\mapp}_{\mathcal{C}}(X,\mathbf{1})\otimes X\xrightarrow{\sim}\underline{\mapp}_{\mathcal{C}}(X,X)$. Let $c:\mathbf{1}\to \underline{\mapp}_{\mathcal{C}}(X,\mathbf{1})\otimes X$ be the inverse image of the identity map $\mathrm{id}:\underline{\mapp}_{\mathcal{C}}(X,X) \to \underline{\mapp}_{\mathcal{C}}(X,X)$. Then it is straightforward to check that the counit $e: \underline{\mapp}_{\mathcal{C}}(X,\mathbf{1})\otimes X\to \mathbf{1}$ and $c:\mathbf{1}\to \underline{\mapp}_{\mathcal{C}}(X,\mathbf{1})\otimes X$ form a duality datum.
	\end{proof}
	
	\begin{prop}
		Let $\mathcal{C}^d\subset\mathcal{C}$ be the full subcategory consisting of dualizable objects. Then the profunctor $\mathcal{C}^d\times\mathcal{C}^d\to\mathcal{S}$ given by $\operatorname{Map}_{\mathcal{C}}\left(\mathbf{1}, - \otimes -\right)$ is a balanced profunctor \textup{(see \cite[\href{https://kerodon.net/tag/03MM}{03MM}]{lurie2025kerodon})}, which induces a natural equivalence of $\infty$-categories $(-)^\vee=:\underline{\mapp}_{\mathcal{C}}(-,\mathbf{1}): (\mathcal{C}^d)^{\op{op}}\xrightarrow{\sim}\mathcal{C}^d$. Furthermore, $(-)^{\vee\vee}\simeq\op{Id}_{\mathcal{C}^d}$ is equivalent to the identity functor.
	\end{prop}
	
	\begin{proof}
		It suffices to observe that if $c:\mathbf{1}\to X\otimes Y$ is part of a duality datum for $X$, then it is also part of a duality datum for $Y$.
	\end{proof}
	
	\begin{rem}
		In fact, this perfect pairing can be enhanced to a symmetric monoidal perfect pairing and hence induces an equivalence of symmetric monoidal $\infty$-categories $(-)^\vee=:\underline{\mapp}_{\mathcal{C}}(-,\mathbf{1}): (\mathcal{C}_d^{\op{op}})^\otimes\xrightarrow{\sim}(\mathcal{C}^d)^\otimes$, \textup{see \cite[Proposition 3.2.4]{ec1}}.
	\end{rem}
	
	\begin{prop}
		The full subcategory $\mathcal{C}^d\subset\mathcal{C}$ is closed under tensor product, hence it forms a symmetric monoidal full subcategory.
	\end{prop}
	
	\begin{proof}
		Let $X,Y\in \mathcal{C}^d$. Choosing $c=c_X\otimes c_Y:\mathbf{1}\simeq \mathbf{1}\otimes\mathbf{1}\to (X\otimes X^\vee)\otimes (Y\otimes Y^\vee)\simeq (X\otimes Y)\otimes (Y^\vee \otimes X^\vee)$, we see that $c$ exhibits $Y^\vee \otimes X^\vee$ as a dual of $X\otimes Y$.
	\end{proof}
	
	\begin{de}
		Let $\mathcal{C}^{cot}\subset\mathcal{C}$ be the full subcategory consisting of cotensorable objects. We define the functor $$\underline{\mapp}_{\mathcal{C}}(-,-): (\mathcal{C}^{cot})^{\op{op}}\times\mathcal{C}\to \op{Fun}'(\mathcal{C}^{\op{op}},\mathcal{S})\simeq \mathcal{C}$$ given by $(X,Y)\mapsto \op{Map}_{\mathcal{C}}(-\otimes X,Y)$, where $\op{Fun}'(\mathcal{C}^{\op{op}},\mathcal{S})\simeq \mathcal{C}$ denotes the full subcategory of representable functors.
	\end{de}
	
	\begin{lem}\label{l3.7}
		Let $\mathcal{K}$ be a collection of simplicial sets. If $\mathcal{C}$ is $\mathcal{K}$-cocomplete and the monoidal structure on it is compatible with $K$-colimits for any $K\in\mathcal{K}$ (meaning $-\otimes-$ preserves $K$-colimits separately), then for any $K\in\mathcal{K}$, the full subcategory $\mathcal{C}^{cot}\subset\mathcal{C}$ is closed under $K$-colimits and for any diagram $X_{(-)}:K\to \mathcal{C}^{cot}$, the natural map $\varprojlim_{\alpha\in K} \underline{\mapp}_{\mathcal{C}}(X_\alpha,-)\simeq \underline{\mapp}_{\mathcal{C}}(\varinjlim_{\alpha\in K}X_\alpha,-)$ is an equivalence in $\op{Fun}(\mathcal{C},\mathcal{C})$.
	\end{lem}
	
	\begin{proof}
		Consider the following diagram:
		$$
		\begin{tikzcd}
			(\mathcal{C}^{cot})^{\op{op}} \arrow[r] \arrow[rd, "{X\mapsto\underline{\mapp}_{\mathcal{C}}(X,-)}"', dashed] & \mathcal{C}^{\op{op}} \arrow[r, "X\mapsto (-)\otimes X"] \arrow[rd, "\phi"] & {\op{Fun}(\mathcal{C},\mathcal{C})}^{\op{op}} \arrow[d, "i_L", hook] \\
			& {\op{Fun}(\mathcal{C},\mathcal{C})} \arrow[r, "i_R", hook] & {\op{Fun}(\mathcal{C}^{\op{op}}\times\mathcal{C},\mathcal{S})}
		\end{tikzcd}
		$$
		where $i_L$ is given by $F\mapsto \op{Map}_{\mathcal{C}}(F(-),-)$ and $i_R$ is given by $G\mapsto \op{Map}_{\mathcal{C}}(-,G(-))$. An object $X\in\mathcal{C}$ is cotensorable if and only if $\phi(X)$ lies in the image of $i_R$. Now given a diagram $K\in \mathcal{K}$, it suffices to observe that:
		\begin{enumerate}
			\item $(\mathcal{C}^{cot})^{\op{op}}=\phi^{-1}(\op{Im}(i_R))$.
			\item $\phi$ preserves $K^{\op{op}}$-limits and the inclusion $i_R$ is closed under $K^{\op{op}}$-limits.
		\end{enumerate}
	\end{proof}
	
	\begin{prop}\label{du}\,
		\enu{
			\item If $\mathcal{C}^\otimes$ is pointedly symmetric monoidal (meaning that $\mathcal{C}$ is pointed and the tensor product of the zero object with any object is zero), then the zero object $*$ is dualizable.
			\item If $\mathcal{C}$ is idempotent complete, then $\mathcal{C}^d\subset\mathcal{C}$ is closed under retractions.
			\item If $\mathcal{C}^\otimes$ is semiadditively symmetric monoidal, then $\mathcal{C}^d\subset\mathcal{C}$ is closed under finite coproducts and hence forms a full semiadditive subcategory.
			\item If $\mathcal{C}^\otimes$ is stably symmetric monoidal, then $\mathcal{C}^d\subset\mathcal{C}$ is closed under finite colimits and finite limits and hence forms a full stable subcategory.
		}
	\end{prop}
	
	\begin{proof}
		Applying  \cref{cri} and  \cref{l3.7} to $\mathcal{K}=\{\varnothing\}, \{\op{N}(\text{Idem})\}, \{\text{finite discrete diagrams}\}, \{\text{finite diagrams}\}$ respectively, we proved (1), (2), (3) and the ``closed under finite colimits'' part of (4). For the ``closed under finite limits'' part of (4), it suffices to show that $\mathcal{C}^d\subset\mathcal{C}$ is closed under desuspension. This follows from $\Sigma^{-1} X^\vee=(\Sigma X)^{\vee}$ for a dualizable object $X\in\mathcal{C}^d$.
	\end{proof}
	
	\subsection{Duality of Bimodules}\label{3.2}
	
	Extending beyond symmetric monoidal categories, we detail the higher duality theory specifically adapted for bimodules over $\mathbb{E}_1$-algebras. 
	
	\begin{cov}
		Throughout  \cref{3.2}, we fix a monoidal $\infty$-category $\mathcal{C}^\otimes \to \op{Ass}^\otimes$ which admits geometric realizations of simplicial objects and such that the tensor product $\otimes: \mathcal{C} \times \mathcal{C} \rightarrow \mathcal{C}$ preserves geometric realizations of simplicial objects.
	\end{cov}
	
	\begin{de}
		Let $X \in { }_A \operatorname{BMod}_B(\mathcal{C})$ and $Y \in { }_B \operatorname{BMod}_A(\mathcal{C})$. Let $c: B \rightarrow Y \otimes_A X$ be a map in ${ }_B \operatorname{BMod}_B(\mathcal{C})$. We say $c$ exhibits $X$ as the right dual of $Y$, or $c$ exhibits $Y$ as the left dual of $X$, if there exists a map $e: X \otimes_B Y \rightarrow A$ in ${ }_A \operatorname{BMod}_A(\mathcal{C})$ such that
		$$
		\begin{aligned}
			& X \simeq X \otimes_B B \xrightarrow{\mathrm{id} \otimes c} X \otimes_B Y \otimes_A X \xrightarrow{e \otimes \mathrm{id}} A \otimes_A X \simeq X \\
			& Y \simeq B \otimes_B Y \xrightarrow{c \otimes \mathrm{id}} Y \otimes_A X \otimes_B Y \xrightarrow{\mathrm{id} \otimes e} Y \otimes_A A \simeq Y
		\end{aligned}
		$$
		are homotopic to $\mathrm{id}_X$ and $\mathrm{id}_Y$, respectively.
	\end{de}
	
	\begin{prop}[See \cite{ha} 4.6.2.18]
		Let $A \in \operatorname{Alg}(\mathcal{C})$, let $X \in \operatorname{LMod}_A(\mathcal{C})$, let $Y \in \operatorname{RMod}_A(\mathcal{C})$, and let $c: \mathbf{1} \rightarrow Y \otimes_A X$ be a map in $\mathcal{C}$. Then the following are equivalent:
		\enu{
			\item The map $c: \mathbf{1} \rightarrow Y \otimes_A X$ exhibits $Y$ as a left dual of $X$.
			\item For each $C \in \mathcal{C}$ and each $M \in \operatorname{RMod}_A(\mathcal{C})$, the composite map
			$$
			\operatorname{Map}_{\operatorname{RMod}_A(\mathcal{C})}(C \otimes Y, M) \rightarrow \operatorname{Map}_{\mathcal{C}}\left(C \otimes Y \otimes_A X, M \otimes_A X\right) \xrightarrow{\circ c} \operatorname{Map}_{\mathcal{C}}\left(C, M \otimes_A X\right)
			$$
			is a homotopy equivalence.
			\item For each $C \in \mathcal{C}$ and each $N \in \operatorname{LMod}_A(\mathcal{C})$, the composite map
			$$
			\operatorname{Map}_{\operatorname{LMod}_A(\mathcal{C})}(X \otimes C, N) \rightarrow \operatorname{Map}_{\mathcal{C}}\left(Y \otimes_A X \otimes C, Y \otimes_A N\right) \xrightarrow{\circ c} \operatorname{Map}_{\mathcal{C}}\left(C, Y \otimes_A N\right)
			$$
			is a homotopy equivalence.
		}
	\end{prop}
	
	\begin{cor}\label{ldualeq}
		Let $A \in \operatorname{Alg}(\mathcal{C})$. Let $\operatorname{LMod}_A(\mathcal{C})^{ld}\subset \operatorname{LMod}_A(\mathcal{C})$ denote the full subcategory of left dualizable left $A$-modules, and let $\operatorname{RMod}_A(\mathcal{C})^{rd}\subset \operatorname{RMod}_A(\mathcal{C})$ denote the full subcategory of right dualizable right $A$-modules. Then the profunctor $\operatorname{RMod}_A(\mathcal{C})^{rd}\times\operatorname{LMod}_A(\mathcal{C})^{ld}\to\mathcal{S}$ given by $\operatorname{Map}_{\mathcal{C}}\left(\mathbf{1}, - \otimes_A -\right)$ is a balanced profunctor \textup{(see \cite[\href{https://kerodon.net/tag/03MM}{03MM}]{lurie2025kerodon})}, which induces a natural equivalence of $\infty$-categories $$^\vee(-):\operatorname{LMod}_A(\mathcal{C})^{ld}\underset{}{\stackrel{\sim}{\rightleftarrows}}(\operatorname{RMod}_A(\mathcal{C})^{rd})^{\op{op}}: (-)^\vee .$$
	\end{cor}
	
	\begin{proof}
		It suffices to observe that $c: \mathbf{1} \rightarrow Y \otimes_A X$ exhibits $Y$ as a left dual of $X$ if and only if it exhibits $X$ as a right dual of $Y$.
	\end{proof}
	
	\begin{cor}\label{rigidlygene}
		Suppose that $\mathcal{C}^\otimes$ is a cocompletely symmetric monoidal (potentially large) $\infty$-category, i.e., $\mathcal{C}$ admits small colimits and the tensor product $\otimes: \mathcal{C} \times \mathcal{C} \rightarrow \mathcal{C}$ preserves small colimits separately.
		Let $A \in \operatorname{Alg}(\mathcal{C})$, let $X \in \operatorname{LMod}_A(\mathcal{C})$, let $Y \in \operatorname{RMod}_A(\mathcal{C})$, and let $c: \mathbf{1} \rightarrow Y \otimes_A X$ be a map in $\mathcal{C}$. If $\mathcal{C}$ is generated by dualizable objects under small colimits, then the following are equivalent:
		\enu{
			\item The map $c: \mathbf{1} \rightarrow Y \otimes_A X$ in $\mathcal{C}$ exhibits $Y$ as a left dual of $X$.
			\item The functor
			$$
			\operatorname{Map}_{\mathcal{C}}\left(\mathbf{1}, Y \otimes_A -\right): \operatorname{LMod}_A(\mathcal{C})\to \widehat{\mathcal{S}}
			$$
			is corepresented by $X$ with the element $c: \mathbf{1} \rightarrow Y \otimes_A X$.
			\item The functor
			$$
			\operatorname{Map}_{\mathcal{C}}\left(\mathbf{1}, - \otimes_A X \right): \operatorname{RMod}_A(\mathcal{C})\to \widehat{\mathcal{S}}
			$$
			is corepresented by $Y$ with the element $c: \mathbf{1} \rightarrow Y \otimes_A X$.
		}
		(Note that we use the notation $\widehat{\mathcal{S}}$ above because $\mathcal{C}$ is not necessarily small here.)
	\end{cor}
	\section{Ind(Pro)-completion of large $\infty$-categories}
	The $\infty$-categorical theory frequently encounters set-theoretic size issues. In this appendix, we resolve the Ind- and Pro-completions of large $\infty$-categories relative to large regular cardinals. 
	\begin{cov}
		We work relative to a chain of strongly inaccessible cardinals \(\delta_{0}<\delta_{1}<\delta_{2}\). Then \(V_{\delta_i}\) is a Grothendieck universe, and elements of \(V_{\delta_0},V_{\delta_1},V_{\delta_2}\) are called small, large and very large, respectively. 
	\end{cov}
	\begin{thm}[See \cite{htt} 5.3.6.10]\label{httcocompletion}
		Let $\mathcal{K} \subset \mathcal{K}^{\prime}$ be $\delta_1$-small collections of simplicial sets. Let $\widehat{\operatorname{Cat}}_{\infty}^{\mathcal{K}}$ denote the subcategory spanned by those $\infty$-categories which admit $\mathcal{K}$-indexed colimits and those functors which preserve $\mathcal{K}$-indexed colimits, and let $\widehat{\operatorname{Cat}}_{\infty}^{\mathcal{K}^{\prime}}$ be defined likewise. Then the inclusion
		$$
		\widehat{\operatorname{Cat}}_{\infty}^{\mathcal{K}^{\prime}} \subset \widehat{\operatorname{Cat}}_{\infty}^{\mathcal{K}}
		$$
		admits a left adjoint given by $\mathcal{C} \mapsto \mathcal{P}_{\mathcal{K}}^{\mathcal{K}^{\prime}}(\mathcal{C})$.
	\end{thm}

	\begin{prop}[See \cite{hesselholt2024dirac} A.2]\label{diracA2}
		Let $\mathcal{C}$ be a coaccessible $\infty$-category (i.e. $\mathcal{C}^{\mathrm{op}}$ is accessible). For a functor $X: \mathcal{C}^{\mathrm{op}} \rightarrow \mathcal{S}$, the following conditions are equivalent:
		\enu{	\item The functor $X: \mathcal{C}^{\mathrm{op}} \rightarrow \mathcal{S}$ is accessible.
			\item The functor $X: \mathcal{C}^{\mathrm{op}} \rightarrow \mathcal{S}$ is the left Kan extension of a functor $Y:\left(\mathcal{C}_\kappa\right)^{\mathrm{op}} \rightarrow \mathcal{S}$ along the canonical inclusion $i:\left(\mathcal{C}_\kappa\right)^{\mathrm{op}} \rightarrow \mathcal{C}^{\mathrm{op}}$ for some small regular cardinal $\kappa$, where $\mc{C}_\kappa\subset\mc{C}$ denotes the full subcategory of $\kappa$-cocompact objects.
			\item The functor $X$ : $\mathcal{C}^{\mathrm{op}} \rightarrow \mathcal{S}$ is a colimit in $\operatorname{Fun}\left(\mathcal{C}^{\mathrm{op}}, \mathcal{S}\right)$ of a small diagram of representable functors.
			
		}
		
	\end{prop}
	\begin{cor}[See \cite{hesselholt2024dirac} A.4]
		Let $\mc{C}$ be a coaccessible $\infty$-category. Then the Yoneda embedding $$\mc{C}\to \funct^{\op{ac}}(\mc{C}^{\mathrm{op}},\mc{S})$$ exhibits $\funct^{\op{ac}}(\mc{C}^{\op{op}},\mc{S})\simeq \mc{P}^{\op{small}}_\emptyset(\mc{C})$, where $\funct^{\op{ac}}(\mc{C}^{\op{op}},\mc{S})$ denotes the full subcategory of accessible functors.
	\end{cor}
	\begin{prop}[See \cite{hesselholt2024dirac} A.9]\label{diracA9}
		Let $\mathcal{C}$ be a coaccessible $\infty$-category. The $\infty$-category $\mathcal{P}^{\op{ac}}(\mathcal{C})$ of accessible presheaves of anima on $\mathcal{C}$ admits all small limits and colimits, and both are calculated pointwise.
	\end{prop}
	\begin{prop}
		Let $\mc{C}$ be a copresentable $\infty$-category and $\kappa$ be a small regular cardinal. Then the Yoneda embedding $$\mc{C}\to \funct^{\op{ac}}_{\kappa\op{-lex}}(\mc{C}^{\mathrm{op}},\mc{S})$$ exhibits $\funct^{\op{ac}}_{\kappa\op{-lex}}(\mc{C}^{\op{op}},\mc{S})\simeq \mc{P}^{\op{small}\, \kappa\op{-fil}}_\emptyset(\mc{C})=\op{Ind}_\kappa(\mc{C})$.
	\end{prop}
	\begin{proof}
		First we observe that $\op{Ind}_\kappa(\mc{C})=\funct^{\op{ac}}_{\kappa\op{-lex}}(\mc{C}^{\op{op}},\mc{S})\subset \funct_{\kappa\op{-lex}}(\mc{C}^{\op{op}},\widehat{\mc{S}})=\widehat{\op{Ind}}_\kappa(\mc{C})$ is closed under small $\kappa$-filtered colimits. We claim that any object in $\funct^{\op{ac}}_{\kappa\op{-lex}}(\mc{C}^{\op{op}},\mc{S})$ can be written as the retraction of a small $\kappa$-filtered colimit of representable functors. Then the result immediately follows. 
		
		Now let $F\in \funct^{\op{ac}}_{\kappa\op{-lex}}(\mc{C}^{\op{op}},\mc{S})$. Since $F\in \widehat{\op{Ind}}_\kappa(\mc{C})$, it can be written as a large $\kappa$-filtered colimit of representable functors $F\simeq \varinjlim_{i\in I}h_{X_i}$. For each small $\kappa$-filtered full subcategory $I^{\prime} \subset I$, let $F_{I^{\prime}}$ denote the colimit $\varinjlim_{\alpha \in I^{\prime}} h_{X_\alpha}$. Then by \cite[\href{https://kerodon.net/tag/0620}{0620}]{lurie2025kerodon}, $F$ can be written as a large $\delta_0$-filtered colimit of the diagram $\left\{F_{I^{\prime}}\right\}$, where $I^{\prime}$ ranges over all small $\kappa$-filtered full subcategories of $I$. However, by \cref{diracA2}, $F$ is largely $\delta_0$-compact in $\widehat{\op{Ind}}_\kappa(\mc{C})$, so $F$ is a retraction of some $F_{I'}$.
	\end{proof}
	\begin{prop}
		Let $\mc{C}$ be a copresentable $\infty$-category and $\kappa$ be a small regular cardinal. Then $\op{Ind}_\kappa(\mc{C})\simeq \mc{P}^{\op{small}}_{\kappa\op{-small}}(\mc{C})$.
	\end{prop}
	\begin{proof}
		By the construction in \cite[Corollary 5.3.6.10]{htt}, it is equivalent to prove that $\op{Ind}_\kappa(\mc{C})\subset\widehat{\op{Ind}}_\kappa(\mc{C})$ is the smallest full subcategory which contains representables and closed under small colimits, i.e. to prove $\op{Ind}_\kappa(\mc{C})=\widehat{\op{Ind}}_\kappa(\mc{C})^{\delta_0}$. Since the representables generate $\widehat{\op{Ind}}_\kappa(\mc{C})$ under large colimits, it suffices to show that $\op{Ind}_\kappa(\mc{C})\subset\widehat{\op{Ind}}_\kappa(\mc{C})^{\delta_0}$ and $\op{Ind}_\kappa(\mc{C})$ is idempotent complete. Those are implied by \cref{diracA2} and \cref{diracA9}.
	\end{proof}
	\begin{de}
		Let $\hatcat^{\kappa\op{-lex}}$ denote the subcategory spanned by those $\infty$-categories which admit finite limits and those functors which preserve $\kappa$-small limits, where $\kappa<\delta_1$ is a large regular cardinal.
	\end{de}
	\begin{prop}
		Let $\kappa<\lambda<\delta_1 $ be two large regular cardinals. Then there exists an adjoint pair
		$$\hatcat^{\kappa\op{-lex}}\underset{}{\stackrel{\op{Pro}^\lambda_\kappa}{\rightleftarrows}}\hatcat^{\lambda\op{-lex}}$$
		by the dual version of \textup{\cite[Corollary 5.3.6.10]{htt}.}
	\end{prop}
	\begin{rem}
		We have the identification $\op{Pro}^\lambda_\kappa(\mc{C})\simeq \op{Ind}^{\lambda}_\kappa(\mc{C}^{\op{op}})^{\op{op}}$.
	\end{rem}
	\begin{prop}
		The above adjunction can be promoted to a symmetric monoidal adjunction
		$$\hatcat^{\kappa\op{-lex},\otimes}\underset{}{\stackrel{\op{Pro}^\lambda_\kappa}{\rightleftarrows}}\hatcat^{\lambda\op{-lex},\otimes}$$ by the dual version of \textup{\cite[Proposition 4.8.1.3]{ha}. }
	\end{prop}
	\begin{rem}[Dual of \cite{ha} 4.8.1.9]
		$\mathrm{CAlg}(\hatcat)$ can be identified with the very large $\infty$-category of large symmetric monoidal $\infty$-categories. \\Unwinding the definitions, we see that $\mathrm{CAlg}(\hatcat^{\kappa\op{-lex}})$ can be identified with the subcategory of $\mathrm{CAlg}(\hatcat)$ spanned by the symmetric monoidal $\infty$-categories which are compatible with $\kappa$-small limits (meaning the tensor product $-\otimes-$ preserves $\kappa$-small limits separately), and those symmetric monoidal functors which preserve $\kappa$-small limits.
		
	\end{rem}
	\begin{cor}[Dual version of \cite{ha} 4.8.1.10] \label{moncomp}
		Let $\kappa<\lambda<\delta_1 $ be two large regular cardinals. By the following adjunction, $$\calg(\hatcat^{\kappa\op{-lex}})\underset{}{\stackrel{\op{Pro}^\lambda_\kappa}{\rightleftarrows}}\calg(\hatcat^{\lambda\op{-lex}})$$
		we see that for any large symmetric monoidal $\infty$-category $\mathcal{C}^{\otimes}$ for which the monoidal structure on $\mathcal{C}$ is compatible with $\kappa$-small limits, there exists a $\lambda$-completely large symmetric monoidal \infcat $\mathcal{D}^{\otimes}$ and a symmetric monoidal functor $\mathcal{C}^{\otimes} \rightarrow \mathcal{D}^{\otimes}$ with the following properties:
		\enu{
			\item The symmetric monoidal structure on $\mathcal{D}^{\otimes}$ is compatible with $\lambda$-small limits.
			\item The underlying functor $f: \mathcal{C} \rightarrow \mathcal{D}$ preserves $\kappa$-small limits.
			\item $f$ induces an identification $\mathcal{D} \simeq \op{Pro}^\lambda_\kappa(\mathcal{C})$, and is therefore fully faithful.
			\item $\mathcal{D}^{\otimes}$ is universal among those satisfying (1)-(3).}
		
	\end{cor}

	\printbibliography[heading=bibintoc]

@article{1,
    AUTHOR = {Strickland, Neil P.},
     TITLE = {Formal schemes and formal groups},
   JOURNAL = {Contemp. Math.},
    VOLUME = {239},
      YEAR = {1999},
     PAGES = {263--352}
}

@misc{2,
    AUTHOR = {{The Stacks Project Authors}},
     TITLE = {Stacks Project},
       URL = {https://stacks.math.columbia.edu},
      YEAR = {2025},
 SHORTHAND = {Stacks}
}

@book{5,
    AUTHOR = {Hartshorne, Robin},
     TITLE = {Algebraic geometry},
    VOLUME = {52},
 PUBLISHER = {Springer},
      YEAR = {1977}
}

@book{12,
    AUTHOR = {Demazure, Michel and Gabriel, Pierre},
     TITLE = {Groupes alg{\'e}briques},
 PUBLISHER = {Masson \& Cie},
      YEAR = {1970}
}

@book{htt,
    AUTHOR = {Lurie, Jacob},
     TITLE = {Higher Topos Theory},
 PUBLISHER = {Princeton University Press},
      YEAR = {2009},
 SHORTHAND = {HTT}
}

@misc{ha,
    AUTHOR = {Lurie, Jacob},
     TITLE = {Higher Algebra},
      NOTE = {\url{https://www.math.ias.edu/~lurie/papers/HA.pdf}},
      YEAR = {2017},
 SHORTHAND = {HA}
}

@misc{stefanich2023classification,
    AUTHOR = {Stefanich, Germ{\'a}n},
     TITLE = {Classification of fully dualizable linear categories},
      NOTE = {\arxiv{2307.16337v2}},
      YEAR = {2025}
}

@misc{sag,
    AUTHOR = {Lurie, Jacob},
     TITLE = {Spectral Algebraic Geometry},
      NOTE = {\url{https://www.math.ias.edu/~lurie/papers/SAG-rootfile.pdf}},
      YEAR = {2018},
 SHORTHAND = {SAG}
}

@article{mathew2016galois,
    AUTHOR = {Mathew, Akhil},
     TITLE = {The Galois group of a stable homotopy theory},
   JOURNAL = {Adv. Math.},
    VOLUME = {291},
      YEAR = {2016},
     PAGES = {403--541}
}

@book{rudyak1998thom,
    AUTHOR = {Rudyak, Yuli B.},
     TITLE = {On Thom spectra, orientability, and cobordism},
 PUBLISHER = {Springer},
      YEAR = {1998}
}

@misc{ec1,
    AUTHOR = {Lurie, Jacob},
     TITLE = {Elliptic Cohomology I: Spectral Abelian Varieties},
      NOTE = {\url{https://www.math.ias.edu/~lurie/papers/Elliptic-I.pdf}},
      YEAR = {2018},
 SHORTHAND = {ECI}
}

@misc{dag4,
    AUTHOR = {Lurie, Jacob},
     TITLE = {Derived Algebraic Geometry IV: Deformation Theory},
      NOTE = {\url{https://www.math.ias.edu/~lurie/papers/DAG-IV.pdf}},
      YEAR = {2009},
 SHORTHAND = {DAGIV}
}

@misc{dag13,
    AUTHOR = {Lurie, Jacob},
     TITLE = {Derived Algebraic Geometry XIII: Rational and p-adic Homotopy Theory},
      NOTE = {\url{https://www.math.ias.edu/~lurie/papers/DAG-XIII.pdf}},
      YEAR = {2011},
 SHORTHAND = {DAGXIII}
}

@article{bachmann2022chow,
    AUTHOR = {Bachmann, Tom and Kong, Hana Jia and Wang, Guozhen and Xu, Zhouli},
     TITLE = {The Chow t-structure on the $\infty$-category of motivic spectra},
   JOURNAL = {Ann. of Math. (2)},
    VOLUME = {195},
    NUMBER = {2},
      YEAR = {2022},
     PAGES = {707--773}
}

@misc{lurie2025kerodon,
    AUTHOR = {Lurie, Jacob},
     TITLE = {Kerodon},
      NOTE = {Online monograph available at \url{https://kerodon.net}},
      YEAR = {2025},
 SHORTHAND = {Ker}
}

@article{hesselholt2023dirac,
    AUTHOR = {Hesselholt, Lars and Pstr{\k{a}}gowski, Piotr},
     TITLE = {Dirac geometry I: Commutative algebra},
   JOURNAL = {Peking Math. J.},
      YEAR = {2023},
     PAGES = {1--76}
}

@article{antieau2021cartier,
    AUTHOR = {Antieau, Benjamin and Nikolaus, Thomas},
     TITLE = {Cartier modules and cyclotomic spectra},
   JOURNAL = {J. Amer. Math. Soc.},
    VOLUME = {34},
    NUMBER = {1},
      YEAR = {2021},
     PAGES = {1--78}
}

@misc{ramzi2024locally,
    AUTHOR = {Ramzi, Maxime},
     TITLE = {Locally rigid $\infty$-categories},
      NOTE = {\arxiv{2410.21524v2}},
      YEAR = {2026}
}

@misc{ramzi2024dualizable,
    AUTHOR = {Ramzi, Maxime},
     TITLE = {Dualizable presentable $\infty$-categories},
      NOTE = {\arxiv{2410.21537}},
      YEAR = {2024}
}

@misc{burklund2020galois,
    AUTHOR = {Burklund, Robert and Hahn, Jeremy and Senger, Andrew},
     TITLE = {Galois reconstruction of Artin--Tate R-motivic spectra},
      NOTE = {\arxiv{2010.10325}},
      YEAR = {2020}
}

@misc{hebestreit2024note,
    AUTHOR = {Hebestreit, Fabian and Scholze, Peter},
     TITLE = {A note on higher almost ring theory},
      NOTE = {\arxiv{2409.01940}},
      YEAR = {2024}
}

@misc{ben2024perspective,
    AUTHOR = {Ben-Bassat, Oren and Kelly, Jack and Kremnizer, Kobi},
     TITLE = {A perspective on the foundations of derived analytic geometry},
      NOTE = {\arxiv{2405.07936}},
      YEAR = {2024}
}

@misc{raksit2020hochschild,
    AUTHOR = {Raksit, Arpon},
     TITLE = {Hochschild homology and the derived de Rham complex revisited},
      NOTE = {\arxiv{2007.02576}},
      YEAR = {2020}
}

@article{kelly2022analytic,
    AUTHOR = {Kelly, Jack and Kremnizer, Kobi and Mukherjee, Devarshi},
     TITLE = {An analytic Hochschild-Kostant-Rosenberg theorem},
   JOURNAL = {Adv. Math.},
    VOLUME = {410},
      YEAR = {2022},
     PAGES = {108694}
}

@article{hesselholt2024dirac,
    AUTHOR = {Hesselholt, Lars and Pstr{\k{a}}gowski, Piotr},
     TITLE = {Dirac geometry II: Coherent cohomology},
   JOURNAL = {Forum Math. Sigma},
    VOLUME = {12},
      YEAR = {2024},
     PAGES = {e27}
}

@article{toen2009dessous,
    AUTHOR = {To{\"e}n, Bertrand and Vaqui{\'e}, Michel},
     TITLE = {Au-dessous de Spec(Z)},
   JOURNAL = {J. K-Theory},
    VOLUME = {3},
    NUMBER = {3},
      YEAR = {2009},
     PAGES = {437--500}
}

@article{banerjee2012centre,
    AUTHOR = {Banerjee, Abhishek},
     TITLE = {Centre of monoids, centralisers, and localisation},
   JOURNAL = {Comm. Algebra},
    VOLUME = {40},
    NUMBER = {11},
      YEAR = {2012},
     PAGES = {3975--3993}
}

@article{banerjee2017noetherian,
    AUTHOR = {Banerjee, Abhishek},
     TITLE = {Noetherian schemes over abelian symmetric monoidal categories},
   JOURNAL = {Int. J. Math.},
    VOLUME = {28},
    NUMBER = {07},
      YEAR = {2017},
     PAGES = {1750051}
}

@article{mantovani2023localizations,
    AUTHOR = {Mantovani, Lorenzo},
     TITLE = {Localizations and completions of stable $\infty$-categories},
   JOURNAL = {Rend. Semin. Mat. Univ. Padova},
    VOLUME = {151},
      YEAR = {2024},
     PAGES = {1--62}
}

@article{hoyois2020cdh,
    AUTHOR = {Hoyois, Marc},
     TITLE = {Cdh descent in equivariant homotopy $K$-theory},
   JOURNAL = {Doc. Math.},
    VOLUME = {25},
      YEAR = {2020},
     PAGES = {457--482}
}

@book{cohn1985free,
    AUTHOR = {Cohn, Paul Moritz},
     TITLE = {Free rings and their relations},
 PUBLISHER = {Academic Press},
      YEAR = {1971}
}

@article{angeleri2020flat,
    AUTHOR = {Angeleri H{\"u}gel, Lidia and Marks, Frederik and {\v{S}}t'ov{\'i}{\v{c}}ek, Jan and Takahashi, Ryo and Vit{\'o}ria, Jorge},
     TITLE = {Flat ring epimorphisms and universal localizations of commutative rings},
   JOURNAL = {Q. J. Math.},
    VOLUME = {71},
    NUMBER = {4},
      YEAR = {2020},
     PAGES = {1489--1520}
}

@misc{arakawa2025monoidal,
    AUTHOR = {Arakawa, Kensuke},
     TITLE = {Monoidal Relative Categories Model Monoidal $\infty$-Categories},
      NOTE = {\arxiv{2504.14884}},
      YEAR = {2025}
}

@article{neeman2006non,
    AUTHOR = {Neeman, Amnon and Ranicki, Andrew and Schofield, Aidan},
     TITLE = {A non-commutative generalisation of Thomason's localisation theorem},
   JOURNAL = {London Math. Soc. Lecture Note Ser.},
    VOLUME = {330},
      YEAR = {2006},
     PAGES = {60}
}

@book{schofield1985representation,
    AUTHOR = {Schofield, Aidan Harry},
     TITLE = {Representation of rings over skew fields},
    VOLUME = {92},
 PUBLISHER = {Cambridge University Press},
      YEAR = {1985}
}

@incollection{lurie2008classification,
    AUTHOR = {Lurie, Jacob},
     TITLE = {On the classification of topological field theories},
 BOOKTITLE = {Current Developments in Mathematics, 2008},
 PUBLISHER = {International Press},
      YEAR = {2009},
     PAGES = {129--280}
}

@article{baez1995higher,
    AUTHOR = {Baez, John C. and Dolan, James},
     TITLE = {Higher-dimensional algebra and topological quantum field theory},
   JOURNAL = {J. Math. Phys.},
    VOLUME = {36},
    NUMBER = {11},
      YEAR = {1995},
     PAGES = {6073--6105}
}

@misc{kelly2025localisinginvariantsderivedbornological,
    AUTHOR = {Kelly, Jack and Mukherjee, Devarshi},
     TITLE = {Localising invariants in derived bornological geometry},
      NOTE = {\arxiv{2505.15750}},
      YEAR = {2025}
}

@misc{ramzi-BousfieldKan,
    AUTHOR = {Ramzi, Maxime},
     TITLE = {Deducing the Bousfield-Kan formula for homotopy (co)limits from first principles},
      NOTE = {Notes on author's webpage, \url{https://sites.google.com/view/maxime-ramzi-en/notes/bousfield-kan}}
}

@article{avramov1999locally,
    AUTHOR = {Avramov, Luchezar L.},
     TITLE = {Locally complete intersection homomorphisms and a conjecture of Quillen on the vanishing of cotangent homology},
   JOURNAL = {Ann. of Math. (2)},
    VOLUME = {150},
    NUMBER = {2},
      YEAR = {1999},
     PAGES = {455--487}
}

@article{bondarko2010weight,
    AUTHOR = {Bondarko, Mikhail V.},
     TITLE = {Weight structures vs. t-structures; weight filtrations, spectral sequences, and complexes (for motives and in general)},
   JOURNAL = {J. K-Theory},
    VOLUME = {6},
    NUMBER = {3},
      YEAR = {2010},
     PAGES = {387--504}
}

@article{bachmann2021norms,
    AUTHOR = {Bachmann, Tom and Hoyois, Marc},
     TITLE = {Norms in motivic homotopy theory},
   JOURNAL = {Ast{\'e}risque},
    VOLUME = {425},
      YEAR = {2021}
}

@article {MR4217953,
    AUTHOR = {Rasekh, Nima and Stonek, Bruno},
     TITLE = {The cotangent complex and {T}hom spectra},
   JOURNAL = {Abh. Math. Semin. Univ. Hambg.},
  FJOURNAL = {Abhandlungen aus dem Mathematischen Seminar der Universit\"at
              Hamburg},
    VOLUME = {90},
      YEAR = {2020},
    NUMBER = {2},
     PAGES = {229--252},
      ISSN = {0025-5858,1865-8784},
   MRCLASS = {55P43 (14F10)},
  MRNUMBER = {4217953},
MRREVIEWER = {Inbar\ Klang},
       DOI = {10.1007/s12188-020-00226-8},
       URL = {https://doi.org/10.1007/s12188-020-00226-8},
}

@misc{rossanigo2026quasiaffineschemessinglycompactly,
    AUTHOR = {Rossanigo, Giovanni},
     TITLE = {Quasi-affine schemes and singly compactly generated $t$-structures},
      NOTE = {\arxiv{2606.19417}},
      YEAR = {2026}
}
	\bigskip
	
	\textsc{Jiacheng Liang, Department of Mathematics, Johns Hopkins University, Baltimore, MD 21218, USA}
	
	\emph{Email address:} \href{mailto:jliang66@jhu.edu}{\texttt{jliang66@jhu.edu}}\bigskip

\end{document}